\documentclass{amsart}
\usepackage{graphicx}
\usepackage{graphics}
\usepackage{psfrag}
\usepackage[active]{srcltx}
\usepackage{amssymb}
\usepackage{amscd}

\newcommand\bt{\beta}

\newcommand\gm{\gamma}
\newcommand\g{\gamma}
\newcommand\kp{\kappa}
\newcommand\ka{\kappa}
\newcommand\lmbd{\lambda}
\newcommand\la{\lambda}
\newcommand\lph{\alpha}
\newcommand\ps{\psi}
\newcommand\rh{\rho}

\newcommand\vph{\varphi}
\newcommand\x{\xi}

\newcommand\wi{i}
\newcommand\wj{j}
\newcommand\wk{k}

\newcommand\wu{u}
\newcommand\wv{v}
\newcommand\ww{w}
\newcommand\wx{x}
\newcommand\wy{y}
\newcommand\wz{z}

\newcommand\uov{\wu;1;\wv}
\newcommand\woz{\ww;1;\wz}
\newcommand\xox{\wx;1;\wx}
\newcommand\xoy{\wx;1;\wy}
\newcommand\xoz{\wx;1;\wz}
\newcommand\yox{\wy;1;\wx}
\newcommand\yoy{\wy;1;\wy}
\newcommand\yoz{\wy;1;\wz}

\newcommand\ygy{\wy;g_1;\wy}

\newcommand\gxlab{G_\wx/\L_{a;b}}
\newcommand\gxla{G_\wx/\L_a}
\newcommand\gxlb{G_\wx/\L_b}

\newcommand\gyra{G_\wy/\R_a}
\newcommand\gyrb{G_\wy/\R_b}

\newcommand\gzrab{G_\wz/\R_{a;b}}

\newcommand\labla{\L_{a;b}/\L_a}

\newcommand\ralbra{(\R_a;\L_b)/\R_a}

\newcommand\cm[1]{Cm(#1\/)}

\newcommand\gcm[1]{{\mathfrak {Cm}}(#1\/)}

\newcommand\h[1]{H_{#1}}
\renewcommand\k[1]{K_{#1}}

\newcommand\m[1]{M_{#1}}

\newcommand\hs[2]{H_{#1,#2}}
\newcommand\ks[2]{K_{#1,#2}}

\newcommand\e[1]{e_{#1}}

\newcommand\vphi[1]{\vph_{#1}}

\newcommand\pair[2]{(#1,#2)}
\newcommand\kai[1]{\ka_{#1}}
\newcommand\G[1]{G_{#1}}

\newcommand\R{R}

\newcommand\mo{^{-1}}
\newcommand\seq{\subseteq}
\newcommand\ssm{^{\scriptscriptstyle\smile}}
\newcommand\scir{\raise2pt\hbox{$\,\scriptscriptstyle\circ\,$}}

\newcommand\suml{\sum\limits}

\newcommand\tsbigcreg{\textstyle \bigcup}
\newcommand\tbigcup{\textstyle \bigcup}

\newcommand\sumreg{\textstyle\sum}
\newcommand\tsum{\textstyle\sum}
\newcommand\prodreg{\textstyle\prod}
\newcommand\tprod{\textstyle\prod}

\newcommand\id[1]{id_{#1}}

\newcommand\co{\textnormal{,}\ }     
\newcommand\opar{\textnormal{(}}     
\newcommand\cpar{\textnormal{)}}     

\mathcode`\;="203B \mathchardef\sim="2218

\newcommand\smbcomma{\,,}

\newcommand\f[1]{{\mathfrak {#1}}}

\newcommand\inv{{}^{\scriptstyle -1}}

\newcommand\comp{{}^-}
\newcommand\conv{{}^{\scriptstyle\smallsmile}}

\newcommand\rp{\!\mid\!}

\newcommand\diff{\sim}

\newcommand\mult{@!@!@!\cdot@!@!@!}
\newcommand\rmult{\!@!@!@!\cdot@!@!@!}

\newcommand\ssum[1]{{\textstyle\sum {#1}}}

\newcommand\ra{relation algebra}

\newcommand\sub[1]{\bssS@,@,@,(#1)}

\newcommand\qeddef{\qed}

\newcommand\qedcon{\qed}

\newcommand{\num}[2]{\par\medskip\smallskip\noindent(#1)\hfill\parbox[t]{4.50in}{#2}}

\newcommand{\enum}{\par\medskip\smallskip\noindent}

\renewcommand\r[2]{R_{{#1},{#2}}}
\newcommand\craset[2]{#1\myshortspace[\mathcal{#2}\,]}
\newcommand\cra[2]{\mathfrak{#1}\myshortspace[\mathcal{#2}\,]}

\newcommand{\mc}[1]{\mathcal{#1}}

\newcommand{\refL}[1]{Lemma~\ref{L:#1}}
\newcommand{\refD}[1]{Definition~\ref{D:#1}}
\newcommand{\refC}[1]{Corollary~\ref{C:#1}}
\newcommand{\refT}[1]{Theorem~\ref{T:#1}}
\newcommand{\refS}[1]{Section~\ref{S:#1}}

\usepackage{amssymb}
\usepackage{amsmath}
\usepackage{amsthm}
\usepackage[mathscr]{eucal}
\usepackage{eufrak}
\usepackage{latexsym}

\newcommand{\myspace}{{\hspace*{.5pt}}}
\newcommand{\myshortspace}{{\hspace*{.1pt}}}
\newcommand{\comma}{\textnormal{,}\ }
\newcommand\rcomma{\textnormal{\myshortspace,}}
\newcommand\rrcomma{\,\textnormal{,}}
\newcommand\po{\textnormal{.}\ }     

\newcommand\ident{1\mynegspace\textnormal{\rq}}

\renewcommand\k[1]{K_{#1}}
\renewcommand\a{\alpha}
\renewcommand\b{\beta}

\def\mpicture #1 by #2 (#3 scaled #4){
      \vbox to #2{
      \hrule width #1 height 0pt depth 0pt
      \vfill
      \special{picture #3 scaled #4}}}

\def\centerpicture #1 by #2 (#3 scaled #4){
   \dimen0=#1 \dimen1=#2                                                
    \divide\dimen0 by 1000 \multiply\dimen0 by #4        
    \divide\dimen1 by 1000 \multiply\dimen1 by #4
         \noindent                                                                      
         \vbox{                                                                          
            \hspace*{\fill}                                                         
            \mpicture \dimen0 by \dimen1 (#3 scaled #4)       
            \hspace*{\fill}                                                         
            \vfill}}                                                                      

\usepackage{xspace}

\renewcommand\comp{-}
\newcommand\vp{\varphi}
\renewcommand\ident{1\mynegspace\textnormal{\rq}}

\renewcommand\L{L}

\newcommand\refEq[1]{(\ref{Eq:#1})}

\newcommand\refCo[1]{Convention~\ref{Co:#1}}

\newcommand\hvphs[1]{\hat{\vph}_{#1}}
\newcommand\vphs[1]{\vph_{#1}}
\newcommand\LS[1]{\L_{#1}}
\newcommand\RS[1]{\R_{#1}}
\newcommand\MS[1]{\M_{#1}}

\newcommand\axy{{\al_{\wx\wy}}}
\newcommand\axx{{\al_{\wx\wx}}}
\newcommand\axz{{\al_{\wx\wz}}}
\newcommand\ayz{{\al_{\wy\wz}}}
\newcommand\ayx{{\al_{\wy\wx}}}
\newcommand\ai{{\al_{\wi}}}
\newcommand\cl{{c}}

\newcommand\per{\textnormal{\myspace.\ }}
\newcommand\vth{\vartheta}
\renewcommand\comp{-}
\renewcommand\ident{1\mynegspace\textnormal{\rq}}

\renewcommand\L{L}
\renewcommand\R{R}

\renewcommand\P{P}
\newcommand\M{M}

\newcommand\bl{b}
\newcommand\al{a}

\newcommand\lax{\L_{a,\x}}
\newcommand\lbe{\L_{b,\eta}}
\newcommand\lal{\L_{a}}
\newcommand\lbl{\L_{b}}
\newcommand\rbx{\R_{b,\x}}
\newcommand\lbx{\L_{b,\x}}
\newcommand\rax{\R_{a,\x}}
\newcommand\rbe{\R_{b,\eta}}
\newcommand\rar{\R_{a}}
\newcommand\rbr{\R_{b}}
\newcommand\aab{\al;\al\ssm;b}
\newcommand\bba{\bl;\bl\ssm;\al}
\newcommand\abp{{\al;\bl}}
\newcommand\wh{h}
\newcommand\PS[1]{\P_{#1}}

\newcommand\azy{\al_{\wz\wy}}

\newcommand\aj{\al_{\wj}}
\newcommand\ak{\al_{\wk}}
\newcommand\ls[2]{\L_{#1,#2}}
\newcommand\rs[2]{\R_{#1,#2}}
\newcommand\as[2]{a_{#1,#2}}

\renewcommand\ident{1\textnormal{\rq}}

\newcommand\xx{{xx}}

\newcommand\xy{{xy}}
\newcommand\yx{{yx}}
\newcommand\zy{{zy}}
\newcommand\yz{{yz}}
\newcommand\xz{{xz}}

\newcommand\zx{{zx}}

\newcommand\uv{{uv}}

\newcommand\wwz{{wz}}
\newcommand\wwy{{wy}}

\newcommand\dl{d}
\newcommand\famxi{^{(\xi)}}

\newcommand\cc[3]{C_{#1#2#3}}

\newcommand\trip[3]{(#1,#2,#3)}

\newcommand\ez[1]{\mc E_{#1}}
\newcommand\hh{\h\xy\scir\h\xz}

\newcommand\sk[1]{\textsf{#1}}
\renewcommand\wi{\xy}
\renewcommand\wj{\yz}
\renewcommand\wk{\xz}

\newcommand\zoy{z;1;y}
\newcommand\zox{z;1;x}

\newcommand\et{\eta}

\newcommand\alp[1]{p_{#1}}

\newcommand\pairalp[2]{\pair {p_{#1}}{p_{#2}}}

\newcommand\cs[2]{{#1}_{#2}}

\newcommand\comment[1]{\,}
\newcommand\varnot{\varnothing}
\newcommand\relsum{\setbox0\hbox{$+$}\ensuremath
    \mathop{\rlap{\lower.31ex\hbox to\wd0{\hfil\kern-.02em \upshape,\hfil}}+}}

\newtheorem{theorem}{Theorem}[section]
\newtheorem{cor}[theorem]{Corollary}
\newtheorem{lm}[theorem]{Lemma}

\theoremstyle{definition}
\newtheorem{df}[theorem]{Definition}

\newtheorem{con}[theorem]{Convention}

\theoremstyle{remark}

\numberwithin{equation}{section}

\newcommand{\lmref}[1]{Lemma~\ref{#1}}

\newcommand{\conref}[1]{Convention~\ref{#1}}

\newif\ifShowLabels
\ShowLabelstrue
\newdimen\theight
\newcommand\TeXref[1]{%
     \leavevmode\vadjust{\setbox0=\hbox{{\tt
               \quad\quad #1}}%
                    \theight=\ht0
                         \advance\theight by \dp0
                              \advance\theight by \lineskip
                                   \kern -\theight \vbox to
                                             \theight{\rightline{\rlap{\box0}}%
                                                   \vss}%
                                                         }}%
\newcommand{\labelp}[1]{\label{#1}%
    \ifShowLabels \TeXref{{#1}} \fi}
    \ShowLabelsfalse

    \renewcommand\L{H}
\renewcommand\R{K}

\begin{document}

\title[A representation theorem for measurable relation algebras]{A representation theorem for measurable relation algebras}
\author{Steven Givant and Hajnal Andr\'eka}%
\address{Steven Givant\\Mills College\\5000 MacArthur Boulevard, Oakland, CA 94613}\email{givant@mills.edu}
\address{Hajnal Andr\'eka\\Alfr\'ed R\'enyi Institute of Mathematics\\Hungarian Academy of Sciences\\Re\'altanoda utca 13-15\\ Budapest\\ 1053 Hungary}\email{andreka.hajnal@renyi.mta.hu}
\thanks{This research was
partially supported  by Mills College and  the Hungarian National
Foundation for Scientific Research, Grants T30314 and T35192.}
\comment{\begin{abstract} Givant\,\cite{giv1}  generalized the
notion of a pair-dense relation algebra from Maddux\,\cite{ma91} by
defining the notion of a \textit{measurable relation algebra}, that
is to say, a relation algebra in which the identity element is a sum
of atoms that can be measured in the sense that the ``size" of each
such atom can be defined in an intuitive and reasonable way (within
the framework of the first-order theory of relation algebras). He
constructed a large class of examples of such algebras using systems
of groups and coordinated systems of isomorphisms between quotients
of the groups. Unfortunately, this class of \textit{group relation
algebras} is not large enough to prove a generalization of Maddux's
representation theorem for pair-dense relation algebras, namely that
every atomic, measurable relation algebra is isomorphic to a set
relation algebra. In fact, such a generalization proves to be false:
there exist finite, measurable relation algebras that are not
representable as set relation algebras at all.

 In Andr\'eka-Givant\,\cite{andgiv1},
the class of examples of measurable relation algebras was
considerably extended by adding one more ingredient to the mix:
systems of cosets that are used to ``shift" the operation of
relative multiplication.  It was shown there that, under certain
additional hypotheses on the system of cosets, each such
\textit{coset relation algebra} with a shifted operation of
relative multiplication is an example of a complete and  atomic, measurable relation
algebra.
The class of
coset relation algebras  contains  examples of measurable
relation algebras that are not representable as set relation
algebras.

We prove in this article that the class of coset relation
algebras is adequate to the task of  describing all atomic, measurable
relation algebras in the sense that every atomic, measurable
relation algebra is essentially isomorphic to a coset relation
algebra.
\end{abstract}}

\begin{abstract} A relation algebra is called measurable when its identity is the sum
of measurable atoms, where an atom is called measurable if its
square is the sum of functional elements.

In this paper we show that atomic measurable relation algebras have
rather strong structural properties: they are constructed from
systems of groups, coordinated systems of isomorphisms between
quotients of the groups, and systems of cosets that are used to
"shift" the operation of relative multiplication. An atomic and
complete measurable relation algebra is completely representable if
and only if there is a stronger coordination between these
isomorphisms induced by a scaffold (the shifting cosets are not
needed in this case). We also prove that a measurable relation
algebra in which the associated groups are all finite is atomic.
\end{abstract}

\maketitle

\section{Introduction}\labelp{S:sec1}

The well-known pair of papers \cite{jt51} and \cite{jt52}, by J\'onsson and Tarski, were motivated by Tarski's efforts to prove that every model of his axiomatization of the calculus of relation algebras is representable, that is to say, every (abstract) relation algebra is isomorphic
to a set relation algebra  consisting of a universe of (binary) relations on some base set, under the standard set-theoretic operations on such relations. In the second of these papers, the
 authors proved several representation theorems for limited classes of relation algebras.  In particular, they proved that an atomic relation algebra in which the atoms satisfy a specific ``singleton  inequality"
is isomorphic to a set relation algebra.  The singleton inequality is an inequality that is satisfied by a non-empty relation $R$ and its converse in a set relation algebra of all binary relations on a base set   if and only if $R$ is a singleton relation in the sense that it has the form $R=\{\pair pq\}$ for some elements $p$ and $q$ in the base set.

Maddux\,\cite{ma91} considerably strengthened this representation theorem.  He eliminated the assumptions that the given relation algebra be atomic and that every atom satisfy the singleton inequality.  Instead, he assumed only that the identity element be the sum of a set of non-zero \emph{elements} satisfying the singleton inequality.  Actually, he  proved an even stronger version of this theorem   by showing that every relation algebra in which the identity element is the sum of a set of non-zero elements satisfying the singleton inequality or a corresponding ``doubleton inequality" is isomorphic to a set relation algebra.  He called such relation algebras \emph{pair dense}.

The purpose of this paper is to prove a substantial generalization of Maddux's theorem.  The task  is complicated by the fact that for no natural number $n\ge 3$ is there  an equation or inequality that characterizes relations consisting of at most $n$ pairs.  This obstacle may be overcome by
allowing oneself to use formulas from first-order logic instead of just equations and inequalities.  In \cite{giv1}, an atom $x\le\ident$ is defined to be \emph{measurable} if the square $x;1;x$ is the sum of a set of \emph{functions}, that is to say, a set of abstract  elements $f$ satisfying the functional inequality $f\ssm;f\le \ident$. These
 functions turn out to be abstract versions of permutations, and the set of these permutations that are non-zero and below the square $x;1;x$  form a group.  The size of the group
 gives a measure of the size of $x$.  A relation algebra is said to be \emph{measurable} if the identity element is the sum of measurable atoms, and \emph{finitely measurable} if all of the atoms in this sum have finite measure.

 In \cite{giv1}, a large class of examples of measurable set relation algebras is constructed using systems  of groups  and corresponding systems
 of isomorphisms between quotients of the groups. The resulting algebras are
called (\emph{generalized}) \emph{group relation algebras}, and every such   algebra is an example of a complete and atomic measurable relation algebra.  The class of these examples, however,
does not comprehend all possible examples of complete and atomic  measurable relation algebras.  In \cite{andgiv1}, the class of examples is expanded by using systems of cosets to ``shift", or change the value,
of the operation of relational
composition in group relation algebras.  A characterization is given in \cite{andgiv1} of when such ``shifted" group relation algebras are relation algebra, and therefore   examples of
complete and atomic measurable relation algebras. They are called \emph{coset relation algebras} An example is given in \cite{andgiv1} of a coset relation algebra---and therefore of an atomic, measurable relation algebra---that is not isomorphic to any set relation algebra, so not all atomic measurable relation algebras are representable in the classical sense of the word.

The purpose of the present paper is to prove that the class of coset relation algebras is adequate for the task of ``representing in a wider sense" all atomic, measurable relation algebras.
In the main theorem of the paper, we show that every atomic, measurable relation algebra $\f B$ is\emph{ essentially isomorphic} to a coset
relation algebra $\f C$ in the sense that the completion (the minimal complete extension) of $\f B$ is isomorphic to $\f C$.  (The passage to the completion does not change the structure of $\f B$, it only fills in any missing infinite sums that are needed in order to obtain isomorphism with the complete relation algebra $\f C$).  In particular, every measurable relation algebra that is finite is isomorphic to a coset relation algebra.
 If the algebra $\f B$ is not finite, but is finitely measurable, then the assumption that $\f B$ be atomic may be dropped.
We also prove that a measurable relation algebra $\f B$ is essentially isomorphic to a group relation algebra if and only if $\f B$ has a ``scaffold" of atoms,
and this occurs if and only if $\f B$ is completely representable.

Except for basic facts about groups, this article is intended to be self-contained. The definition of a relation algebra, and the relatively few basic relation
algebraic laws that are needed to follow the proof in the paper are presented in \refS{sec2}.  Readers who wish to learn more about the subject  are recommended to look
at Hirsch-Hodkinson\,\cite{hh02}, Maddux\,\cite{ma06}, or Givant\,\cite{giv18}.

\section{Relation Algebras}\labelp{S:sec2}
 In the next few sections, most of the calculations will involve the arithmetic of relation algebras.  This section provides
 a  review the essential
 results that will be needed.

 A relation algebra is an algebra   of the form
\[\f A=( A\smbcomma +\smbcomma
-\smbcomma ;\smbcomma\,\conv\smbcomma\ident)\comma  \]  where $\,+\,$ and $\,;\,$ are binary operations called \textit{addition} and \textit{relative multiplication},
while $\,-\,$ and $\,\conv\,$ are unary operations called \textit{complement} and \textit{converse}, and $\ident$ is a distinguished constant called the
\textit{identity element}, such that the
following axioms are satisfied for all elements $r$, $s$, and $t$
in $\f A$\per   \begin{enumerate} \item [(R1)] $r+s=s+r$\per \item
[(R2)] $r+(s+t)=(r+s)+t$\per \item [(R3)] $-(-r+s) +
-(-r+-s)=s$\per
\item[(R4)]
$r;(s;t)=(r;s);t$\per
\item[(R5)]
$r;\ident=r$\per \item[(R6)] $r\conv\conv=r$\per
\item[(R7)]
$(r;s)\conv=s\conv;r\conv$\per \item[(R8)] $(r+s);t=r;t +s;t$\per
\item[(R9)]
$(r+s)\conv=r\conv +s\conv$\per \item[(R10)] $r\conv;-(r;s)+
-s=-s$.\labelp{Eq:10} \end{enumerate}  The axioms are commonly referred to by the following names: (R1)
is the \textit{commutative law for addition}, (R2) is the
\textit{associative law for addition}, (R3) is
\textit{Huntington's law}, (R4) is the \textit{associative law for
relative multiplication},  (R5) is the (right-hand)
\textit{identity law for relative multiplication}, (R6) is the
\textit{first involution law},  (R7) is the \textit{second
involution law}, (R8) is the (right-hand) \textit{distributive law
for relative multiplication}, (R9) is the \textit{distributive law
for converse}, and (R10) is \textit{Tarski's law}.  Under the
assumption of the remaining axioms, (R10) is equivalent to the
implication   \begin{equation*}\tag{R11}\labelp{Eq:11}
  \text{if}\quad(r;s)\cdot t=0,\quad
\text{then}\quad(r\conv;t)\cdot s=0, \end{equation*}
   which we shall call the \textit{cycle law}. It is this form of
(R10) that we  shall always use. Axioms (R1)--(R3) secure that $( A\smbcomma +\smbcomma
-)$ is a Boolean algebra. It is called the \emph{Boolean part} of $\f A$.   We shall justify a consequence of these three axioms with   the phrase \emph{by Boolean algebra}.  The Boolean operation of multiplication $\,\cdot\,$ is defined in the  usual way in terms of addition and complement.
An element $x$ in $\f A$ is called a \textit{subidentity element}
  if it
is below the identity element, in symbols $x\le \ident$.

Whenever parentheses indicating the order of
performing operations are  lacking, it is
understood that unary operations
have priority over
binary operations, and multiplications   have priority over addition.

\begin{lm}\labelp{L:laws} The operation of converse is an automorphism of the Boolean part
of a relation algebra\per In particular\comma the following laws hold\per
\begin{enumerate}
\item[(i)] $1\ssm = 1,\quad 0\ssm =
0,\quad
            \ident\,\ssm =\ident$\per
\item[(ii)] $(a \cdot b)\ssm = a\ssm \cdot
b\ssm$\po
\item[(iii)] $(\comp a)\ssm = \comp(a\ssm)$\po
\item[(iv)]$a \leq b$ if and only if $a\ssm \leq
b\ssm$\po
\item[(v)] $a=0$ if and only if $a\ssm=0$\po
\item[(vi)] $a$ is an atom if and only if $a\ssm$ is an atom\po
\item[(vii)]   $x\ssm=x$ whenever $x$ is a subidentity element\po
\end{enumerate}
\end{lm}

\begin{lm}\labelp{L:laws.1}
\begin{enumerate}
 \item[(i)] $a;0 = 0;a=0$\po
\item[(ii)] $1;1 = 1$\po
\item[(iii)] If $a \leq b$ and $c\leq d$\rcomma\ then
$a;c \leq b;d$\po
\item[(iv)] $(a;b)\cdot c = 0$\quad if and only if\quad
$(a\ssm;c)\cdot b = 0$,\quad if and only if\quad
$(c;b\ssm)\cdot a = 0$\po
\item[(v)] If $a$\comma $b$\comma and $c$ are atoms\comma then
      \[c\le a;b\quad\text{if and only if}\quad
b\le a\ssm;c,\quad\text{if and only if}\quad
a\le c;b\ssm\per\]
\item[(vi)] $a \leq a;1$\po
\item[(vii)] $(a;1) \cdot (1;b) = a;1;b$\po
\item[(viii)] If $x$ and $y$ are subidentity atoms\co then \[1;x;1=1;y;1\qquad\text{if and only if}\qquad x;1;y\neq 0\per\]
\end{enumerate} \end{lm}

 The laws in \refL{laws}(iv) and \refL{laws.1}(iii) are known as the
 \emph{monotony laws} for converse and relative multiplication respectively.  In referring to one of these laws
to justify a step in some proof, we shall usually just say \emph{by monotony}. The equivalences in \refL{laws.1}(iv),(v) are
usually called the \emph{cycle laws} and the \emph{cycle laws for atoms}, respectively---as opposed to the cycle law, which is (R11) and which is just one of the implications in (iv).  Again,
in using these equivalences to justify some step in a proof, we shall usually just say \emph{by the cycle laws.}

 The operations of relative multiplication and
converse are
\textit{completely distributive} over  addition
in the  sense that   for any two sets  of elements $X$ and $Y$, if the sums $\tsum X$ and
$\tsum Y$ exist,   then the  sums of the sets
\begin{gather*}X;Y=\{a;b:a\in X\text{ and } b\in Y\},\qquad Y\ssm=\{b\ssm:b\in Y\}\\
\intertext{exist,
and}(\tsum X);(\tsum Y)=\tsum X;Y,\qquad \tsum Y\ssm=(\tsum Y)\ssm\per\end{gather*}
   In referring to one of these laws to justify a step
in some proof, we shall usually just say  \emph{by complete distributivity}.

The \textit{domain} and \emph{range} of an element $a$   are
defined to be the elements $(a;1)\cdot \ident$ and $(1;a)\cdot \ident$ respectively.  Notice that they are subidentity elements.
Every law about domains has a corresponding dual law about ranges.  Therefore, only the law
concerning domains will usually be given.

\begin{lm} \labelp{L:domain}
Let  $\wx$\comma $\wy $\comma and $\wz$ be subidentity atoms\per Every
non-zero element
$a\le \xoy$ has $x$ as its domain  and $y$ as its range\comma and consequently\ the following laws hold\per
\begin{enumerate}
\item[\opar i\cpar] $\wx= (a;a\ssm)\cdot \ident=(a;1)\cdot\ident$
and $y=(a\ssm;a)\cdot \ident=(1;a)\cdot \ident$\per
\item[\opar ii\cpar] $\wx;1= a;1$ and $1;y=1;a$\per
\item[\opar iii\cpar] $\wx;a=a$ and $a;\wy=a$\po
\item[(iv)] lf $a\leq \xoy$ and $ b\leq
\yoz$\per
 then $a;b=0$ if and only if $a=0$ or  $b=0$\per
In particular\co if $a\neq 0$\co then $a;a\ssm\neq 0$ and
$a\ssm;a\neq 0$\po
\end{enumerate}
\end{lm}

 An element $f$  is called a \textit{function}, or a \textit{functional element},  if
$f\ssm;f\le\ident$.   If the converse of a function  $f$ is also a function,
then $f$ is said to be \textit{bijective}.
The element $f$ is a \textit{permutation}, or a  \textit{permutational element}, with domain $x$ if it is bijective and
if its domain and range are $x$.

\begin{lm}\labelp{L:funclaws}
Let $f,g$ be functions\comma  and
$a,b$  arbitrary elements\po
\begin{enumerate}
\item[(i)] $f;(a\cdot  b)=(f;a)\cdot  (f;b)$ and $(a\cdot  b);f\ssm = (a;f\ssm)\cdot
(b;f\ssm)$\po
\item[(ii)] $f;g$ is a function\po
\item[(iii)] If $a\le f$\co then $a$ is also a function\po
\item[(iv)] If $f$ and $g$ are bijective\co then so are $f\ssm$ and $f;g$\po
\item[(v)] If $f$ and $g$ are  permutations with domain $x$\co then so are
$f\ssm$ and
$f;g$\po\  Consequently\co the permutations with domain $x$ form a
group under the operations of relative multiplication and converse\comma
with $x$  as the identity element of
the group\po
\item[(vi)] A function  is an atom if and only if its domain is an
atom\po
\end{enumerate}
\end{lm}
Part (i) of the preceding lemma  says that if the left-hand argument of a relative product is a function, or the right-hand argument is the converse of a function, then  the operation of
 relative multiplication distributes over multiplication.  We shall refer to this law as the \emph{distributive law for functions}.
 It plays an extremely important role in this work.

A \textit{square}   is an element of the form $x;1;x$
for some subidentity element $x$, and a \textit{rectangle} is an element of the form
$x;1;y$ for some subidentity elements $x$ and $y$. The elements $x$ and $y$ are sometimes referred
to as the  \textit{sides} of the rectangle

\begin{lm}\labelp{L:square} Let $x,y,z,w$ be  subidentity
elements\po\
\begin{itemize}
\item [(i)] $(x;1;y)\cdot \ident = x\cdot y$\po
\item[(ii)] $(x;1;y)\cdot  (w;1;z) =(x\cdot  w);1;(y\cdot  z)$\po
\item[(iii)] $(x;1;y)\ssm = y;1;x$\po
\item[(iv)] $(x;1;y);(y;1;z)\leq x;1;z$\rcomma\  and equality holds whenever
$x$\comma $y$\comma and $z$ are atoms such that $x;1;y$ and $y;1;z$ are both non-zero\per
\end{itemize}
\end{lm}

\section{Measurable atoms}\labelp{S:sec3}

Throughout this and the next few sections, we assume that all elements belong to an arbitrary, but fixed, relation algebra $\f A$ with universe $A$. In order not to have to worry about the existence of certain infinite sums, we assume that $\f A$ is complete.  This assumption   in no way restricts the applicability of the main results of the paper.

\begin{df}\labelp{D:numeric}  A subidentity atom   $x$ is \textit{measurable} if the square $x;1;x$ with side $x$ is a sum of functions.  If
this square is
actually the sum of finitely many functions\co then $x$ is said to be
\textit{finitely measurable}\po
\end{df}

It turns out that the set of non-zero functions   below the square
$x;1;x$ of a measurable atom $x$ coincides with the set of atoms below the square,
and the cardinality of this set is a measure of the ``size" of  $x$.

\begin{lm}\labelp{L:funcatom1}  If $x$ is a measurable atom\comma then an element below the square $x;1;x$ is an atom if and only if it is
non-zero  function\po
\end{lm}

\begin{proof} 
Let $x$ be a measurable atom, and   $F$   the set of   functions
below $\xox$. The definition of measurability implies that \begin{equation*}\tag{1}\labelp{Eq:funcatom1}\xox = \tsum F\per
\end{equation*} If $f$ is an atom below
  $\xox$, then
\[0< f = f\cdot (\xox) = f\cdot (\tsum F) = \tsum\{f\cdot g: g\in F\}\comma
\] by Boolean algebra and \refEq{funcatom1}, so
there must be a function $g$ in $F$ such that $f\cdot g$ is non-zero.  But then
$f\le g$, because $f$ is an atom.  Any element below a function is itself
a function, by
\refL{funclaws}(iii), so $f$ is a non-zero function.

On the other hand, if $f$
is a non-zero function  below $\xox$, then the domain of $f$ is $x$, by \refL{domain}, and $x$ is an atom, by assumption, so $f$ is an atom, by \refL{funclaws}(vi).
\end{proof}

The non-zero functions below  the square on a measurable atom actually form
a group umder the operatons of relative multiplication and converse.

\begin{lm} \labelp{L:nonzerofunct}
If $\wx$ is a measurable atom\comma  then the set of  non-zero functions below the square
$\xox$ coincides with the set of permutations with domain  $x$\po
This set forms  a group under the operations $\,;\,$ and $\ssm$\co
with identity element $\wx$\po
\end{lm}

\begin{proof} 
If  $f$ is a non-zero function below $\xox$, then    $f$
is an atom, by \refL{funcatom1}. The converse of an atom is an atom, by \refL{laws}(vi), so $f\ssm$ is also an atom.  Apply \refL{funcatom1}
again to
conclude that $f\ssm$  is a function, and therefore $f$ is a bijection. The element $x$ is assumed to be an atom,
so every non-zero element below the square $x;1;x$ has domain and range $x$, by \refL{domain}.  In particular, the domain and range
of $f$ are both $x$, so $f$ is a permutation with domain $x$.

 If
$g$ is an arbitrary permutation  with domain $\wx$, then $g$ is an atom, and hence non-zero, by \refL{funclaws}(vi).   Furthermore,
\begin{equation*}
g \le (g;1)\cdot (1;g)
= (\wx;1)\cdot (1;\wx)
= \xox\comma
\end{equation*}
by Lemmas \ref{L:laws.1}(vi), \ref{L:domain}(ii), and \ref{L:laws.1}(vii), so
 $g$ is a non-zero function below $\xox$. Conclusion: the set of non-zero functions below $\xox$ coincides
 with the set of permutations with domain $x$.  This last set is a group under the operations of relative multiplication
 and converse, with $x$ as the identity element, by \refL{funclaws}(v), so the same must be true of the set of non-zero functions below $\xox$.
\end{proof}

The preceding lemma justifies the following definition.

\begin{df}\labelp{D:gxdef} The group of non-zero functions below the square on a measurable atom $x$ is denoted by  $\G x$\po  The cardinality
of this group is called the \textit{measure} of $x$.
\end{df}

The following corollary is an immediate consequence of Definitions \ref{D:numeric} and \ref{D:gxdef}
\begin{cor}\labelp{C:sumofGx} If $x$ is a measurable atom\comma then $\xox=\tsum\cs Gx$\per
\end{cor}
\begin{lm}\labelp{L:disjointgrps} If $\wx$ and $\wy$ are distinct measurable atoms\comma then the groups $\G\wx$ and $\G\wy$ are disjoint\po
\end{lm}
\begin{proof} The squares $\xox$ and $\yoy$ are disjoint, because
\[(\xox)\cdot(\yoy)=(x\cdot y);1;(x\cdot y) = 0;1;0 = 0\comma\]
by Lemmas \ref{L:square}(ii) and \ref{L:laws.1}(i).  The groups $\G\wx$ and $\G\wy$ consist
of non-zero elements below these respective squares, so they can have no elements in
common.
 \end{proof}

\comment{
\textbf{Perhaps delete this remark and the next lemma.  }The next lemma says that the relativization of the \ra\ $\f A$ to the
square of a measurable atom is faithfully embeddable into the complex algebra of a group.
(Recall that a faithful embedding is one that preserves arbitrary sums and that maps
the set of atoms of the domain algebra onto the set of atoms of the target algebra; see
\refD{faith}.)

\begin{lm}\labelp{L:grprel}  Let $x$ be a measurable atom in a \ra\
$\f A$\po Then mapping
from the relativization  $\f A(\xox)$
into the group relation algebra $\gcm{G_\wx}$ defined by
\[
a\longmapsto \{ f:f\text{ is a functional and } 0< f\le a\}
\]
for each element $a$  is a faithful embedding\po The mapping is onto
just in case the sum of each set of atoms below $\xox$ exists\po
\end{lm}

\begin{proof}
It follows from Lemmas \ref{L:funcatom1} and \ref{L:nonzerofunct} that
the relativization  $\f A(\xox)$ of $\f A$ to $\xox$ is atomic with functional
atoms.  Thus, if
$a$ is an element below
$\xox$ and
$F_\al$ is the set of non-zero functions below $a$, then
\begin{equation*}a=\tsum F_\al\per\tag{1}\end{equation*}

Consider now the mapping $\vth$ from $ A(\xox)$
to $\cm{G_\wx}$ defined by
\[
\vth(a)= F_\al =\{f\in G_\wx: f \leq a\}\per
\]
Since the relativization is atomic, $\vth$ must be one-one.  If  $f$ is an atom of the
relativization, then
\[\vth(f)= F_f =\{f\}\per\]
In other words, $\vth$ maps the set of atoms of the relativization onto the set of
atoms of  $\gcm{G_\wx}$.  Since $x$ is an atom of the relativization, the argument also
shows that $\vth$ maps the identity element of the relativization onto the identity element
of $\gcm{G_\wx}$.

Suppose that
$a$ and
$b$ are below
$\xox$. It follows from atomicity   that
\[
\vth(a+b)= F_{a+b} = F_\al \cup F_\bl=\vth(a)\cup \vth(b)
\]
and
\[
\vth(-a\cdot (\xox))= F_{-a\cdot(\xox)}= G_\wx \sim F_\al=\vth(\xox)\sim\vth(a)\per
\]
Also, since the elements of $\G x$  arepermutations,   $f$ will be a
function below
$a\ssm$ just in case
$f\ssm$ is a function below $a$. In other words,
\[
\vth(a\ssm)= F_{a\ssm}= F_\al\ssm= \vth(a)\ssm\per
\]
Finally,
\begin{equation*}
a;b = (\ssum F_\al); (\ssum F_\bl)
 = \ssum\{ f;g: f\in F_\al\text{ and } g\in F_\bl\}
= \ssum F_\al;F_\bl\comma
\tag{2}\end{equation*}
by the complete distributivity of relative multiplication.  Since $F_\al$ and
$F_\bl$ are subsets of
$\G x$\comma and since $\G x$ is closed under relative multiplication, the complex
product
$F_\al;F_\bl$ is itself a subset of
$G_\wx$\per  Now
$a;b=\ssum F_{a;b}
$\comma by (1).  Combine this with (2) --- and use the fact that an element uniquely determines,
and is uniquely determined by, the set of atoms below it --- to conclude that
$F_{a;b}= F_\al;F_\bl$\per Hence,
\[
\vth(a;b)= F_{a;b}= F_\al;F_\bl=\vth(a);\vth(b)\per
\]
This shows that $\vth$ is a faithful embedding of $\f A(\xox)$ into
$\gcm{G_\wx}$. The function $\vth$ maps each element to the set of atoms
beneath it.  Furthermore, the universe of $\gcm{G_\wx}$ consists of
all subsets of $G_\wx$\per  Therefore, $\vth$ will be surjective just in case $\ssum F$ exists
in $\f A$ for each $F\seq G_\wx$\per This completes the proof of the
lemma.
\end{proof}

}

Fix two measurable atoms $x$ and $y$, and
let   $f$ be an element in $\G x$\comma that is to say, let $f$ be a non-zero function below $\xox$.  Define a mapping
$\vth_f$  on the set
\[A(\xoy)=\{a\in A: a\le \xoy\}\] by stipulating that
\begin{equation*} \vth_f(a) = f;a
\end{equation*} for every $a$ in $A(\xoy)$.   The relative
product
$f;a$ is called the
\textit{left translation of}  $a$ \emph{by} $f$, so  $\vth_f$ maps every element in
$A(\xoy)$ to its left translation by $f$.

\begin{lm} \labelp{L:mappings} Each mapping $\vth_f$ is a
permutation of the set $A(\xoy)$\co and the correspondence $(f , a)\longmapsto \vth_f(a)$ defines a
left action of the group \,$G_\wx$ on the set  $A(\xoy)$ in the sense that
\[\vth_x(a)=a\qquad\text{and}\qquad \vth_g(\vth_f(a))=\vth_{g;f}(a) \] for all $a$ in $A(\xoy)$\per
\end{lm}
\begin{proof} 
Consider elements   $f$ and $g$   in $G_\wx$, and
$a$     in $A(\xoy)$. We have
\begin{equation*}
  \vth_f (a) =f;a\le (\xox);(\xoy)\le \xoy\comma
\end{equation*}
 by the definition  of $\vth_f$, the assumptions on $f$ and $a$,    monotony, and \refL{square}(iv).
Consequently,  $\vth_f(a)$ belongs to the set $A(\xoy)$.
Also,
\begin{equation*}
\vth_g( \vth_f(a))=g;(f;a) =(g;f);a=\vth_{g;f}(a)\comma \tag{1}
\label{Eq:map1}
\end{equation*}
by the definitions of $\vth_f$,  $\vth_g$, and $\vth_{g;f}$, and  the associative law; and
\begin{equation*}
\vth_\wx(a) = x;a = a\comma\tag{2}\labelp{Eq:map2}
\end{equation*}
by the definition of $\vth_\wx$ and \refL{domain}(iii). Thus, the correspondence  \[(f , a)\longmapsto \vth_f(a)\]does define a
left action of the group \,$G_\wx$ on the set  $A(\xoy)$.  It follows that
\begin{equation*}
      a = \vth_\wx (a)   = \vth_{f\ssm;f}(a)  = \vth_{f\ssm}(\vth_f (a))\comma
         \end{equation*}
 by \refEq{map2},   \refL{nonzerofunct}, and  \refEq{map1}, and dually,
 \[a =\vth_{f}(\vth_{f\ssm} (a))\comma \] so that  the mappings $\vth_{f\ssm}$
and
$\vth_f$ are inverses of one another. In particular, they must be one-to-one and onto, and hence
permutations of the set  $A(\xoy)$\per
\end{proof}

The set $A(\xoy)$ is closed under the binary operations of addition $\,+\,$ and multiplication $\,\cdot\,$ in $\f A$,
and also under the unary relativized complement operation $\comp_{\xoy}$ that is defined by
\[\comp_{\xoy} a=(\xoy)\cdot (-a) \] for all $a$ in $A(\xoy)$,  where $-a$ is the complement of $a$ in $\f A$.  Under these operations, the set $A(\xoy)$ becomes a Boolean algebra, and actually a relativization of the Boolean part of $\f A$.  Notice that an element belonging to this relativization  is an atom in $\f A$ just in case it is
an atom in the relativization.

\begin{lm}\labelp{L:auto} The mapping $\vth_f$ is an automorphism
of
the relativized Boolean algebra
\[
( A(\xoy)\smbcomma +\smbcomma\cdot \smbcomma \comp_{\xoy} \,)\per
\]
In particular\co an element $a\le \xoy$ is an atom if and only if $f;a$ is an
atom\po
\end{lm}

\begin{proof}  The mapping $\vth_f$ is a permutation of the set $A(\xoy)$, by \refL{mappings}. The distributive law (R8) implies that $\vth_f$
preserves the operation of addition,
\[\vth_f(a+b) = f;(a+b) = f;a +f;b=\vth_f(a) + \vth_f(b)\per
\]
 The distributive law for    functions, \refL{funclaws}(i), and the assumption
 that $f$ is a function, imply
  that $\vth_f$  preserves  multiplication,
\[\vth_f(a\cdot b) = f;(a\cdot b) = (f;a) \cdot (f;b)=\vth_f(a) \cdot  \vth_f(b)\per
\]
The element  $0$ is mapped to itself,
\[\vth_f(0) = f;0 = 0\comma
\] by  \refL{laws.1}(i).
Finally, $\vth_f$ maps the unit $\xoy$ of the relativization to itself,
\begin{equation*}
  \vth_f (\xoy) = f;\xoy= f;1;\wy= \xoy\comma
\end{equation*}
by the definition of $\vth_f$, and Lemmas \ref{L:nonzerofunct} and \ref{L:domain}(ii).
The operation of complement in the relativization  can be defined in terms of
  addition and multiplication, with the help of the elements $0$ and $\xoy$, so
$\vth_f$ must also preserve the operation of complement in the relativization. Conclusion: $\vth_f$ is an automorphism of the relativization.
Automorphisms obviously map atoms to atoms, so the second assertion of the lemma follows
at once from the first one, together with the remarks preceding the lemma.
\end{proof}

One of the main  points of \refL{auto} is that left translation by an element $f$ in the group $\G x$ maps the set of atoms of $\f A$ that are below $\xoy$ bijectively to itself.
\begin{df}\labelp{D:stabnot}
For each element $a\le \xoy$, the \emph{left stabilizer} of $a$ \emph{in}
$\G x$  under the group action of left translation
is defined to be  the set
\[\{f\in\G\wx:f;a=a\}\per\] It will be denoted by
$\L_a$\po\end{df}

The next
corollary is an immediate consequence of
\refL{mappings} and   well-known basic facts about group actions.

\begin{cor}\labelp{C:stab} For each $a \le \xoy$\comma the left stabilizer
$\L_a$ is a subgroup of $\G x$\po\  For any two elements  $f$ and $g$ in $
G_\wx$\comma  we have
$f;a=g;a$ if and only if
$f$ and $g$ are in the same left coset of $\L_a$\per
\end{cor}

The preceding corollary implies that all elements in a coset   $\L_\xi$   of $\L_a$ give rise to
the same left translation of $a$. Write $\L_\xi;a$ to denote this left translation.  This notation  helps to avoid the
 cumbersome task of specifying in advance  a representative $f$ of the coset $\L_\xi$, and writing $f;a$.
In a similar vein, for any subset $X$ of $\cs Gx$,   write \[X;a=\{f;a:f\in X\}\comma\qquad\text{so that}\qquad \tsum X;a=\tsum\{f;a:f\in X\}\per\]

In complete analogy with the definition of $\vth_f$\comma for each element
$g$ in $\G y$ one can define  a mapping $\psi_g$ that sends every element $a$ below $\xoy$ to its
\emph{right translation} by $g$,
\begin{equation*} \psi_g(a) = a;g\per
\end{equation*}
The  mapping $\psi_g$ is a permutation of the set $A(\xoy)$, and the correspondence $(g, a)\longmapsto \psi_g(a)$ defines
a
right action of the group $G_\wy$ on   the set $A(\xoy)$. The mapping $\psi_g$ is an automorphism of the relativized Boolean algebra corresponding to $A(\xoy)$.  In particular, an element $a\le \xoy$ is an atom if and only if $a;g$ is an atom.
The \emph{right stabilizer} of an element
$a\le \xoy$ under the group action of right translation is defined to be the set
\[\{g\in\G\wy:a;g=a\}\comma
\]
and is denoted by $\R_a$\per  The right stabilizer proves to be a subgroup of $\G\wy$\comma and for any two elements $f$ and
$g$ in $\G\wy$\comma the right translations $a;f$ and $a;g$ are equal just in case $f$ and
$g$ are in the same right coset of $\R_a$\per  Consequently, if $\R_\eta$ is a right  coset of $\R_a$, then it makes sense to write $a;\R_\eta$ to denote
the uniquely determined element that is the right translation of $a$ by   elements in $\R_\eta$.

Every result about left translations and left stabilizers has a corresponding dual result about right translations and right stabilizers.  In general, we  will usually formulate only the left-hand version, while allowing ourselves to refer to the right-hand versions in later proofs that require it.  The following easy lemma gives an example.
\begin{lm} \labelp{L:include} Let $\wx$\comma $\wy$\comma and $\wz$ be
measurable atoms\po If $\al\le\xoy$ and
$\bl\le\yoz$\co then $\lal\seq\LS\abp$\po
\end{lm}

\begin{proof}
If $f$ is  in  $\lal$\comma then $f;\al=\al$\comma by the
definition of $\lal$, and therefore   \[f;\abp=\abp\per\]  It follows that
$f$ is in $\L_{a;b}$.
\end{proof}

\medskip

\section{Left-regular and right-regular elements}\labelp{S:sec4}

A special class of elements called regular  elements   plays an
important role in the subsequent discussion. As will be seen, the
prototypical regular element  is an atom.  More generally, if $a$
is an atom below $\xoy$, and if $M$ is a subgroup of $\cs G\wx$
that extends the left stabilizer $\cs Ha$ of $a$, then the element
\[b=\tsum  M;a=\tsum\{f;a:f\in M\}\] is a regular element.

Many of the properties of regular elements hold for broader
classes of elements called left-regular elements and right-regular
elements respectively.  We begin with a study of these elements.

\begin{df} \labelp{D:regular}
Let $\wx$ and $\wy$ be measurable atoms\per  An
 element $a\le \xoy$  is called \textit{left-regular} or \textit{right-regular} respectively\comma according to whether
\[a;a\ssm=\suml \L_\al\qquad\text{or}\qquad a\ssm;a=\suml \R_\al\comma\]  and $a$ is called
\textit{regular} if it is both left and
right-regular\po
\qeddef
\end{df}

The next lemma and the   remarks following it are intended to clarify this
definition.

\begin{lm} \labelp{L:unique.set}
Let   $\wx$ and $\wy$ be measurable atoms\per For any   elements
$a$ and $b$ below $\xoy$\comma
   there are  uniquely determined  sets $E\seq
G_\wx$ and $F\seq G_\wy$ such that
\[
a;b\ssm=\ssum E
\qquad\text{and}\qquad
a\ssm;b=\ssum F\per
\]
If $a=b\neq 0$\co then $E$ and $F$ contain $\wx$ and $\wy$
respectively and  are closed under the operation of converse\po
\end{lm}
\begin{proof} 
Assume
\begin{equation*}
  0 \le a,b\le \xoy\per\tag{1}\labelp{Eq:unique.1}
\end{equation*}
Use \refEq{unique.1}, \refL{laws}(i), monotony, and \refL{square}(iii) to obtain
 \begin{equation*}
 0\le b\ssm\le (\xoy)\ssm=\yox\per\tag{2}\labelp{Eq:unique.2}
\end{equation*}
Use  \refEq{unique.1}, \refEq{unique.2}, monotony, and
\refL{square}(iv) to arrive at
\begin{equation*}
0\le a;b\ssm  \le (\xoy) ; (\yox)\tag{3}\labelp{Eq:unique.3}
   \le \xox\per
\end{equation*}

The set $A(\xox)$ of elements below $\xox$ is a Boolean algebra with unit $\xox$, under
 the operations of addition and  complement relativized to $\xox$,
by the remarks preceding
\refL{auto}.  The unit $\xox$ is the sum of the set $G_x$\comma by \refC{sumofGx} and the assumption that $x$ is measurable.
  Moreover,    $G_x$ coincides with the set of atoms that are below $\xox$, by \refL{funcatom1}.
    It follows that the relativized Boolean algebra $A(\xox)$ is atomic, and
its set of atoms is $\G\wx$.  In an atomic Boolean algebra, each
 element  is the sum of a uniquely determined
 set of atoms.  Combine these remarks to conclude that each
element below $\xox$ is the sum of a uniquely determined subset of
$G_x$.  This applies in particular to the element $a;b\ssm$, by
\refEq{unique.3}, so there must be   a unique
 subset $E $ of $G_{\wx}$ such that
\begin{equation*}
a;b\ssm=\ssum E\per\tag{4}\labelp{Eq:unique.40}
\end{equation*}

Assume next that $a=b\neq 0$, and observe that $a;a\ssm\neq 0$, by
\refEq{unique.1}, \refEq{unique.2} (with $a$ in place of $b$), and \refL{domain}(iv). The element $a;a\ssm$ is left fixed by converse,
because
\begin{equation*}
  (a;a\ssm)\ssm=a\ssm{}\ssm;a\ssm=a;a\ssm\comma\tag{5}\labelp{Eq:unique.4}
\end{equation*}
by the involution laws (R7) and (R6).  Consequently,
\begin{equation*}
  \ssum E =  a;a\ssm = (a;a\ssm)\ssm =
(\ssum E)\ssm= \ssum\{f\ssm:f\in E\}\comma\tag{6}\labelp{Eq:unique.5}
\end{equation*}
by  \refEq{unique.40} (with $a$ in place of $b$),
\refEq{unique.4}, \refEq{unique.40},  and  complete
distributivity. Two sums of sets of atoms are equal
just in case the sets themselves are equal, so \refEq{unique.5}
implies that
\[ E=\{f\ssm:f\in E\}\per\] Thus,  the set
$E$ is closed under converse. The element $a$ is non-zero, by
assumption, and below $\xoy$, by \refEq{unique.1}, so it
 has as
its domain the atom $x$, by \refL{domain}. It follows that $x\le
a;a\ssm$, by the assumption on $x$ and \refL{domain}(i), and therefore $x$ is in $E$, by the definition of $E$.

The proof for   $a\ssm;b$ is just the dual of the preceding argument.
\end{proof}

Fix an element $a$ below $\xoy$.  Because of the preceding lemma
there is always a unique set of atoms $E\seq G_\wx$ such that
$a;a\ssm=\suml E$.  In what follows, this set will be denoted by
$X_\al$\per  Similarly, there is a unique set of atoms $F\seq
G_\wy$ such that $a\ssm;a=\suml F$, and this set will be denoted
by $Y_\al $.  This notation and the definitions of left- and right-regular elements immediately
imply the following corollary.
\begin{cor}\labelp{C:lrreg}
  An element $a\le\xoy$ is left-regular or right-regular if and
  only if $\cs Xa=\cs Ha$ or $\cs Ya=\cs Ka$ respectively\per
\end{cor}

\begin{lm} \labelp{L:regnonzero}
A left-regular or right-regular element is always non-zero\po
\end{lm}
\begin{proof}
The stabilizer $\cs H0$ of  the zero element $0$ is $\G x$\comma because $f;0=0$ for all
$f\in G_\wx$\comma by \refL{laws.1}(i).  Notice that this stabilizer is not empty since
it contains, for example, the  element $x$. On the other hand,
\[0;0\ssm =0=\tsum \varnot\comma\] by \refL{laws.1}(i), so the set $\cs X0$ of atoms below $0;0\ssm$ is empty.
It follows that the sets $\cs H0$ and $\cs X0$ cannot be equal, and therefore $0$
cannot be left-regular, by \refC{lrreg}. A dual argument proves the corresponding result for right-regular elements.
\end{proof}

In the remainder of this section, we shall usually only formulate lemmas and theorems for left-regular element,
leaving the formulations and proofs of the dual results for right-regular elements to the
reader. When there is a need to refer to such a result, we shall simply refer to ``the right-regular version of \dots".
\begin{lm} \labelp{L:leftcos}
Suppose  $\wx $ and $\wy$ are measurable atoms\comma and $0<a \le \xoy$\po
\begin{itemize}
\item[\opar i\cpar] An element $f$ in $ G_\wx$  belongs to $ X_\al$ if and only if
$(f;a)\cdot a\not= 0$\po
\item[\opar ii\cpar] $\L_\al\seq \L_{a;a\ssm}\seq X_\al\seq
G_\wx$\po
\end{itemize}
\end{lm}
\begin{proof} 
For any element  $f$  in $G_\wx$,
\begin{align*}
f\in X_\al \qquad&\text{if and only if}\qquad f\le a;a\ssm\comma\\
\qquad&\text{if and only if}\qquad f\cdot (a;a\ssm)\not= 0\comma\\
\qquad&\text{if and only if}\qquad  (f;(a\ssm)\ssm)\cdot a\not= 0\comma\\
\qquad&\text{if and only if}\qquad (f;a)\cdot a\not= 0\comma
\end{align*}
 by the definition of $X_\al$\comma the fact that $f$ is an atom (by Lemmas \ref{L:funcatom1}, \ref{L:nonzerofunct},
 and the definition of $\cs Gx$),  the cycle laws, and the first involution law.
This proves (i).

To establish the first inclusion in (ii), recall that
\[\al\ssm\le (\xoy)\ssm=\yox\comma\] by the assumption on $a$, monotony, and \refL{square}(iii).  Apply  \refL{include} (with $a\ssm$ and $x$ in place of $b$ and $z$ respectively)
to arrive at the desired inclusion.  To establish the second inclusion, assume that $f$ is in $\L_{a;a\ssm}$\per
The second assertion of \refL{unique.set} implies that $x$ is below $a;a\ssm$.  Consequently,
\begin{equation*}
f= f;x\le f;a;a\ssm=a;a\ssm=\ssum X_\al\comma\tag{1}\labelp{Eq:leftcos.1}
\end{equation*}
 by \refL{nonzerofunct},   monotony, the assumption that $f $ is in the stabilizer of $ a;a\ssm$, and the definition of  $X_\al$\per
  Since $f$ is an atom, and $\cs Xa$  a set of atoms,
it follows from \refEq{leftcos.1} that $f$ must belong to
$X_\al$\per
The final inclusion in (ii) is a consequence of the definition of  $\cs Xa$.
\end{proof}

It is of some interest to conclude from \refL{leftcos} that the set $X_\al$ is a union of left cosets of
$\L_\al$ in $G_\wx$.
Indeed, if
$f $ and $g$  are in the same left coset of
$\L_\al$\comma then $f;a=g;a$ by \refC{stab}, and therefore
\[
(f;a)\cdot a\not=0
\quad\text{if and only if}\quad
(g;a)\cdot a\not=0\per
\]
It  follows from this equivalence and  part (i) of the preceding lemma that $f$ is in $X_\al$ if and only if $g$ is in $
X_\al$\per In other words, if one element of a left coset of $\L_\al$ belongs to $X_\al$\comma then
the entire coset is included in $X_\al$\per

\begin{lm} \labelp{L:equiv}
Let $\wx$ and $\wy$ be measurable atoms\per For each non-zero $a\le \xoy$\comma the
following are equivalent\per
\begin{enumerate}
\item[\opar i\cpar] $a$ is left-regular\per
\item[\opar ii\cpar] For any $f$ in $G_\wx$\rrcomma\ either $f;a=a$ or
$(f;a)\cdot a=0$\per
\item[\opar iii\cpar] For any $f$ and $g$ in $ G_\wx$\rrcomma\ either $f;a=g;a$ or
$(f;a)\cdot (g;a)=0$\po
\end{enumerate}
\end{lm}

\begin{proof} 
Let $f$ and $g$ be elements of $\G x$, and  $a$  a non-zero element below
$\xoy$.  To establish the implication from (i) to (ii), assume  that
$a$ is left-regular, and observe that
\begin{equation*}\tag{1}\labelp{Eq:equiv.1}
  \L_\al = X_\al\comma
\end{equation*}
  by \refC{lrreg}.  If
$f$ is in $
\L_\al$\comma then
$f;a=a$, by definition of the left stabilizer, and if
$f$ is not in $\L_\al$\comma then
$f$ is not in $ X_\al$\comma by \refEq{equiv.1}, and consequently $(f;a)\cdot  a=0$, by
\refL{leftcos}(i).  Thus, (ii) holds.

To derive (iii) from (ii), observe first that
\begin{equation*}\tag{2}\labelp{Eq:equiv.2}
 f;a= g;a \qquad\text{if and only if}\qquad g\ssm ;f; a=a\per
\end{equation*}  Indeed, if
$f;a=g;a$\comma then
\[g\ssm;f;a=g\ssm;g;a=x;a=a\comma\]
by Lemmas \ref{L:nonzerofunct} and \ref{L:domain}(iii).  On the other hand, if $g\ssm;f;a=a$\comma then $g\ssm;f$
is in the left stabilizer $\h a$ of $a$,   so that $f$ and $g$ must belong to the same left coset of   $\cs Ha$.  Use \refC{stab} to conclude
that $f;a=g;a$.
From \refEq{equiv.2}, (ii) applied to the element $g\ssm;f$, and the cycle laws, it follows that
\begin{align*}
 f;a\not = g;a \qquad &\text{if and only if}\qquad g\ssm ;f; a\not =a\comma\\
&\text{if and only if}\qquad (g\ssm;f; a)\cdot a=0\comma\\
&\text{if and only if}\qquad (g;a)\cdot (f;a)=0\per
\end{align*}

The implication from (iii) to (ii) is trivial: just take $g$ to
be the element $\wx$\comma and use \refL{domain}(iii).

Finally, to derive (i) from (ii), assume  that (ii) holds. Certainly, $\L_\al$ is included in the set $ X_\al$\comma by
\refL{leftcos}(ii). For the reverse inclusion, consider an element $f$ in $\cs Xa$. Use  \refL{leftcos}(i) to see that $(f;a)\cdot a\not=0$,
and then invoke (ii) to obtain $f;a=a$. This implies that $f$ is in the left stabilizer $\cs Ha$, so  $\cs Xa$ is included in $\cs Ha$. Thus,
 \refEq{equiv.1} holds, so $a$ is left-regular, by \refC{lrreg}.
\end{proof}

The next corollary implies that in a measurable relation algebra, atoms are always regular elements. This will play a very important
role in the proof of the representation theorem for measurable relation algebras.
\begin{cor} \labelp{C:atoms}
Let $\wx $ and $\wy$ be measurable atoms\po Every atom below $\xoy$ is
regular\po
\end{cor}

\begin{proof} 
Let $a$ be an atom below $\xoy$. For each $f$ in $G_\wx$  the left translation $f;a$
is also an atom by \refL{auto}, so $f;a=a$ or
$(f;a)\cdot a=0$. Apply \refL{equiv} to conclude that $a$ is left-regular. A dual argument,
involving the version of \refL{equiv} that applies to right-regular elements, shows that $a$ is right-regular. Consequently, $a$ is regular.
\end{proof}

\begin{lm}\labelp{L:lessthan}  Let  $x$ and $y$ be measurable atoms\comma and $a$ and $b$ left-regular
elements below $\xoy$\po\ If $a\le b$\co then $\L_\al\seq \L_\bl$\po
\end{lm}

\begin{proof}
If $a\le b$, then
\begin{equation}\tag{1}\labelp{Eq:lessthan.1}
\ssum \L_\al  = \ssum X_\al
= a;a\ssm
\leq b;b\ssm
= \ssum X_\bl
= \ssum \L_\bl\comma
\end{equation}
by \refC{lrreg} and the assumption that $a$ is left-regular, the definition of $\cs Xa$, the assumption that $a\le b$ and monotony,
the definition of $\cs Xb$, and \refC{lrreg} and the assumption that $b$ is left-regular.
The desired inclusion   follows from \refEq{lessthan.1} and the fact that $\L_\al$ and $\L_\bl$ are sets of atoms.
\end{proof}

   Suppose  $x$ and $y$ are
measurable atoms and  $0<a\le\xoy$.   Fix a left coset system in $\cs Gx$ of the left stabilizer  $\cs Ha$, and denote it by
$\langle \L_{a,\x} : \x<\kp_\al\rangle$ (where $\cs \kappa a$ is, say, an ordinal number that is defined to coincide with the set
of its predecessors).   When no confusion will arise,  we shall
drop the reference to $a$ and write simply $\langle \L_\x : \x <\kp\rangle$.
Similarly, fix a right coset system in $\cs Gy$ of the right stabilizer  $\R_\al$, and denote it by
$\langle \R_{a,\eta} : \eta <\la_\al \rangle$ or simply by $\langle
\R_\eta : \eta<\la\rangle$.  The next lemma  uses the notation $\cs H\xi;a$ that was introduced after \refC{stab} to denote the element $f;a$
for $f$ in $\cs H\xi$.
\begin{lm}[First Partition Lemma] \labelp{L:PartitionLemma}
 Let $\wx$ and $\wy$ be measurable atoms\po\ If $a\le \xoy$
is  left-regular\comma
then $\langle
\L_\x;a : \x <\kp\rangle$ forms a partition of $\xoy$\po
\end{lm}

\begin{proof} 
 It must be shown that
the
elements $\L_\x;a$ are non-zero\co pairwise disjoint\co and sum
to $\xoy$.  They  are pairwise distinct by \refC{stab}, and therefore
 pairwise disjoint by \lmref{L:equiv}(iii).
 The left-regular element $a$ is non-zero, by   \refL{regnonzero}, so its translation $\L_\x;a$ is non-zero by  \refL{auto}.
Finally,
\begin{multline*}
\xoy = \wx;1;a
= \xox;a
= (\ssum G_\wx); a\\
= \ssum \{ h;a : h\in G_\wx\}
= \ssum \{ \L_\x;a : \x <\kp\}\comma
\end{multline*}
by Lemmas \ref{L:domain}(ii),(iii),
\refC{sumofGx}, and complete distributivity.
\end{proof}

 \begin{cor}\labelp{C:equalmult}  Let  $\wx$ and $\wy$ be measurable
atoms\comma and
$a\le \xoy$  a left-regular element\per If $X$ and $Y$ are unions of left cosets of $\L_\al$\comma then
\begin{alignat*}{3}\tsum X;a &\le \ssum Y;a&\qquad&\text{if and only if}&\qquad X&\seq Y\comma\\
\intertext{and consequently\comma}\tsum X;a &= \ssum Y;a&\qquad&\text{if and only if}&\qquad X&= Y\per
\end{alignat*}
\end{cor}
\begin{proof}  The proof of the implication  from right to left in the first assertion is   trivial.  To prove the reverse
implication, assume that
\begin{equation*}\tag{1}\labelp{Eq:equalmult.1}
  \tsum X;a\le\tsum Y;a\comma
\end{equation*}
and  consider a left coset $\L_\xi$ that is included in
$X$.  Obviously,
\begin{equation*}\L_\xi;a\le \ssum X;a \le \ssum Y;a\comma
\end{equation*} by \refEq{equalmult.1} and the fact that
the element $\cs H\xi;a$ belongs to the set $X;a$.
Two left translations of $a$ by cosets of $\cs Ha$ are either equal or  disjoint, by \refL{PartitionLemma},
so there must be a left coset $\L_\eta$ included in $Y$ such that
$\L_\xi;a=\L_\eta;a$.  Distinct left cosets of $\L_\al$ give rise to disjoint left
translations of $a$, again by \refL{PartitionLemma}, so $\xi=\eta$.  Thus, every  left coset of $\cs Ha$  that is included in $X$ is also included in  $Y$, and therefore $X$ must be  included in  $ Y$.  This proves the first assertion of the corollary.  The second  is an immediate consequence of the first.
\end{proof}

A sense of the importance of    Partition \refL{PartitionLemma} can be gained from the following
consequence.

\begin{lm}[Atomic Partition Lemma]\labelp{L:sumatom} Let $\wx$ and
$\wy$ be measurable atoms\po\ If $a\le \xoy$  is an atom\comma
  then  $\langle \L_\x;a :\x<\kp\rangle$ is
a listing of the distinct atoms below $\xoy$ and these atoms sum to $\xoy$\po
Consequently\co every element below $\xoy$ is the sum  of a unique subset of
these atoms\po
\end{lm}

\begin{proof} An atom $a$   below $\xoy$  is a regular element, by
\refC{atoms}, so the elements in the system
\begin{equation*}\tag{1}\labelp{Eq:sumatom.1}
  \langle \L_\x;a:\x\le\kp\rangle
\end{equation*}
    form a
partition of $\xoy$, by Partition \refL{PartitionLemma}. In particular, they are mutually disjoint and sum to
the unit $\xoy$ of the relativized Boolean algebra $A(\xoy)$. Moreover, they are all atoms, by \refL{auto}.  It follows by Boolean algebra that the
relativized Boolean algebra $A(\xoy)$ is atomic, and therefore each of its elements is the sum of a unique set of atoms from \refEq{sumatom.1}.
\end{proof}

The preceding lemma says that if there is an atom $a$ below $\xoy$, then every left
translation of $a$ is again an atom and these atoms partition $\xoy$.  The
same is of course true for the right translations of $a$.

Suppose
$x$ and $y$ are measurable atoms, and
$a$ and $b$  left-regular elements below $\xoy$ with
$a\le b$.  By \refL{lessthan}, the left stabilizer   $\cs Ha$ is a
subgroup of the left stabilizer   $\cs Hb$.  Let
\[\langle \lax:\xi<\ka\rangle\qquad\text{and}\qquad \langle \lbe:\eta<\la\rangle\] be  left coset systems for $\L_\al$   and
 $\L_\bl$ respectively in $\G x$.  As is well known from group theory, there must be  a partition $\langle
\Gamma_\eta:\eta<\la\rangle$ of the index set $\{\xi:\xi<\ka\}$ such that
\begin{equation*}\lbe = \tbigcup\{\lax: \xi\in\Gamma_{\eta} \}
\end{equation*}
for each $\eta<\lambda$.  The next lemma, a generalization of the First Partition Lemma, refers to these assumptions.

\begin{lm}[Second Partition Lemma]\labelp{L:GeneralizedPartitionLemma}  Let $x$ and $y$ be
measurable atoms\comma and
$a$ and $b$  left-regular elements below $\xoy$\po  If $a\le b$\co
then
$\langle \lax; a : \xi\in\Gamma_\eta\rangle$ is a
partition of $\lbe;b$\co  and in particular\comma
\[\lbe;b= \ssum
\lbe;a=\tsum\{\lax;a:\xi\in\cs\Gamma\eta\}\] for each
$\eta<\la$\po
\end{lm}
\begin{proof} 
By assumption, $a\le b$.  Also,  for each $\xi$ in $\Gamma_\eta$, the left coset $\lax$ is a subset of the left coset $\lbe$, by the remarks
preceding the lemma\per  Use monotony and \refC{stab} to obtain
\begin{equation*}
\lax;a\le \lax;b=\lbe;b\per\tag{1}\labelp{Eq:P20/1}\end{equation*}
 The system $\langle \lbe;b :\eta<\la\rangle$
is a partition of $\xoy$\comma by    Partition \refL{PartitionLemma}, so the elements
$\L_{b,\zeta};b$ and $\lbe;b$
 are disjoint for indices $\zeta, \eta<\lambda$ with $\zeta\not=\eta$.
 Consequently, for $\xi$ in $\cs\Gamma\zeta$,
\begin{equation*}
(\lax;a)\cdot(\lbe;b)\le (\lax;b)\cdot(\lbe;b) = (\L_{b,\zeta};b)\cdot  ( \lbe;b)=0\comma
\tag{2}\labelp{Eq:P20/2}\end{equation*} by \refEq{P20/1} (with $\zeta$ in place of $\eta$) and Boolean algebra.
 The system
\begin{equation*}\tag{3}\labelp{Eq:P20/3}
  \langle
\lax;a:\x<\ka\rangle
\end{equation*}
 is also a partition of
$\xoy$, by    \refL{PartitionLemma}, so the elements of this system are non-zero, mutually disjoint, and
\[
\ssum\{\lax;a:\xi<\ka\} = \xoy\per
\]
 Multiply both sides of this last equation by $\lbe;b$, and use \refEq{P20/1},
\refEq{P20/2}, and Boolean algebra to obtain
\begin{equation*}\tag{4}\labelp{Eq:P20/4}
\tsum\{\lax;a:\xi\in\Gamma_\eta\}= \lbe;b\per
\end{equation*}
Conclusion: $\langle \lax:\xi\in\Gamma_\eta\rangle$ is a partition of $\lbe;b$.

The second
assertion of the lemma is an almost immediate consequence of the first:
\begin{equation*}\lbe;b =
\tsum\{\lax;a:\xi\in\Gamma_\eta\}=\tsum\bigl(\tbigcup\{\lax:
\xi\in\Gamma_\eta\}\bigr);a = \tsum \lbe;a\comma
\end{equation*} by \refEq{P20/4},   complete distributivity, and the remarks preceding the lemma.
\end{proof}

Non-zero products of regular elements play an important role in the analysis of the
behavior of regular elements.

\begin{lm}[First  Product Lemma] \labelp{L:a-b.leftregular} Let $\wx$ and $\wy$ be measurable
atoms\co and   $a$ and $b$  left-regular elements below $x;1;y$\po
If $a\cdot b\not=0$\co then
$a\cdot  b$ is left-regular and $\L_{a\cdot  b}=
\L_\al\cap \L_\bl$\po
\end{lm}

\begin{proof} 
Observe that
 \begin{equation*} 
   (a\cdot b); (a\cdot b)\ssm= (a\cdot b); (a\ssm\cdot b\ssm)\le
(a;a\ssm)\cdot (b;b\ssm)\comma
 \end{equation*}
 by \refL{laws}(ii) and monotony, so
\begin{equation*} \tag{1} \labelp{Eq:P16/1}
   \ssum X_{a\cdot b}\seq (\ssum X_\al)\cdot (\ssum X_\bl)\comma
 \end{equation*}
by the  definitions of the sets $X_{a\cdot b}$\comma $X_\al$\comma and $X_\bl$\per
Since these are all sets of atoms, it follows from \refEq{P16/1} by Boolean algebra that
\begin{equation*} 
X_{a\cdot b}\seq X_\al \cap X_\bl\per \tag*{(2)} \labelp{Eq:P16/3}
\end{equation*}
Use \refL{leftcos}(ii) (with $a\cdot b$ in place of $a$), the assumption that $a\cdot b\neq 0$, \ref{Eq:P16/3},  the assumed left-regularity of $a$ and $b$, and \refC{lrreg} to obtain
\begin{equation*} \tag{3}\labelp{Eq:P160/4}
\L_{a\cdot b} \seq X_{a\cdot b}
 \seq X_\al\cap X_\bl
= \L_\al\cap \L_\bl\per
\end{equation*}
On the other hand, if $f$ is in $ \L_\al\cap \L_\bl$\comma then $f$ is also in $
\L_{a\cdot b}$\comma because
\[
f;(a\cdot b) = (f;a)\cdot (f;b)= a\cdot b\comma
\]
by the distributive law for functions.    Consequently,
\begin{equation*}\tag{4}\labelp{Eq:P160/5}
  \cs Ha\cap\cs Hb\seq \cs H{a\cdot b}\per
\end{equation*}
 Combine \refEq{P160/4} with \refEq{P160/5} to arrive at
 \[\cs H{a\cdot b}=\cs X{a\cdot b}=\cs Ha\cap\cs Hb\per
 \] The left-regularity of the product $a\cdot b$ is an immediate consequence of the first
 of these equalities and \refC{lrreg} (with $a\cdot b$ in place of $a$).
\end{proof}

\begin{lm}\labelp{L:char.less}Let $\wx$ and $\wy$ be measurable atoms\comma  and  $a$ and $b$
 left-regular elements below $x;1;y$\per  If  $a\cdot b\neq 0$\comma then
 \begin{alignat*}{3}
  \L_\al&\seq \L_\bl &\qquad&\text{if and only if}&\qquad a&\le b\comma \\
  \intertext{and consequently}
   \L_\al&= \L_\bl &\qquad&\text{if and only if}&\qquad a&=b\per
   \end{alignat*}
\end{lm}
\begin{proof}  Assume that $a\cdot b\neq 0$.  If $a\le b$, then $\cs Ha$ is included in $\cs Hb$, by  \refL{lessthan}.  To establish
 the reverse implication, suppose that $\L_\al\seq \L_\bl$\per  Use this assumption, the assumption that $a\cdot b\neq 0$, and Product \refL{a-b.leftregular} to
  see that   $a\cdot b$ is a left-regular element, and that
\begin{equation*} 
\L_{a\cdot b}=\L_\al\cap\L_\bl= \L_\al\per \tag*{(1)} \labelp{Eq:P16/4}
\end{equation*}
Let $\langle f_\xi:\xi<\ka\rangle$ be a system of representatives for the left
cosets of $\L_\al$ in $\G x$\per   Partition \refL{PartitionLemma} (with $\cs f\xi$ in placed of $\cs H\xi$) says that
\begin{equation*}\tag{2}\labelp{Eq:char.less.2}
  \langle f_\x ;a :\x <\kp\rangle
\end{equation*}
 is a partition of $
\xoy.$
Now  $\langle f_\x : \x <\kp\rangle$ is also a   system
of
representatives for  the left cosets of $\L_{a\cdot b}$\comma by \ref{Eq:P16/4}, so
\begin{equation*}\tag{3}\labelp{Eq:char.less.3}
  \langle f_\x ; (a\cdot b):\x<\kp\rangle
\end{equation*}
is  a partition of
$\xoy$, by \refL{PartitionLemma} and the left-regularity of   $a\cdot b$.  Since
\begin{equation*}
f_\x ; (a\cdot b)\le f_\x ;a\tag*{(4)} \labelp{Eq:P16/5}
\end{equation*} for each $\xi$,
by monotony,   the   partitions in \refEq{char.less.2} and \refEq{char.less.3} must be equal.  Therefore, equality must actually hold in \ref{Eq:P16/5}.
 Use this observation and the distributive law for  functions
to obtain
\[
f_\x; a = f_\x; (a\cdot b)= (f_\x ;a) \cdot (f_\x ;b)\per
\] It follows that
\[f_\x;a\le f_\x;b\per\] Take
$\x=0$ in this inequality, and use \refL{domain}(iii), together with the convention that
$f_0=\wx$, to arrive at
$a\le b$\per  This completes the proof of the first equivalence in the lemma.

The second equivalence  is an immediate consequence of the first.
\end{proof}

The preceding lemma leads naturally to the question,  for two left-regular elements $a$ and $b$
below $\xoy$,  when is the  product $a\cdot b$ non-zero? A necessary and sufficient condition
for this to happen is given below in the Second  Product Lemma.

The next lemma says  that any translation, left or right, of a
left-regular element $a$ is again left-regular, and the left stabilizer of such a
translation can be computed from the left stabilizer of $a$.

\smallskip
\begin{lm}[First Translation Lemma] \labelp{L:left.reg}
Let $\wx$ and $\wy$ be measurable atoms\co and
$a\le\xoy$ a left-regular element\per
\begin{enumerate}
\item[\opar i\cpar] For every $f$ in $G_\wx$\comma the left translation $f;a$ is
left-regular\comma and its left stabilizer is \[\L_{f;a}= f;\L_\al;f\ssm\per\] If $\L_\al$ is a
normal subgroup of $\cs Gx$\comma then $\L_{f;a}=\L_\al$\po
\item[\opar ii\cpar] For every element  $g$ in $G_\wy$\comma the right translation
$a;g$ is left-regular and its left stabilizer is $\L_{a;g}= \L_\al$\po
\end{enumerate}
\end{lm}

\begin{proof} 
Consider an  element  $f$   in $ G_\wx$\per Use the definition of the set $X_{f;a}$, the second involution law and the associative law for relative multiplication,
the definition of the set $\cs Xa$, the assumed left-regularity of the element $a$, together with \refC{lrreg}, and   complete distributivity  to obtain
\begin{multline*}\tag{1}\labelp{Eq:left.reg.1}
\ssum  X_{f;a} = (f;a); (f;a)\ssm
 = f;a;a\ssm;f\ssm
 = f;(\ssum  X_\al);f\ssm\\ = f; (\ssum  \L_\al);f
 = \ssum  (f;\L_\al;f\ssm)\per
\end{multline*}
It is easy to check, and it follows from \refL{mappings} and group theory, that the left stabilizer of the left translation $f;a$ is the subgroup
\begin{equation*}\L_{f;a}=f;\L_\al;f\ssm\per\tag*{(2)}\labelp{Eq:A1/14}
\end{equation*} Combine \refEq{left.reg.1} with \ref{Eq:A1/14} to arrive at
\[\ssum  X_{f;a}= \ssum \L_{f;a}\per\]
Since $X_{f;a}$ and $\L_{f;a}$ are both sets of atoms,   the preceding equation implies that the two sets must be  equal. Use \refC{lrreg} (with $f;a$ in place of $a$) to conclude that $f;a$ is left-regular.
 If, in addition, $\L_\al$ is a normal subgroup of $\cs Gx$, then this subgroup
must coincides with $f;\L_\al;f\ssm$, and therefore also with $ \L_{f;a}$\comma by \ref{Eq:A1/14}.
This proves (i).

To prove (ii), assume that  $g $ is in $ G_\wy$\per Use the definition of the set $X_{a;g}$, the second involution law and the associative law for relative
multiplication, \refL{nonzerofunct} (with $y$ in place of $x$), \refL{domain}(iii),    and the  left-regularity of $a$ to arrive at
\begin{equation*}
\ssum  X_{a;g} = (a;g);(a;g)\ssm
=a;g;g\ssm;a\ssm
= a;\wy;a\ssm
  = a;a\ssm
 = \ssum  \L_\al\per
\end{equation*}
Since $\cs X{a;g}$ and $\cs Ha$ are sets of atoms, we may conclude  from this computation that
\begin{equation*}\tag{3}\labelp{Eq:left.reg.3}
  X_{a;g}= \L_\al\per
\end{equation*}
 Use Lemmas \ref{L:nonzerofunct} and \ref{L:domain}(iii) to obtain
\begin{align*}
f;a;g= a;g
&\qquad\text{if and only if}\qquad
f;a;g;g\ssm= a;g;g\ssm\comma\\
&\qquad\text{if and only if}\qquad
f;a;y= a;y\comma\\
&\qquad\text{if and only if}\qquad
f;a=a\per
\end{align*}
These equivalences show that
$f$ belongs to the left stabilizer $ \L_{a;g}$ if and only if it belongs to the left stabilizer $\L_\al$\comma so that
\begin{equation*}\tag{4}\labelp{Eq:left.reg.4}
 \L_{a;g}= \L_\al\per
\end{equation*}   Combine \refEq{left.reg.3} amd \refEq{left.reg.4} to arrive at
\[ X_{a;g}=\L_{a;g}=\cs Ha\comma\] and use these equalities together with \refC{lrreg} (with $a;g$ in place of $a$) to conclude that $a;g$ is left-regular
and its left stabilizer is $\cs Ha$.
\end{proof}

The following corollary is a very important   consequence of    Atomic Partition \refL{sumatom}  and Translation \refL{left.reg}.

\begin{cor} \labelp{C:a-b.atoms}
Let $\wx$ and $\wy$ be measurable atoms.  If  $a\le \xoy$ is an atom\comma then
its left and right stabilizers $\L_\al$  and $\R_\al$ are normal
subgroups of   $G_\wx$ and $G_\wy$ respectively\po If $b\le \xoy$ is also an atom\comma
then $\L_\bl=\L_\al$ and $\R_\bl=\R_\al$\po
\end{cor}
\begin{proof} 
Let $a $ and $b$ be arbitrary atoms below $\xoy$.  The  version
of \refL{sumatom} for right-regular elements says that the right translations  of $a$ constitute all of the atoms below
$\xoy$.  In particular, there must be an element $g$ in $\G y$ such that
$b= a;g$.  Use this equality and part (ii) of  \refL{left.reg} to
obtain
\begin{equation*}\L_\bl=\L_{a;g}=\L_\al\per \tag*{(1)} \labelp{Eq:P15/2}
\end{equation*}

For any element  $f$ in $G_\wx$\comma the left translation $f;a$ is an atom below $\xoy$, by
   \refL{sumatom}.  Take this element for $b$ in \ref{Eq:P15/2}, and use part (i) of  \refL{left.reg}
  to arrive at
  \begin{equation*}\tag{2}\labelp{Eq:a-b.atoms.2}
    \L_\al=\L_{f;a}= f;\L_\al;f\ssm\per
  \end{equation*}
The equality of the first and last terms in \refEq{a-b.atoms.2} for every $f$ in $\cs Gx$ implies that the subgroup $\cs Ha$ is normal in $\cs Gx$,
and \ref{Eq:P15/2} implies that every atom $b\le \xoy$ has the same left stabilizer as the atom $a$.
A dual argument yields the corresponding result for the right stabilizer $\cs Ka$\per\end{proof}

If a left-regular element $a\le\xoy$ has a
normal left stabilizer, then part (i) of Translation \refL{left.reg} implies that any left translation of $\al$
is a left-regular element with the  same left stabilizer as $a$.  The next lemma implies that any  left-regular element
below $\xoy$ with the same left stabilizer as   $a$ must in fact be
a left translation of $a$.  Thus, the left translations of $\al$ are
precisely the left-regular elements below $\xoy$ with the same left stabilizer as $a$.   In
fact, this property characterizes  left-regular elements with normal
left-stabilizers.

\begin{lm}[Second Translation Lemma] \labelp{L:assertioneq}
 Let $\wx$ and $\wy$ be measurable
atoms\per For every left-regular element   $a\le
\xoy$\comma the following assertions are
equivalent\per
\begin{itemize}
\item[\opar i\cpar] $\L_\al$ is a normal subgroup of $\cs Gx$\per
\item[\opar ii\cpar] For every left-regular element $b\le \xoy$\comma we have $\L_\bl\seq\L_\al$
if and only if
$b$ is below some left translation of
$a$\per
\item[\opar iii\cpar] For every left-regular element  $b\le \xoy$\comma  we have  $\L_\al\seq\L_\bl$
if and only if
$b$ is above some left translation of
$a$\per
\item[\opar iv\cpar]  For every left-regular element  $b\le \xoy$\comma we have  $\L_\bl=\L_\al$ if and only if $b$ is equal to some
left translation of $a$\per
\end{itemize}
\end{lm}
\begin{proof} 
Assume  $a\le \xoy$ is  left-regular. For each
element $f$ in $\G x$\comma the left trans\-la\-tion $f;a$ is left-regular\comma
and
\begin{equation*} 
\L_{f;a}= f;\L_\al;f\ssm\comma \tag*{(1)} \labelp{Eq:P17/3}
\end{equation*}
by part (i) of  Translation \refL{left.reg}.

To establish the implication
from (i) to each of (ii), (iii), and (iv), assume that the left stabilizer
$\L_\al$ is a normal subgroup, and use \ref{Eq:P17/3} to obtain
\begin{equation*} 
\L_{f;a}= \L_\al\per \tag*{(2)} \labelp{Eq:P17/4}
\end{equation*}  Consider any left-regular element $b\le\xoy$\per    Partition \refL{PartitionLemma} implies  that
\begin{equation*}
\xoy=
\ssum \{ f;a: f\in G_x\}\comma
\end{equation*}
and   $b$ is non-zero, by \refL{regnonzero},
so
\begin{equation*} 
b\cdot (f;a)\not= 0 \tag*{(3)} \labelp{Eq:P17/1}
\end{equation*}
 for some $f$ in $ G_x$\per
Use \ref{Eq:P17/4}, and then use \refL{char.less} (with $b$ and $f;a$ in place of $a$ and $b$ respectively) and  \ref{Eq:P17/1}, to arrive at
\begin{alignat*}{3}
\L_\bl &\seq \L_\al& \qquad&\text{if and only if}&\qquad \L_\bl&\seq \L_{f;a}\comma\\
&&\qquad&\text{if and only if}&\qquad  b&\leq f;a\per\\
\intertext{A similar argument yields}
\L_\al &\seq \L_\bl& \qquad&\text{if and only if}&\qquad \L_{f;a}&\seq \L_\bl\comma\\
&&\qquad&\text{if and only if}&\qquad f;a&\leq  b\per\\
\intertext{Combine  these equivalences to conclude that}
\L_\bl&=\L_\al&\qquad&\text{if and only if}&\qquad
f;a&=b\per
\end{alignat*}

\def\f{g}
\def\g{f}
To establish the  implication from (ii) to (i), assume that (ii) holds,  and consider an arbitrary element $\g$    in $ G_\wx$\per The element $b=\g;a$
is left-regular, by the initial observation of this proof, and  obviously $b\le \g;a$, so (ii) implies that
\begin{equation*}\tag{4}\labelp{Eq:assertioneq.4}
\L_\bl\seq\L_\al\per
\end{equation*}
Use \ref{Eq:P17/3}, the choice of $b$, and \refEq{assertioneq.4} to see that
\begin{equation*}\tag{5}\labelp{Eq:assertion.5}
  \g;\L_\al;\g\ssm=\cs H{f;a}=\cs Hb\seq
\L_\al\per
\end{equation*}
 The inclusion of the left side of \refEq{assertion.5} in the right side holds for all $\g$ in $\cs Gx$, so  $\L_\al$
must be a normal subgroup of $\cs Gx$.

The proof of the implication from (iii) to
(i) is similar to the preceding argument, but uses the fact that the subgroup $\LS\al$ is normal just
in case \[\lal\seq f;\lal;f\ssm\] for every element $f$ in $\G\wx$\per
The  implication from
(iv) to (i) is a consequence of the implication from (ii) to
(i).
\end{proof}

Product \refL{a-b.leftregular} has as a hypothesis that the
product of the two left-regular elements $a$ and $b$ be
non-zero.  The next lemma gives necessary and sufficient conditions for this hypothesis to
be satisfied, under the additional assumption that the left stabilizers are normal. It also characterizes the product
subgroup $\cs Ha;\cs Hb$ as the left stabilizer of a specific element.  Recall from \refL{a-b.leftregular} that
\[
\cs H{a\cdot b}=\cs Ha\cap\cs Hb\comma\] so that the coset system for $\cs H{a\cdot b}$ coincides with the coset system for $\cs Ha\cap\cs Hb$, which is
\[\langle\lax\cap\lbe:\xi<\kappa\text{ and }\eta<\lambda\rangle\comma\]
where
\[\langle\lax:\xi<\kappa\rangle\qquad\text{and}\qquad\langle \lbe:\eta<\lambda\rangle \] are respectively coset systems for $\h a$ and $\h b$ in $\h a;\h b$.

\begin{lm}[Second  Product Lemma] \labelp{L:prodlemma}  Let $\wx $ and $\wy$ be measurable
atoms\comma and  $a$ and $b$
 left-regular elements below $x;1;y$ with normal stabilizers $\lal$ and $\lbl$\po
\begin{enumerate}
\item[\opar i\cpar] $a\cdot b\neq 0$  if and only if $a;a\ssm;b=b;b\ssm;a$\po
\item[\opar ii\cpar]  If $a\cdot b\neq 0$\co then the product subgroup $\cs Ha;\cs Hb$ is the left stabilizer of the element
$b;b\ssm;a$\comma and the system of left translations
\[\langle(\lax\cap\lbe);(a\cdot b):\xi<\kappa\text{ and }\eta<\lambda\rangle\] is a partition of $b;b\ssm;a$\comma
where
$\langle\lax:\x<\ka\rangle$ and  $\langle\lbe:\eta<\la\rangle$ are cosets systems for
$\lal$ and $\lbl$ in $\lal;\lbl$\po Different left translations of
$a\cdot b$ coincide with the different products of the left translations of $a$ and
$b$ in the sense that
\[(\lax\cap\lbe);(a\cdot b)=(\lax;a)\cdot (\lbe;b)\] for every $\xi<\kappa$ and $\eta<\lambda$\per
\end{enumerate}
\end{lm}

\begin{proof} 
 To prove (i), assume first   that $a\cdot b\neq 0$.  The product $a\cdot b$
is then a left regular element, by  Product \refL{a-b.leftregular}.  Consequently,
\begin{equation*}\tag{1}\labelp{Eq:prodlemma.1}
\al;\al\ssm;(a\cdot b)=(\tsum\lal);(a\cdot b)=\tsum\lal;(a\cdot b) = \lal;a=a
\end{equation*}
by the assumed left-regularity of $\al$,   complete distributivity,
  the final assertion of Partition \refL{GeneralizedPartitionLemma} (with $a\cdot b$ and $a$ in place of $a$ and $b$ respectively,   and $\lal$ in place of  $\lbe$), and the
  fact that $\cs Ha$ is the stabilizer of $a$. Use \refEq{prodlemma.1} and monotony to get
\[a\le \al;\al\ssm;b\per\]
  Form the relative product of both sides of this inequality on the left with $\bl;\bl\ssm$, and then use monotony, the definition of a regular element and the assumed left-regularity of $a$
  and $b$,   complete distributivity, the assumption that $\cs Ha$ is normal, the definition of $\cs Hb$
   as the stabilizer of $b$, complete distributivity, and the left-regularity of $a$, to arrive at
\begin{multline*}
\bl;\bl\ssm;a\le\bl;\bl\ssm;\al;\al\ssm;b =(\tsum\lbl);(\tsum\lal);b
=\tsum\lbl;\lal;b\\=\tsum\lal;\lbl;b=\tsum\lal;b=(\tsum\lal);b=\al;\al\ssm;b\per
\end{multline*}
A symmetric argument yields the  reverse inequality.  This establishes the
implication from left to right in part (i).

To establish the reverse implication, assume that
\begin{equation*}\al;\al\ssm;b=\bl;\bl\ssm;a\per\tag{2} \labelp{Eq:PL/1}
\end{equation*}
Use the definition of a left-regular element and the assumed left-regularity of $\bl$,  complete distributivity, the definition of $\cs Ha$ as the stabilizer of $a$,  and the
assumption that
this stabilizer is a normal subgroup to get
\begin{equation*}\bba=(\tsum \lbl);\al=\tsum \lbl;\al = \tsum \lbl;\lal;\al =\tsum \lal;\lbl;\al\per
\tag{3} \labelp{Eq:PL/2}
\end{equation*}
The product subgroup $\lal;\lbl$ is the union of the cosets  $\lax$ of
$\lal$ in $\cs Ha;\cs Hb$\comma by assumption, so
\begin{equation*}\tag{4}\labelp{Eq:prodlemma.4}
  \tsum\lal;\lbl;\al=\tsum (\tbigcup_\xi\lax);a=\tsum_\xi\lax;a\comma
\end{equation*}
by complete distributivity.  Also, the elements $\lax;\al$ are  non-zero and pairwise
disjoint, by   Partition \refL{PartitionLemma}. Combine this observation with \refEq{PL/2} and \refEq{prodlemma.4} to see that
\begin{gather*}
   \langle \lax;a:\x<\ka\rangle\tag{5}\labelp{Eq:prodlemma.5}\\
 \intertext{is a partition of  $
\bba$\per
A similar argument shows that}
\langle \lbe;b:\eta<\la\rangle\tag{6}\labelp{Eq:prodlemma.6}
\end{gather*}   is a partition of   $
\aab$.

Use \refEq{PL/1}, \refL{domain}(i), monotony, \refL{domain}(iii),   the
left-regularity of
$b$, and \refL{regnonzero} to obtain
\begin{equation*}\tag{7}\labelp{Eq:prodlemma.07}
  b;b\ssm;a=\aab\ge x;b=b>0\per
\end{equation*}
Since \refEq{prodlemma.5} is a partition of $b;b\ssm;a$, it follows from \refEq{prodlemma.07} that there must be an index $\gm<\ka$ such that
\[(\L_{a,\gm};\al)\cdot b \neq 0\per\]
Put $\bar{a}=\L_{a,\gm};a$ and write the preceding inequality as
\begin{equation*}\bar{a}\cdot \bl\neq 0\per
\tag{8} \labelp{Eq:PL/6}
\end{equation*}
The element $\bar{a}$ is, by definition, a left translation of the left-regular element $a$ with a normal
left stabilizer,  so $\bar{a}$ is itself left-regular with the same normal
left stabilizer
$\lal$\comma by  part (i) of Translation \refL{left.reg}.  In view of \refEq{PL/6},
 Product \refL{a-b.leftregular} may be applied (with $\bar{a}$ in place of $a$) to
conclude that
$\bar{a}\cdot b$ is left-regular with left stabilizer $\lal\cap\lbl$\per

As is well known from group theory, the normal subgroup $\lal\cap\lbl$ has the  coset system
\begin{equation*}\tag{9}\labelp{Eq:prodlemma.09}
  \langle\lax\cap\lbe:\x<\ka\text{ and } \eta<\la\rangle
\end{equation*}
in $\lal;\lbl$.  Every left translation of a left-regular element
by a left coset of its left stabilizer is again left-regular, by part (i) of Translation \refL{left.reg}.  In
particular, each left translation
\begin{equation*}(\lax\cap\lbe);(\bar{a}\cdot \bl)
\end{equation*} of $\bar{a}\cdot b$
is left-regular, and therefore non-zero, by  \refL{regnonzero}.  Choose $\x$    so
that
$\lax$ is the coset inverse of $\L_{a,\gm} $ in the quotient group $\cs Gx/\cs Ha$\comma   and use monotony, the definition of $\bar{a}$, the inverse property from group theory, and the definition of $\cs Ha$ as the left stabilizer of $a$
to obtain
\begin{equation*}\tag{10}\labelp{Eq:prodlemma.9}(\lax\cap\lbe);(\bar{a}\cdot \bl)\le \lax;\bar{a} = \lax;\L_{a,\gm};\al =
\lal;\al = \al\per
\end{equation*}
  Similarly, take  $\eta=0$, so that $\lbe$
coincides with the identity coset $\lbl$\comma and use monotony and the definition of $\cs Hb$ to obtain
\begin{equation*}\tag{11}\labelp{Eq:prodlemma.101}(\lax\cap\lbe);(\bar{a}\cdot \bl)\le \lbe;\bl =   \lbl;\bl
=\bl\per
\end{equation*}
Form the products of the left and right sides of  \refEq{prodlemma.9} and \refEq{prodlemma.101},
and use Boolean algebra
 to arrive at
\[0<(\lax\cap\lbe);(\bar{a}\cdot \bl)\le a\cdot b\per\]
Conclusion:  $\al\cdot\bl\neq 0$, as was to be shown.

To prove (ii), assume that $\al\cdot\bl\ne 0$. It follows from part (i) of the lemma that \refEq{PL/1} also holds, so it makes sense to write
 \begin{equation*}\tag{12}\labelp{Eq:prodlemma.012}
   c=b;b\ssm;a=a;a\ssm;b\per
\end{equation*} The first task is to check that the product subgroup
 $\cs Ha;\cs Hb$ coincides with the left stabilizer $\cs Hc$.
  Consider   an element $h$ in $\cs Gx$.  If $h$ is in the   product subgroup, then there must be elements $f$ in $\cs Ha$ and $g$ in $\cs Hb$ such that $h=f;g$, by the definition of the product subgroup. Consequently,
 \begin{multline*}\tag{13}\labelp{Eq:prodlemma.10}
  h;c= h;b;b\ssm;a=f;g;b;b\ssm;a=f;b;b\ssm;a\\=f;a;a\ssm;b=a;a\ssm;b=b;b\ssm;a=c\comma
 \end{multline*} by  \refEq{prodlemma.012}, the assumptions on $h$, the assumption that $g$ is in the left stabilizer of $b$, part (i) of the lemma, which implies that \refEq{PL/1} holds, the assumption that
 $f$ is in the left stabilizer of $a$, and  \refEq{PL/1} again. It follows from \refEq{prodlemma.10} that $h$ belongs to the left stabilizer $\cs Hc$.

 On the other hand, if $h$ belongs to $\cs Hc$\comma then
 \begin{equation*}\tag{14}\labelp{Eq:prodlemma.014}
   h;b;b\ssm;a=h;c=c=b;b\ssm;a\comma
 \end{equation*}
  by \refEq{prodlemma.012} and the definition of $\cs Hc$. Form the relative product of the left and right sides of \refEq{prodlemma.014}, on the right, with $a\ssm$ to obtain
 \begin{equation*}\tag{15}\labelp{Eq:prodlemma.11}
   h;b;b\ssm;a;a\ssm=b;b\ssm;a;a\ssm\per
 \end{equation*}
 Use complete distributivity, the left-regularity of $a$ and $b$, \refEq{prodlemma.11}, the left-regularity of $a$ and $b$ again,  and complete distributivity to get
 \begin{multline*}\tag{16}\labelp{Eq:prodlemma.12}
   \tsum \cs Hb;\cs Ha=(\tsum\cs Hb);(\tsum \cs Ha)=b;b\ssm;a;a\ssm=h;b;b\ssm;a;a\ssm\\=h;(\tsum\cs Hb);(\tsum\cs Ha)=h;(\tsum\cs Hb;\cs Ha)=\tsum h;\cs Hb;\cs Ha\per
 \end{multline*} The sets involved in the first and last sums of \refEq{prodlemma.12} are sets of atoms, so the two sets must be equal.  Combine this with the assumption that $\cs Ha$ is normal to arrive at \[\cs Ha;\cs Hb=\cs Hb;\cs Ha=h;\cs Hb;\cs Ha=h;\cs Ha;\cs Hb\per\] Thus, the left coset of the product subgroup $\cs Ha;\cs Hb$ determined by each element  $h$
 belonging to the left stabilizer  $\cs Hc$ coincides with    $\cs Ha;\cs Hb$, so each such  $h$ belongs to the product subgroup.  Conclusion:
  \begin{equation*}\tag{17}\labelp{Eq:prodlemma.17}
    \cs Hc=\cs Ha;\cs Hb\per
\end{equation*}

 Product \refL{a-b.leftregular}, and the assumption that $a$ and $b$ are left-regular elements with   $a\cdot b\neq 0$ imply that  $\al\cdot\bl$ is a left-regular element with
  left stabilizer $\lal\cap\lbl$\per  Recall that \refEq{prodlemma.09} is a coset system for this left stabilizer in $\cs Ha;\cs Hb$.  Consequently,  $\cs Hc$ is the union of the  cosets $ \lax\cap\lbe$ for $\xi<\ka$ and $\eta<\la$, by   \refEq{prodlemma.17}.  The assumption that $a\cdot b\neq 0$, together with Boolean algebra, \refEq{prodlemma.07}, and \refEq{prodlemma.012}, implies that
  \[0<a\cdot b\le b\le c\comma\] In view of these observations,
    Partition \refL{GeneralizedPartitionLemma} may be applied (with $a\cdot b$ and $c$ in place of $a$ and $b$ respectively, and with $\lax\cap\lbe$ and  $\cs Hc$ in place of $\lax$ and $\lbe$ respectively)   to conclude that
\begin{equation*}\tag{18}\labelp{Eq:prodlemma.018}
  \langle (\lax\cap\lbe);(\al\cdot\bl):\xi<\kappa\text{ and }\eta<\la\rangle
\end{equation*}
is a partition of $\cs Hc;c$, and therefore a partition of $c$, by the definition of $\cs Hc$ as the left stabilizer of $c$.  In particular, the elements in this system are
 non-zero and pairwise disjoint.

Monotony   implies that
\[(\lax\cap\lbe);(\al\cdot\bl)\le\lax;a\quad\text{and}\quad(\lax\cap\lbe);(\al\cdot
\bl)\le\lbe;b\comma\] and therefore
\begin{equation*}(\lax\cap\lbe);(\al\cdot \bl)\le(\lax;a)\cdot(\lbe;b)\comma
\tag{19}\labelp{Eq:PL/9}
\end{equation*} by Boolean algebra.  The elements on the left side of this inequality are left-regular and therefore non-zero,
so the products on the right must also be  non-zero.
It has already been shown that, on the basis of \refEq{PL/1}, the systems of left translations in \refEq{prodlemma.5} and in \refEq{prodlemma.6} are both
partitions of $c$. Combine this with the preceding observation, and use Boolean algebra, to conclude that the system
\begin{equation*}\tag{20}\labelp{Eq:prodlemma.19}
  \langle (\lax;a)\cdot(\lbe;b):\x<\ka\text{ and }\eta<\la\rangle
\end{equation*} is also a partition of $c$.  Summarizing,  \refEq{prodlemma.018} and \refEq{prodlemma.19}  are both partitions of $c$.  The inequality
in \refEq{PL/9} therefore implies that the two partitions  must coincide, so that equality holds in \refEq{PL/9}.  This completes the proof of (ii).
\end{proof}

\section{Regular elements with normal stabilizers}\labelp{S:sec5}

  Suppose  $x$ and $y$ are measurable atoms, and  $a\le \xoy$  a regular element with
normal stabilizers.
The assumption on the stabilizers  implies, in particular, that   it is possible to form the quotient groups $\G x/\L_\al$ and $\G y/\R_\al$\per  It turns out
that the element $a$ induces in a canonical fashion an isomorphism between these two quotients.  To prove
this, we begin with a lemma.

\begin{lm} \labelp{L:withnormalstab}
Let $\wx$ and $\wy$ be measurable atoms\po  If an element $a\le \xoy$ is regular  with normal
stabilizers\co then for every $f$ in $ G_\wx$ and every
$g $ in $ G_\wy$\comma \[
(f;a)\cdot (a;g)\not= 0
\qquad\text{if and only if}\qquad
f;a=a;g\per
\]
\end{lm}


\begin{proof} 
The implication from right to left is obvious, since translations of regular elements are
regular, by Translation \refL{left.reg}, and regular elements are never zero, by \refL{regnonzero}.  To derive the reverse implication,
assume that the hypotheses of the  lemma are satisfied, and
suppose that
\begin{equation*} 
(f;a)\cdot (a;g)\not=0\per
\end{equation*}
Both $f;a$ and $a;g$ are  regular elements below $\xoy$, by  Translation \refL{left.reg} and its right-regular version, so
the product
\begin{equation*}\tag{2}\labelp{Eq:withnormalstab.2}
  c=(f;a)\cdot (a;g)
\end{equation*}
  must be a regular element below $\xoy$, by
 the  Product \refL{a-b.leftregular} and its right-regular version.  Moreover, this product has the same left and right stabilizers as
$a$, because
\begin{equation*}\tag{3}\labelp{Eq:withnormalstab.3}
\cs Hc= \L_{f;a}\cap \L_{a;g}
= (f;\L_\al; f\ssm)\cap \L_\al
= \L_\al\comma
\end{equation*}
by  \refEq{withnormalstab.2}, Product \refL{a-b.leftregular}, Translation \refL{left.reg}, and the assumption that $\L_\al$ is normal, and similarly,
\begin{equation*}
\cs Kc= \R_{f;a}\cap \R_{a;g}
= \R_\al \cap (g\ssm; \R_\al;g)
= \R_\al\comma
\end{equation*}
by the right-regular versions of Lemmas \ref{L:a-b.leftregular} and  \ref{L:left.reg}.

Use the assumption that $\cs Ha$ is normal,  the regularity of $c$, \refEq{withnormalstab.3}, and the implication from (i) to (iv) in  Translation \refL{assertioneq} (with $c$ in place of $b$) to get that $c$ is a left translation of $a$, in symbols,
\begin{equation*}
c=(f;a)\cdot (a;g)= h;a\tag*{(4)} \labelp{Eq:P16/10}
\end{equation*}
for some $h$ in $G_\wx$\per
In particular, $h;a\le f;a$.  Any two left translations of $a$ are equal or disjoint, by
\refL{equiv}(iii), so
$h;a=f;a$. It follows from this equation, \ref{Eq:P16/10}, and Boolean algebra  that
$
f;a\le a;g\per
$
A dual argument using
right-regularity establishes the reverse inequality.
\end{proof}

\begin{cor} \labelp{C:eqtrans}
Let $\wx$ and $\wy$ be measurable atoms\per If  $a\le \xoy$ is a regular element
with normal stabilizers\comma then every right translation of $a$ is also a left
translation\co and conversely\po
\end{cor}
\begin{proof}If   an element $a$ satisfies the hypotheses of the corollary,
then $a$ and all of its translations are regular elements below $\xoy$, by Translation \refL{left.reg} and its right-regular version,
and in particular they are not $0$, by
\refL{regnonzero}.  Let
$g$ be any element in
$\G\wy$\per  Since the left translations  $f;a$, for $f$ in $\G\wx$\comma sum
to $\xoy$, by   Partition \refL{PartitionLemma}, there must be an element $f$ in $\G\wx$ such that
\[(f;a)\cdot (a;g)\neq 0\per
\] Consequently,
\[(f;a)= (a;g)\comma
\]
by \refL{withnormalstab}.  In other words, the right translation $a;g$ can be written
as a left  translation $f;a$.  The converse is proved in a similar way.
\end{proof}

We turn now to the task of constructing a   group triple in an atomic, measurable relation algebra, and verifying the  semi-frame conditions for this triple.   The definitions of these notions, which are from \cite{andgiv1}, are not needed in this section, but they are needed in the next section. We give them here in order to motivate the  discussion.  The reader may choose to ignore them for now and  refer back to them at the appropriate moment in the next section.  A \emph{group triple}
\[\mc F=\trip G \vph
C\] consists of a system
\[G=\langle
\G x:x\in I\,\rangle\]  of disjoint groups, a system
\[\varphi=\langle\vph_{xy}:\pair x y\in \mc E\,\rangle\]  of associated
quotient isomorphisms, with $\vph_{xy}$ mapping a quotient group $\cs Gx/\cs H{xy}$ to a quotient group $\cs Gy/\cs K{xy}$ for each pair $\pair xy$ in
a fixed equivalence relation $\mc E$ on the group index set $I$, and finally
a system
\[C=\langle \cc x y z:\trip x y z\in\ez 3\rangle
\] of associated cosets, with   $\cc x y z$  a coset of the normal subgroup
$\hh$ in $\G\wx$ for each triple $\trip xy z$ in the set $\ez 3$ of triples $\trip xyz$ such that $\pair xy$ and $\pair yz$ are both in $\mc E$.
\begin{df}\labelp{D:semiframedef} A group triple
\[\mc F=\trip G \vph C
\]
is a (\textit{coset}) \textit{semi-frame} if  the following
\textit{semi-frame conditions} are satisfied\per
\begin{enumerate}
\item[(i)]
$\vphi {xx}$ is the identity automorphism of $\G x/\{\e x\}$ for
all $x$ in $I$\per
\item [(ii)]
$\vphi \yx=\vphi \xy\mo$ whenever $ \pair x y$ is in $\mc E$\per
\item[(iii)]  $\vphi \xy[\h
\xy\scir\h \xz]=\k \xy\scir\h \yz$ whenever $\trip x y z$ is in
$\ez 3$\per
\item[(iv)] $\hvphs \xy\rp\hvphs \yz=\tau\rp \hvphs
\xz$ whenever $\trip x y z$ is in $\ez 3$\comma
\end{enumerate}
where  $\hvphs \xy$\comma $\hvphs \yz$\comma and $\hvphs
\xz$ are the quotient isomorphisms induced on \[\cs Gx/(\cs H\xy\scir\cs H\xz)\comma \qquad \cs Gy/(\cs K\xy\scir\cs H\yz)\comma\cs Gx/(\cs H\xy\scir\cs H\xz)\] by $\vphi\xy$\comma $\vphi\yz$\comma  $\vphi\xz$ respectively\comma and $\tau$ is the inner automorphism of $\cs Gx/(\cs H\xy\scir\cs H\xz)$ determined by the coset $\cc x
y z$\per\qed\end{df}

  The global assumption of measurability is not needed in order to accomplish most of the task of constructing a group triple and verifying the semi-frame conditions, so for now we continue  with the local assumptions that $x$ and $y$ are measurable atoms, and $a\le \xoy$ is a  regular element with normal stabilizers.

\begin{theorem}[First Isomorphism Theorem] \labelp{T:isom}
Let   $\wx$ and $\wy$ be measurable atoms\per If   $a\le \xoy$ is  a regular element
with normal stabilizers\comma and in particular\comma if $a$ is an atom\comma then the quotient groups $\gxla$ and $\gyra$ are isomorphic\po In fact\co if
\[\langle \L_\x :
\x<\kp\rangle\qquad\text{and}\qquad
\langle \R_\eta : \eta<\la\rangle\] are  coset
systems of $\L_\al$ in $\cs Gx$ and $\R_\al$ in $\cs Gy$ respectively\co then there is a unique bijection
$\varphi$ from $\kp$ to $\la$ such that
\[
\L_\x;a=a;\R_{\varphi(\x)}\comma \]
and the mapping $\L_\x\longmapsto \R_{\varphi(\x)}$ is the desired
isomorphism\po
\end{theorem}

\begin{proof} 
Consider a regular element $a\le \xoy$    with normal stabilizers.  The left
and   right translations of $a$ are just the elements in the systems
 \[\langle \L_\x;a:\x<\kp\rangle\qquad\text{and}\qquad\langle
a;\R_\eta:\eta<\lmbd\rangle \] respectively.  For any fixed $\x<\kp$,
there is a $\x'<\lmbd$ such that
\begin{equation*}
\L_\x;a= a;\R_{\x'}\comma\tag*{(1)} \labelp{Eq:P16/11}
\end{equation*}
by \refC{eqtrans}.  If $\x''<\lmbd$ is any other index such that \[
\L_\x;a= a;\R_{\x''}\comma
\]
then  \[a;\R_{\x'}=a;\R_{\x''}\comma \] and therefore $\xi'=\xi''$,
because distinct  cosets of $\R_\al$ lead to disjoint right translations of $a$, by the right-regular version of
\refL{equiv}(iii).  Thus, there is a unique $\xi'<\lmbd$ such that \ref{Eq:P16/11} holds. Define
$\varphi(\x)$ to be
$\x'$.  The preceding remarks imply that $\varphi$ is a well-defined mapping from $\kp$ into
$\lmbd$, and
\begin{equation*} 
\L_\x;a=a;\R_{\varphi(\x)}  \tag*{(2)} \labelp{Eq:P19/2}
\end{equation*}
for each $\xi<\ka$, by \ref{Eq:P16/11} and the definition of $\vp$.

In a completely analogous fashion, define a mapping $\ps$ from
$\lmbd$ into $\kp$ such that
\begin{equation*}\tag{3}\labelp{Eq:isom.1}
  a;\R_\eta= \L_{\ps(\eta)};a
\end{equation*}
 for each $
\eta <\lmbd$.
In particular,
\begin{equation*}\tag{4}\labelp{Eq:isom.4}
\L_\x;a = a;\R_{\varphi(\x)}= \L_{\ps(\varphi(\x))};a\comma
\end{equation*}
  by \ref{Eq:P19/2} and \refEq{isom.1} (with $\vp(\xi)$ in place of $\eta$).
Distinct  cosets of $\L_\al$ lead to disjoint left translations of $a$, by Partiton \refL{PartitionLemma}, so
$\ps(\varphi(\x))=\x$ for each
$\x <\kp$, by \refEq{isom.4}. A symmetric argument shows that $\varphi(\ps(\eta))=\eta$
for each $\eta<\lmbd$.
It follows that $\varphi$ must a bijection from $\kp$ to $\lmbd$, with $\ps$ as its
inverse. Consequently, the correspondence
\begin{equation*} 
\L_\x \longmapsto \R_{\varphi(\x)} \tag*{(5)} \labelp{Eq:P19/3}
\end{equation*}
is a bijection from $\gxla$ to $\gyra$\per

Suppose that
\begin{equation*} 
\L_\x;\L_\eta=\L_\mu\per \tag*{(6)} \labelp{Eq:P19/4}
\end{equation*}
Use \ref{Eq:P19/2}, \ref{Eq:P19/4}, and two more applications of \ref{Eq:P19/2} to arrive at
\begin{equation*}\tag{7}\labelp{Eq:isom.6}
a;\R_{\varphi(\mu)}  = \L_\mu;a
= \L_\x;\L_\eta;a
= \L_\x;a ;\R_{\varphi(\eta)}
= a;\R_{\varphi(\x)}; \R_{\varphi(\eta)}\per
\end{equation*}
Distinct cosets of $\R_\al$ lead to distinct right
translations of $a$, by the right-regular version of \refL{PartitionLemma}, so   \refEq{isom.6} implies that
\[
\R_{\varphi(\x)}; \R_{\varphi(\eta)}= \R_{\varphi(\mu)}\comma
\]
This argument shows that  the mapping \ref{Eq:P19/3} preserves the quotient group composition operation of forming the relative product of two cosets.

In groups, the identity element and the   operation of forming inverses are both definable in terms of
the group composition operation, so any bijection that preserves composition is automatically a group isomorphism.  In particular, it follows from the preceding observations that
the mapping  \ref{Eq:P19/3} must be  an isomorphism between the quotient groups.
\end{proof}

\begin{con} \labelp{Co:canonical}
It is clear from the First Isomorphism Theorem that we may assume
$\kp=\la$, and we may reindex the coset system $\langle
\R_\x:\x<\kp\rangle$ so that $\varphi(\x)=\x$\comma that is to say, so that
\begin{equation*}\tag{1}\labelp{Eq:canonical.1}
  \L_\x\longmapsto \R_\x
\end{equation*}
is the canonical isomorphism from $\gxla$ to $\gyra$
determined by the equation
\begin{equation*} 
\L_\x;a =a;\R_\x\per \tag*{(2)} \labelp{Eq:conv2/1}
\end{equation*}
Furthermore, it may be supposed that $\L_0=\L_\al$\comma and consequently that $\R_0=\R_\al$   (where $0$ denotes an index zero, and not the  zero element in the relation algebra $\f A$ that has been fixed
throughout the  discussion.)  In the subsequent development,  these
notational conventions shall be adopted. When an explicit reference to $a$ is required, we   write $\L_{a,\x}$ and
$\R_{a,\x}$ for $\L_\x$ and $\R_\x$ respectively, and we  denote the canonical isomorphism \refEq{canonical.1} by
$\varphi_\al$, and call it the   \emph{isomorphism determined by}~$a$.\qedcon
\end{con}

The particular isomorphism described in Isomorphism \refT{isom} is determined by the
  element $a$.  Different translations
of $a$ may determine different isomorphisms. How are these various isomorphisms
 related to one other?
Assume  that $b$ is, for example, a left translation of $a$, say $b=\cs H\eta;a$.  In this case, $b$ is also a regular element with the same normal stabilizer
as $a$, by Translation \refL{left.reg} and its right-regular version.  Moreover, if
 $\varphi_\al$ maps $\L_\x$ to $\R_\x$ for each $\x<\kp$, as in \conref{Co:canonical}, or  in
different words, if   equation   \ref{Eq:conv2/1} of the convention holds
for each $\xi<\ka$, then
\begin{equation*}
b=\L_\eta;a=a;\R_\eta\comma\tag*{(3)} \labelp{Eq:auto3/2}
\end{equation*}
and therefore
\begin{multline*}\tag{4}\labelp{Eq:chariso.3}
\L_\x;b = \L_\x;\L_\eta;a= \L_\x;a;\R_\eta=a; \R_\x;\R_\eta\\=a; \R_\eta;\R\ssm_\eta;\R_\x;\R_\eta=b;\R\ssm_\eta;\R_\x;\R_\eta\comma
\end{multline*}
by \ref{Eq:auto3/2} and \ref{Eq:conv2/1}.
The fourth equality uses quotient group identity and inverse properties:
\begin{equation*}\k\xi=K_\al;\k\xi= \R_\eta;\R\ssm_\eta;\R_\x\per
\end{equation*}
It   follows from \refEq{chariso.3} and the definition of the isomorphism determined by $b$
 that
\begin{equation*}\tag{5}\labelp{Eq:auto3/3}
\varphi_\bl(\L_\x)=\R\ssm_\eta;\R_\x;\R_\eta
\end{equation*}
for each $\x$.

If $\tau$ is the inner automorphism of $\G\wx/\h\al$ determined by
\begin{equation*}\tag{6}\labelp{Eq:auto3/4}
\tau(\h\x)=\h\eta\ssm;\h\x;\h\eta
\end{equation*}
for each $\x$, then
\begin{equation*}\tag{7}\labelp{Eq:auto3/5}
\vphi\bl=\tau\rp\vphi\al\per
\end{equation*}
The verification of \refEq{auto3/5} is an easy computation:
\begin{multline*}
(\tau\rp\vphi\al)(\h\x)=\vphi\al(\tau(\h\x))=\vphi\al(\h\eta\ssm;\h\x;\h\eta)=\vphi\al(\h\eta)\ssm;\vphi\al(\h\x);\vphi\al(\h\eta)\\
=\k\eta\ssm;\k\x;\k\eta=\vphi b(\cs H\xi)\comma
\end{multline*}
by the definition of
relational composition, \refEq{auto3/4}, the isomorphism properties of $\vphi\al$\comma the  definition of $\vphi\al$\comma and \refEq{auto3/3}\per
The
isomorphism $\varphi_\bl$   agrees with the isomorphism $\varphi_\al$ on a coset $\h\x$
just in case the images of $\h\xi$ under these two mappings are equal, that is to say, just in case
\begin{equation*}\R_\x =\R\ssm_\eta;\R_\x;\R_\eta\comma
\end{equation*}
or, equivalently, just in case
\begin{equation*}\L_\x =\L\ssm_\eta;\L_\x;\L_\eta\per
\end{equation*}
This is the same as saying that $\L_\eta$ commutes with $\h\x$\per

The preceding argument proves the following theorem.

\begin{theorem} [Second Isomorphism Theorem]\labelp{T:chariso}
Let  $\wx$ and $\wy$ be measurable atoms\comma and    $a\le \xoy$   a regular element with
normal stabilizers\po
  If $b$ is the left translation of
$a$ by a coset
$\L_\eta$ of $\h a$\comma  then $\varphi_\bl$ is just the relational composition
\[\vphi b = \tau\rp\vphi a\comma
\] where $\tau$ is
 the inner automorphism  of  $\gxla$
determined by
\[\tau(\L_\x)=\L\ssm_\eta;\L_\x;\L_\eta\per \]
 The isomorphism
$\varphi_\bl$ agrees with $\varphi_\al$ on precisely those cosets of $\h\al$ that commute
with $\h\eta$\po
\end{theorem}

The theorem allows us to characterize when the isomorphisms determined by a left translation $\bl$
of a regular element $\al$ coincides with the isomorphism $\vphi\al$\per

\begin{cor}
\labelp{C:charagree}
 Under the hypotheses of the First Isomorphism
Theorem\co
  if $b$ is the left translation of
$a$ by a coset
$\L_\eta$ of $\h a$\comma  then \[\varphi_\bl=\vphi\al\] just in case $\h\eta$ is
in the center of $\gxla$\per
\end{cor}

\comment{
This corollary allows one to count the number of distinct isomorphism
determined by translations of a regular element.  Recall that the \textit{center} of $\cs Gx/\cs Ha$, that is to say, the
collection $Z$ of cosets of $\h a$ that commute with every coset of $\cs Ha$, is a subgroup of
$\G\wx/\h a$\per  The
\textit{index} of the center is the number of distinct cosets that the subgroup $Z$  has in $\G\wx/\h a$.
Put another way, it is the cardinality of the quotient group of $\G\wx$ modulo the normal
subgroup
$\bigcup Z$, by the Third Isomorphism Theorem from group theory.  In more detail, $Z=(\tbigcup Z)/\cs Ha$, and the
Third Isomorphism Theorem implies that $(\cs Gx/\cs Ha)/((\tbigcup Z)/\cs Ha)$ is isomorphic to $\cs Gx/(\tbigcup Z)$.

\begin{cor}\labelp{C:oneisom1} Under the hypotheses of the First Isomorphism
Theorem\co   the number of distinct
 isomorphisms from $\gxla$ to $\gyra$ that are determined by  translations of $a$ coincides with
the index of the center of $\gxla$\per
\end{cor}
\begin{proof} Let $a\le \xoy$ be a regular element with normal stabilizers. Every trans\-la\-tion of $a$ may be written
as a left translation of $a$, by \refC{eqtrans}, so we may restrict our attention to left translations of $a$. Let   $b $  and
$c$ be such translations,  say by  cosets $
\L_\eta$\comma and    $\h\rho$ respectively\per   Isomorphism \refT{chariso} implies that
\begin{equation*}\vphi b = \tau\rp\vphi
a\qquad\text{and}\qquad\vphi c = \sigma\rp\vphi
a\comma\tag{1}\labelp{Eq:oneisom1.1}\end{equation*} where $\tau$ and $\sigma$ are the
inner automorphisms of $\G\wx/\L_a$ determined by $\L_\eta$ and $\h\rho$ respectively.
From \refEq{oneisom1.1},  it follows that
\begin{alignat*}{2}\tag{2}\labelp{Eq:oneisom1.2}
\vphi b=\vphi c\qquad&\text{if and only if}&\qquad\tau\rp\vphi a&=\sigma\rp\vphi
a\comma\\ &\text{if and only if}&\qquad\sigma\inv\rp\tau\rp\vphi a&=\vphi a\per
\end{alignat*}
The mapping $\sigma\inv\rp\tau$ is the inner automorphism of $\gxla$ determined by
the coset $\h\rho\ssm;\h\eta$\per Consequently, the final
equation in \refEq{oneisom1.2}  holds just in case
\[\vphi a ((\cs H\rho\ssm ;\cs H\eta)\ssm;\cs H\x;(\cs H\rho\ssm;\cs H\eta))=\vphi a(\cs H\xi)\] for all $\xi$, or equivalently, since $\vphi a$ is a bijection,
just in case
\[(\cs H\rho\ssm ;\cs H\eta)\ssm;\cs H\x;(\cs H\rho\ssm;\cs H\eta)=\cs H\xi \] for all $\xi$.  This last equation is equivalent to saying
that the coset $\h\rho\ssm;\h\eta$ belongs to the
center $Z$ of $\gxla$\comma or, what amounts to the same thing, that
\begin{equation*}\tag{3}\labelp{Eq:oneisom1.3}
\h\rho\ssm;\h\eta;Z=
Z\comma
\end{equation*}
by \refC{charagree}.
Of course, \refEq{oneisom1.3} can be rewritten as
\begin{equation*}\tag{4}\labelp{Eq:oneisom1.4}\h\eta;Z=\h\rh;
Z\per
\end{equation*}
But \refEq{oneisom1.4} just means that the cosets $\h\eta$ and $\h\rho$ of $\h \al$ in $\cs Gx$ belong to
the same coset of $Z$ in $\cs Gx/\cs Ha$.  Combine these observations to conclude that the isomorphisms
$\vphi b$ and $\vphi c$ coincide just in case   $\h\eta$ and $\h\rho$ belong to the same coset of the center $Z$, so the number of different isomorphisms
induced by translations
is the same as the number of different cosets of
$Z$. This is just the definition of the index of $Z$.
 \end{proof}

Another consequence of Isomorphism \refT{chariso} is a characterization of the cosets
on which  \textit{all} of the mappings $\vphi b$ agree, where $b$  ranges over the
translations of $\al$.

\begin{cor} \labelp{C:allagree}
 Under the hypotheses of the First Isomorphism
Theorem\comma
  the  cosets of $\h\al$ on which all isomorphisms induced by  translations of $\al$
agree  are precisely the cosets in the center of $\gxla$\po
\end{cor}
\begin{proof} An element $\bl$ is a translation of $\al$ just in case it is a left translation
of $\al$, by \refC{eqtrans}.  Consequently, it may be assumed that
$b=\h\eta;a$ for some coset $\h\eta$ of $\h a$\per  Let $\tau$ be the inner automorphism
of $\G\wx/\h a$ determined by $\h\eta$\comma so that
\begin{equation*}\tag{1}\labelp{Eq:allagree1}
\vphi b=\tau\rp\vphi a\comma
\end{equation*}
 by   Isomorphism \refT{chariso}.
Every inner automorphism of a group leaves the elements of the center fixed, so $\vphi b$ agrees with  $\vphi a$ on
all cosets that belong to the center of $\cs Gz/\cs Ha$, by \refEq{allagree1}.

 Suppose now that all  isomorphisms determined by left translations  of
$\al$ agree on a given coset $\h\x$ of $\cs Ha$\comma with the goal of showing that $\cs H\x$ must belong to the center of $\cs Gx/\cs Ha$\per
Let  $\h\eta$ be any coset of $\h\al$\comma and write $\bl=\h\eta;\al$\per
The equation \refEq{allagree1} must hold, by Isomorphism \refT{chariso}, and the isomorphism
$\vphi\bl$ is assumed to agree with $\vphi\al$  on $\h\xi$\comma  so
\begin{equation*}\tag{2}\labelp{Eq:allagree.1}
  \vphi a(\h\eta\ssm;\h\x;\h\eta)=\vphi a(\tau(\h\x))=\vphi b (\h\x)=\vphi a(\h\xi)\comma
\end{equation*}
by \refEq{allagree1}, the definition of $\tau$, and the assumption on $\vphi b$.  Since $\vphi a$ is a bijection, it follows from \refEq{allagree.1} that
\[\h\eta\ssm;\h\x;\h\eta=\h\xi\per
\]
This equation, of course, implies that the
the cosets
$\h\x$ and
$\h\eta$  commute.  The   coset $\cs H\eta$ is assumed to be arbitrary, so  $\cs H\x$ commutes with every coset of $\cs Ha$, and therefore
$\h\xi$   belongs to the center of $\gxla$\per
\end{proof}

The difference between \refC{charagree} and \refC{allagree} is that the first corollary characterizes
those translations $b$ of $a$ for which the isomorphisms $\vphi b$ and $\vphi a$ are identical, while the second
characterizes the cosets on which all of isomorphisms $\vphi b$ determined by translations $b$ of $a$ assume
  the same value,  irrespective of whether these isomorphisms coincide with $\vphi a$ or not.

In case the
quotient group $\cs Gx/\cs Ha$ is abelian, every coset of
$\L_\al$ belongs to the center, so all of the isomorphisms determined by   translations of a
regular element $a$ coincide.}

\begin{cor}\labelp{C:oneisom2} Under the hypotheses of the First Isomorphism
Theorem\co  if
$\gxla$ is abelian\co then $\varphi_\bl=\varphi_\al$ for every translation $b$ of $a$\po
\end{cor}

We now explore some of the important consequences of the   Isomorphism Theorems.
 Let $\wx$ and $\wy$ be measurable atoms.
A regular element
 $\al\le \xoy$ with normal stabilizers determines an isomorphism
$\vphs a$ from $\gxla$ to  $\gyra$, by Isomorphism \refT{isom}\per  Another regular element
$\bl\le\xoy$ with normal stabilizers (with $b$ not necessarily a translation of $a$) determines  an isomorphism
$\vphs b$ from $\gxlb$ to  $\gyrb$\per
What is the relationship between these two isomorphisms when
$\al\le\bl$?
To answer this question, we  use the notation and conventions
discussed in
\refCo{canonical}.  In particular,\begin{align*}
                                  \langle\lax:\x<\ka\rangle\qquad&\text{and}
\qquad \langle\lbe:\eta<\la\rangle \\
\intertext{are assumed to be coset systems for $\lal$ and $\lbl$ respectively in $\G x$\comma
and}
\langle\rax:\x<\ka\rangle\qquad&\text{and}
\qquad \langle\rbe:\eta<\la\rangle
                                  \end{align*}
are assumed to be the corresponding coset systems for $\rar$ and $\rbr$ respectively in $\G y$\per

If $\al\le\bl$,  then \begin{align*}\lal\seq\lbl\qquad &\text{and}\qquad \rar\seq\rbr\comma\tag{1}\labelp{Eq:th.what.1}\\
\intertext{by
\refL{lessthan} and its right-regular version, so there are partitions
$\langle\Gamma_\eta:\eta<\la\rangle$ and $\langle\Delta_\eta:\eta<\la\rangle$
of $\ka$ such that}
\lbe=\tbigcup\{\lax:\x\in
\Gamma_\eta\}\qquad&\text{and}\qquad\rbe=\tbigcup\{\rax:\x\in
\Delta_\eta\}\tag{2}\labelp{Eq:RT/1}\\
\intertext{for each $\eta<\la$.  Apply  Partition \refL{GeneralizedPartitionLemma} and its right-regular version to see that}
\langle \lax;\al:\xi\in\Gamma_\eta\rangle\qquad&\text{and}\qquad \langle \al;\rax: \x\in\Delta_\eta\rangle\tag{3}\labelp{Eq:th.what.2}\\
\intertext{partition
$\lbe;\bl$ and
$\bl;\rbe$ respectively.  For each $\eta<\la$ and  $\x<\ka$,}
\lbe;\bl=\bl;\rbe\qquad&\text{and}\qquad\lax;\al=\al;\rax\comma\tag{4}\labelp{Eq:th.what.3}
\end{align*} by \refCo{canonical}. Use   the partition properties of \refEq{th.what.2}, together with \refEq{th.what.3}, to obtain
\[\bl;\rbe=\lbe;\bl=\tsum\{\lax;\al:\xi\in\Gamma_\eta\}=\tsum\{\al;\rax: \x\in\Gamma_\eta\}\per\]
It  follows  from this computation that the
translations $\al;\rax$   with
indices $\x$ in $\Gamma_\eta$  partition
$\bl;\rbe$\per Compare this with the second part of \refEq{th.what.2}, and use the fact that the right translations of $\al$ by distinct cosets of $\rar$ are pairwise disjoint,
by the right-regular version of  Partition \refL{PartitionLemma}, to conclude that  $\Gamma_\eta=\Delta_\eta$.
Consequently,
\begin{equation*}\rbe=\tbigcup\{\rax:\x\in
\Gamma_\eta\}\comma\tag{5}\labelp{Eq:RT/5}
\end{equation*} by \refEq{RT/1}.

In view of \refEq{th.what.1} and \refEq{RT/1},   the isomorphism $\vphs a$ from $\gxla$ to $\cs Gy/\cs Ka$ induces
in a natural way an  isomorphism $\hat\vp_a$ from $\gxlb$ to $\gyrb$ that is defined by
\[\hvphs a(\lbe)=\tbigcup\{\vphs a(\lax):\x\in
\Gamma_\eta\}\] for each $\eta$.
It is easy to check that this induced isomorphism
  coincides with $\vphs b$, because\begin{equation*}\hvphs a(\lbe)=\tbigcup\{\vphs
a(\lax):\x\in\Gamma_\eta\}=\tbigcup\{\rax:\x\in\Gamma_\eta\}=\rbe =\vphs b(\lbe)\comma
\end{equation*}
by the definition of $\hvphs a$, the definition of $\vphs a$ and \refCo{canonical},  \refEq{RT/5}, and the definition of $\vphs b$ and \refCo{canonical} (with $b$  in place of $a$)\per
The following theorem has been proved.

\begin{theorem}[Refinement Theorem] \labelp{T:th.what}
Let   $\wx$ and $\wy$ be measurable atoms\comma and   $\al$ and $\bl$ regular elements below $\xoy$
with normal stabilizers\po  If $\al\le\bl$\co then $\cs Ha$ and $\cs Ka$ are subgroups of $\cs Hb$ and $\cs Kb$ respectively\comma and the isomorphism $\vphi a$ from $\cs Gx/\cs Ha$ to $\cs Gy/\cs Ka$ induces an isomorphism
$\hvphs a$ from $\gxlb$ to $\gyrb$\comma that  coincides with $\vphs b$\po
\end{theorem}

 A measurable atom $x$ is certainly a regular element with normal stabilizers, by Corollaries \ref{C:atoms} and \ref{C:a-b.atoms} (with $y=x$).
By definition, the   left stabilizer $\cs Hx$ consists of those elements $f$ in $\cs Gx$ such that $f;x=x$.  Since $f;x=f$ for all $f$ in $\cs Gx$, by \refL{nonzerofunct}, it follows that \[f\in \cs Hx\qquad\text{if and only if}\qquad f=x\comma\] and therefore $\cs Hx=\{x\}$.  Similarly, $\cs Kx=\{x\}$.  The cosets of these stabilizers are the singletons $\{f\}$ for $f$ in $\cs Gx$.  The quotient isomorphism $\vphs\wx$ maps $\cs Gx/\cs Hx$ to $\cs Gx/\cs Kx$, and takes $\{f\}$ to $\{g\}$ if and only if $f;x=x;g$,  that is to say, if and only if $f=g$, by \refL{nonzerofunct}.  Thus, $\vphs x$ is the identity automorphism of $\cs Gx/\{x\}$. This argument proves the following theorem.

\begin{theorem}[Identity Theorem] \labelp{T:identthm}
  A measurable atom $x$   is a regular element with  left and right stabilizers  both equal to  the trivial normal subgroup $\{\wx\}$\po  The quotient isomorphism $\vphs
\wx$ is the identity automorphism of  $\G \wx/\{\wx\}$\po
\end{theorem}

The next theorem says that if $\al\le \xoy$ is a regular element with normal
stabilizers, then so is
$\al\ssm$, and $\vphs{\al\ssm}$ is just the inverse isomorphism of $\vphs a$\per

\begin{theorem}[Converse Theorem] \labelp{T:convthm}
Let   $\wx$ and $\wy$ be measurable atoms\per If $\al$ is a regular element bekiw $\xoy$    with normal
stabilizers\co and in particular\comma if $a$ is an atom\comma then $\al\ssm$ is a regular element below $\yox$ with normal stabilizers
\[\LS{\al\ssm}=\rar \qquad\text{and}\qquad
\RS{\al\ssm}=\lal\comma\]
and the  quotient isomorphism $\vphs{\al\ssm}=\vph\mo_a$\po
\end{theorem}
\begin{proof}  Assume the hypotheses of the theorem,   write $b=\al\ssm$, and observe that
 \begin{equation*}\tag{1}\labelp{Eq:convthm.0}
  b=\al\ssm\le (\xoy)\ssm=\yox\comma
 \end{equation*}  by the definition of $b$, the assumption that $a\le \xoy$, monotony, and \refL{square}(iii). For each element $f$ in $\cs Gy$,
\begin{align*}
f\in\LS{b}\qquad&\text{if and only if}\qquad f;b=b\comma\\
&\text{if and only if}\qquad f;\al\ssm=\al\ssm\comma\\
&\text{if and only if}\qquad \al;f\ssm=\al\comma\\
&\text{if and only if}\qquad f\ssm\in\RS{\al}\comma\\
&\text{if and only if}\qquad f\in\RS{\al}\comma
\end{align*}
by the definition
of $\LS{b}$, the definition of $b$,
 the two involution laws (R6) and (R7), the definition of $\cs Ka$, and the fact that $\cs Ka$ is a subgroup of $\cs Gy$ and therefore closed under converse. An analogous argument
 applies to $f$ in $\cs Gx$, with $\cs Kb$ and $\cs Ha$ in place of $\cs Hb$ and $\cs Ka$ respectively, so that
 \begin{equation*}\tag{2}\labelp{Eq:convthm.1}
   \cs Hb=\cs Ka\qquad\text{and}\qquad \cs Kb=\cs Ha\per
 \end{equation*}
Similarly, for each element $f$ in $\cs Gy$,
\begin{align*}
  f\in\cs Xb\qquad&\text{if and only if}\qquad f\le b;b\ssm\comma \\
  &\text{if and only if}\qquad f\le a\ssm;a\comma\\
  &\text{if and only if}\qquad f\in\cs Ya\comma
\end{align*}
by the definitions of the set $\cs Xb$, the element $b$, and the set $\cs Ya$. An analogous argument applies to $f$ in $\cs Gx$, with $\cs Yb$ and $\cs Xa$ in place of $\cs Xb$ and $\cs Ya$ respectively, so that
\begin{align*}
   \cs Xb=\cs Ya\qquad&\text{and}\qquad \cs Yb=\cs Xa\per\tag{3}\labelp{Eq:convthm.2}\\
   \intertext{Combine \refEq{convthm.1} with \refEq{convthm.2}, and use the assumed regularity of $a$, together with  \refC{lrreg}, to arrive at}
   \cs Hb=\cs Ka=\cs Ya=\cs Xb\qquad&\text{and}\qquad \cs Kb=\cs Ha=\cs Xa=\cs Yb\per\tag{4}\labelp{Eq:convthm.3}
 \end{align*} With the help of \refC{lrreg} (with $b$ in place of $a$), conclude from \refEq{convthm.0}, \refEq{convthm.3}, and \refEq{convthm.1} that $b$ is a regular element below $\yox$, with left and right stabilizers as in \refEq{convthm.1}.

To prove the final assertion of the theorem, consider cosets  $\LS \x$ and $\R_\eta$ of   $\lal$  and   $\rar$ respectively\per  The equations in \refEq{convthm.1} imply that  $\LS \x$ and $\R_\eta$ are also cosets of $\cs Kb$ and $\cs Hb$ respectively. Use the definition of $\vphs \al$ in \refCo{canonical}, the involution laws, the definition of $b$, the definition of $\vphs b$, and the isomorphism properties of $\vphs b$ to obtain
\begin{align*}
\vphs\al(\LS\x)=\RS\eta\qquad&\text{if and only if}\qquad\LS\x;\al = \al;\RS\eta\comma\\
&\text{if and only if}\qquad \al\ssm;\L\ssm_\x=\R\ssm_\eta;\al\ssm\comma\\
&\text{if and only if}\qquad b;\L\ssm_\x=\R\ssm_\eta;b\comma\\
&\text{if and only if}\qquad \vphs{b}(\R\ssm_\eta)=\L\ssm_\x\comma\\
&\text{if and only if}\qquad \vphs{b}(\R_\eta)=\L_\x\per
\end{align*}
The equivalence of the first and last equations  implies that the isomorphisms $\vphs{b}$ and $\vphs{\al}$ are inverses of one another\per
\end{proof}

\begin{con}\labelp{Co:conver}
Continue with the assumption that  $a\le\xoy$ is a regular element with normal stabilizers, and write $b=a\ssm$.    Converse \refT{convthm} implies that
\begin{align*}
\LS{\bl}=\rar \qquad&\text{and}\qquad
\RS{\bl}=\lal\per\\
\intertext{ It also implies that if}
                                  \langle\lax:\x<\ka\rangle\qquad&\text{and}
\qquad \langle\rax:\x<\ka\rangle \\
\intertext{are coset systems for  $\lal$ and $\cs K a$  in $\G x$ and $\cs Gy$ respectively such that \[\vphi a(\lax)=\rax\] for each $\xi$, then by putting}
\lbx=\rax\qquad&\text{and}\qquad \rbx=\lax\tag{1}\labelp{Eq:conver.1}\\
\intertext{for each $\xi$, we arrive at cosets systems}
\langle\lbx:\x<\ka\rangle\qquad&\text{and}
\qquad \langle\rbx:\xi<\ka\rangle
                                  \end{align*} for $\cs Hb$ and $\cs Kb$ in $\cs Gy$ and $\cs Gx$ respectively such that
                                  \[\vphi b(\lbx)=\rbx\] for each $\xi$.  In what follows, we shall always assume that the coset systems for
$\cs Hb$ and $\cs Kb$ have been chosen so that \refEq{conver.1} holds.\qedcon
\end{con}

 The next goal is to prove, for the relative product of two regular elements with
normal stabilizers, the analogue of Theorems \ref{T:identthm} and \ref{T:convthm}.
In more detail, suppose that
\[\al\le \xoy\qquad\text{and}\qquad \bl\le \yoz\] are regular elements with normal stabilizers. The right stabilizer $\rar$ of $\al$, and the left stabilizer $\lbl$ of $\bl$, are normal subgroups $\cs Gy$, and therefore so is
the product subgroup
$\rar;\lbl$, which includes both $\cs Ka$ and $\cs Hb$ as subgroups.  The quotient isomorphisms $\vph\inv_\al$ from $\cs Gy/\cs Ka$ to $\cs Gx/\cs Ha$, and   $\vphs \bl$ from $\cs Gy/\cs Hb$ to $\cs Gz/\cs Kb$  induce isomorphisms  $\hat\vph\inv_\al$ and $\hvphs \bl$ on the
quotient group
$\G\wy/(\rar;\lbl)$.   The   goal is to prove that $a;b$ is a regular element below $\xoz$ with normal stabiizers
\[\LS\abp=\hat\vph\inv_\a[\rar;\lbl]\qquad\text{and}\qquad
\RS\abp=\hvphs b[\rar;\lbl]\comma\]
and   the isomorphism $\vphs\abp$ from $\G\wx/\LS\abp$ to $\G\wz/\RS\abp$ coincides with the relational composition $\hvphs a\,\vert\,\hvphs b$ of the two
induced isomorphisms $\hvphs a$ and $\hvphs b$\per

We begin with a lemma.
\begin{lm} \labelp{L:alfabeta}
Let $\wx$\comma $\wy$\comma and $\wz$ be measurable atoms\po  If \[\al\le\xoy\qquad\text{and}\qquad \bl\le\yoz\] are
regular elements with normal stabilizers\comma then for  all elements $f$ in $G_\wx$ and  $h$ in
$ G_\wz$ the following conditions are equivalent\per
\begin{itemize}
\item[\opar i\cpar] $f;a;b=a;b;h$\per
\item[\opar ii\cpar] There are elements $g_1$ and $ g_2$ belonging to the same coset of
$\rar;\lbl$
in $\G\wy$ such that
\[f;a=a;g_1\qquad\text{and}\qquad g_2;b=b;h\per\]
\item[\opar iii\cpar] There is an element $g$ in $ G_\wy$ such that
\[f;a=a;g\qquad\text{and}\qquad g;b=b;h\per\]
\end{itemize}
\end{lm}

\begin{proof} 
Let $f$ and $h$ be elements of $ G_\wx$ and $G_\wz$ respectively. It is
easy to establish the implication from (iii) to (i): if
\[f;a=a;g\qquad\text{and}\qquad g;b=b;h\comma\]
then
\[
f;a;b=a;g;b= a;b;h\per
\]

Assume now that (i) holds, with the goal of establishing (ii).
The element $f$ is in
$G_\wx$\comma and therefore belongs to some coset $\L_\x$ of $\L_a$\per Apply Isomorphism \refT{isom}  and \refCo{canonical} to write
\begin{equation*}\tag{1}\labelp{Eq:alfabeta.1}
  \L_\x;a=a;\R_\x\per
\end{equation*}
Take
$g_1$ to be any element in $\RS\x$\comma and observe that  the first equation in
(ii)  holds, by \refEq{alfabeta.1} and the remark following \refC{stab}. An analogous
argument produces an element
$g_2$ in
$G_\wy$ such that the second equation in (ii) holds.  It remains to
verify that $g_1$ and $g_2$ belong to the same coset of $\rar;\lbl$.
First of all,
\begin{multline*}\tag{2}\labelp{Eq:alfabeta.2}
a\ssm;f;a;b;b\ssm = a\ssm;a;g_1;b;b\ssm\\ = (\ssum\R_a);g_1 ;
(\ssum\L_b)   = \ssum \R_a;g_1;\L_b\comma
\end{multline*} by the choice
of $g_1$, the assumed regularity of $\al$ and $\bl$,  and complete distributivity.  An analogous argument shows that
\begin{multline*}\tag{3}\labelp{Eq:alfabeta.3} a\ssm;a;b;h;b\ssm= a\ssm;a;g_2;b;b\ssm\\=
(\ssum\R_a);g_2;(\ssum\L_b) = \ssum\R_a;g_2;\L_b\per
\end{multline*}
The assumption in (i) implies that
\begin{equation*}\tag{4}\labelp{Eq:alfabeta.4}
  a\ssm;f;a;b;b\ssm=a\ssm;a;b;h;b\ssm\per
\end{equation*}
\[\]
Combine \refEq{alfabeta.2}--\refEq{alfabeta.4} to arrive at \[\ssum \R_a;g_1;\L_b= \ssum\R_a;g_2;\L_b\per
\]
Sums of sets of atoms can only be equal when the sets themselves are
equal, so the preceding equation implies that
\begin{equation*} 
\R_a;g_1;\L_b= \R_a;g_2;\L_b\per \tag*{(5)} \labelp{Eq:P22/1}
\end{equation*}
Use \ref{Eq:P22/1} and the assumption that the right stabilizer $\R_a$ is a normal subgroup to conclude that
\begin{equation*}\tag{6}\labelp{Eq:alfabeta.6}
  \cs g1;\cs Ka;\cs Hb=\R_a;g_1;\L_b= \R_a;g_2;\L_b=\cs g2;\cs Ka;\cs Hb\per
\end{equation*}
The equality of the first and last terms implies that $\cs g1$ and $\cs g2$ belong to the same coset of $\cs Ka;\cs Hb$.

Turn, finally, to the implication from (ii) to (iii). If
  $g_1$ and $g_2$ are in the same coset of $\rar;\lbl$, then by interchanging  the first two terms, and also the last two terms, in \refEq{alfabeta.6}, one checks that
\ref{Eq:P22/1} holds. The element $g_1$ belongs to the coset
$\R_a;g_1;\L_b$\comma because the identity element $y$ of $\cs Gy$ belongs to both  $\R_a$ and  $\L_b$\comma and
\[g_1 =\ygy\per
\] It follows from \ref{Eq:P22/1} that  $g_1$ belongs to $\R_a;g_2;\L_b$\comma so there must be elements  $g_a$ in $ \R_a$ and
$g_b$ in $\L_b$ such that
\[g_1=g_a;g_2;g_b\per\]
Put
\begin{equation*} 
g=g_a\ssm;g_1=g_2;g_b\per \tag*{(7)} \labelp{Eq:P22/2}
\end{equation*}
Use the assumption in (ii), the assumption that $\cs ga$, and hence also $\cs ga\ssm$, belongs to $\cs Ka$, the definition of $\cs Ka$ as the right stabilizer of $a$, and \ref{Eq:P22/2} to obtain
\begin{equation*}
f;a   = a;g_1
 = a;g_a\ssm;g_1
=a;g\per
\end{equation*}
Similarly, use (ii), the assumption that $\cs gb$ is in $\cs Hb$, the definition of $\cs Hb$ as the left stabilizer of $b$, and \ref{Eq:P22/2} to obtain
\begin{equation*}
b;h   = g_2;b
= g_2;g_b;b
=g;b
\end{equation*}
Thus, (iii) holds.
\end{proof}

   In the statement of the next theorem,
$\vphs a$ is assumed to be an isomorphism from $\cs Gx/\cs Ha$ to $\cs Gy/\cs Ka$ that induces an isomorphism $\hvphs a$ from $\cs Gx/\cs H{a;b}$ to $\cs Gy/(\cs Ka;\cs Hb)$, while $\vphs b$ is an isomorphism from $\cs Gy/\cs Hb$ to $\cs Gz/\cs Kb$   that induces an isomorphism $\hvphs b$ from $\cs Gy/(\cs Ka;\cs Hb)$
to $\cs Gz/\cs K{a;b}$.

\newcommand\abb{\al;\bl;b\ssm}
\renewcommand\bba{\bl;\bl\ssm;\al\ssm}
\newcommand\baa{\bl;\al\ssm;\al}

\begin{theorem}[First Relative Product
Theorem] \labelp{T:else}
Let $\wx$\comma $\wy$\comma and $\wz$ be measurable atoms\per If
\[a\le
\xoy\qquad\text{and}\qquad b\le
\yoz\] are regular elements with normal stabilizers\comma and in particular if they are atoms\comma
then the relative product $\al;\bl$ is a regular  element
 below $\xoz$ with  normal stabilizers
\[\LS\abp=\vph\inv_\al[\rar;\lbl]\qquad\text{and}\qquad\RS\abp
=\vphs \bl[\rar;\lbl]\comma\]
and with quotient isomorphism
\[\vphs\abp=\hvphs\al\,\vert\,\hvphs \bl \po\]
\end{theorem}

\begin{proof} 
The stabilizers $\L_a$ and $\R_b$ are assumed to be normal subgroups of
$G_\wy$\comma so the  product group $\R_a;\L_b$ is also a normal subgroup of
$G_\wy$. It
can of course be written as the union of distinct cosets $\RS\x$ of
$\rar$, say  with $\x<\ka$, so that
\begin{equation*} 
\R_a;\L_b=\tsbigcreg_{\x<\ka} \R_\x\per \tag{1} \labelp{Eq:P23/1}
\end{equation*}
Write
\begin{equation*} 
\L_\x= \vph_a\mo (\R_\x) \tag{2} \labelp{Eq:P23/2}
\end{equation*}
for $
\x<\ka$, and use the definition of $\vphs a$ and \refCo{canonical} to obtain
\begin{equation*} 
\L_\x;a=a;\R_\x
\tag{3} \labelp{Eq:P23/3}
\end{equation*}
for each $
\x$. The first goal is to prove that $a;b$ is left-regular, by showing that
\begin{equation*} 
X_{a;b}=\tsbigcreg_{\x<\ka} \L_\x=\L_{a;b}\per \tag{4}
\labelp{Eq:P23/4}
\end{equation*}

Since $\R_\x$ is a coset of $\R_a$ in $\R_a;\L_b$\comma there must be
an element $h_\x$ in $\L_b$ such that
\begin{equation*} 
\R_\x=\R_a;h_\x\per \tag{5} \labelp{Eq:P23/5}
\end{equation*} Use \refEq{P23/3}, \refEq{P23/5}, the definition of $\R_a$ as the right stabilizer of $a$\comma and the assumption that $h_\x $ belongs to the left stabilizer of $b$, to get
\begin{equation*}
\L_\x;a;b  =a;\R_\x;b
 =a;\R_a;h_\x;b
=a;h_\x;b
= a;b
\end{equation*}
This argument shows that $\cs H\xi$ is included in the left stabilizer $\cs H{a;b}$ for every $\x<\ka$, and therefore
\begin{equation*} 
\tsbigcreg_{\x<\ka} \L_\x\seq \L_{a;b}\per \tag{6}
\labelp{Eq:P23/6}
\end{equation*}
Compute:
\begin{align*} \tag{7}\labelp{Eq:else.7}
\ssum X_{a;b}  &= (a;b);(a;b)\ssm
 =a;b;b\ssm;a\ssm\\
&= a;(\ssum\L_b);a\ssm
  =a;\R_a; (\ssum\L_b);a\ssm\\
&=\ssum a;\R_a;\L_b;a\ssm=\ssum a;(\tbigcup_{\x<\ka}\cs K\x);a\ssm\\
& =\sumreg_{\x<\ka} a;\R_\x;a\ssm
 =\sumreg_{\x<\ka} \L_\x;a;a\ssm\\
&=\sumreg_{\x<\ka} \L_\x; (\ssum\L_a)
=\sumreg(\tsbigcreg_{\x<\ka} (\L_\x;\lal))
= \ssum (\tsbigcreg_{\x<\ka} \L_\x)\comma
\end{align*}
by the definition of $X_{a;b}$\comma the second involution law and the associative law, the left-regularity of $\bl$,  the definition of $\R_a$ as the right stabilizer of $a$\comma  complete distributivity, \refEq{P23/1}, complete distributivity, \refEq{P23/3}, the left-regularity of $\al$, complete distributivity,
and the fact that $\L_a$ is the identity coset of the quotient $\cs Gx/\cs Ha$.
Since  $X_{a;b}$ and $\tsbigcreg_{\x<\ka}\L_\x$ are
both sets of atoms, it follows from  \refEq{else.7} that these sets must be equal. Consequently,
\[
\tsbigcreg_{\x<\ka} \L_\x \seq \L_{a;b} \seq X_{a;b}=
\tsbigcreg_{\x<\ka}\L_\x
\]
by \refEq{P23/6},
\refL{leftcos}(ii),  and the remark following \refEq{else.7}.  The first and last terms are the same, so equality holds everywhere, which   proves \refEq{P23/4}.

Use  \refEq{P23/4} and \refC{lrreg} to conclude
that
the relative product
$\abp$ is left-regular.
 Also,
 \begin{equation*}\tag{8}\labelp{Eq:else.8}
   \L_{a;b}= \tsbigcreg_{\x<\ka} \L_\x =
\tsbigcreg_{\x<\ka} \vph_a\mo (\R_\x)=
\vph_a\mo [\,\tsbigcreg_{\x<\ka} \R_\x\,]=
\vph_a\mo [\R_a;\L_b]\comma
 \end{equation*}
by \refEq{P23/4}, \refEq{P23/2}, the preservation of unions under inverse images, and \refEq{P23/1}.

The group $\R_a;\L_b$ is  normal in $G_\wy$, so
the quotient $\ralbra$ is a normal subgroup of $\gyra$\per Because
$\vph_a$ maps $\gxla$ isomorphically to $\gyra$\comma the inverse image
of
$\ralbra$ under $\vphs a$, which is $\labla$, by \refEq{else.8} and the definition of $\vphs a$, must be a normal subgroup of $\gxla$\comma and therefore $\L_{a;b}$
must be a normal subgroup of $G_\wx$, by group theory.

It has been shown that the relative product $\abp$ is left-regular  and   that its left
stabilizer
$\L_{a;b}$ is  normal  in $\G \wx$ and  coincides
with
$\vph_a\mo [\R_a;\L_b]$\per A symmetric argument shows that
$a;b$ is right-regular, and that its right stabilizer  $\R_{a;b}$ is
normal in
$G_\wz$  and coincides with
$\vph_b [\R_a;\L_b]$\per
It remains to prove that
\begin{equation*}\vphs\abp=\hvphs\al\,\vert\,\hvphs\bl\per
\tag{9}\labelp{Eq:P23/7}
\end{equation*}

Write $c=\abb$, and observe that $c\le \xoy$, since
\begin{multline*}
  c=a;b;b\ssm\le (\xoy);(\yoz);(\yoz)\ssm\\=(\xoy);(\yoz);(\zoy)\le \xoy\comma
\end{multline*}
by the definition of $c$, the assumptions on $a$ and $b$, monotony, and \refL{square}(iii),(iv). Also,
\begin{equation*} 
c=\abb=\al;(\tsum\lbl)=\tsum a;\lbl\comma \tag{10} \labelp{Eq:P23/9}
\end{equation*}
by the left-regularity of $\bl$ and complete distributivity.
Compute:
\begin{multline*} \tag{11}\labelp{Eq:else.11}
\ssum X_c   = c;c\ssm
  = (\abb); (\abb)\ssm\\
 = \abb;\bba
=a;\bl; (\ssum\R_b); \bl\ssm;a\ssm\\
 =\ssum a;\bl; \R_b; \bl\ssm;a\ssm
=a;\bl;\bl\ssm;a\ssm\\
=(\abp);(\abp)\ssm
= \ssum\L_{a;b}\comma
\end{multline*}
by the definition of $X_c$\comma \refEq{P23/9}, the two involution laws and the associative law, the right-regularity of $\bl$,
complete distributivity, the definition of $\RS \bl$ as the right stabilizer of $b$\comma the second involution law and the associative law,
and the left-regularity of $\abp$.
Because $X_c$ and $\L_{a;b}$ are sets of atoms, it follows from \refEq{else.11} that
$X_c=\L_{a;b}$\per Use this observation, the definition of $c$ and  \refL{include} (with $a;b$ and $b\ssm$ in place of $a$ and $b$ respectively), and \refL{leftcos}(ii) (with $c$ in place of $a$) to arrive at
\[X_c=\L_{a;b}\seq\LS c\seq \cs Xc\per
\] The first and last sets are equal, so equality  holds everywhere. Apply \refC{lrreg} to conclude that   $c$ is a left-regular element with  normal left stabilizer
$\LS c=\LS\abp\per$

The proof that $c$ is  right-regular  with    normal right stabilizer
$\cs Kc=\rar;\lbl$ involves a  similar computation.
In more detail,
\begin{multline*}\tag{12}\labelp{Eq:else.12}
\tsum \cs Yc   = c\ssm;c
= (\abb)\ssm;(\abb)\\
 = \bba;\abb
= (\ssum\L_b);(\ssum\rar); (\ssum\L_b)\\
=\ssum \L_b; \R_a;\L_b
= \ssum\R_a;\L_b;\L_b= \ssum \R_a;\L_b\comma
\end{multline*}
by the definition of $Y_c$\comma  the definition of $c$, the two involution laws and the associative law,  the right-regularity of $\al$ and the left-regularity of $\bl$, complete distributivity,
 the assumption that $\R_a$ is a normal subgroup of $\cs Gy$, and \refC{stab} (with $b$ in place of $a$), which implies that $\cs Hb$ is a subgroup of $\cs Gy$ and hence closed under the group operaton of relative multiplication.  The sets $Y_c$ and $\R_a;\L_b$ consist of atoms in $G_\wy$, so \refEq{else.12} implies that these sets
 are equal,
 \begin{equation*}
   \cs Yc=\cs  Ka;\cs Hb\per\tag{13}\labelp{Eq:else.13}
 \end{equation*}

\renewcommand\wk{k}
The next step is to check that \begin{equation*}
\R_a;\L_b\seq\R_c\per
\tag{14} \labelp{Eq:P23/14}
\end{equation*}
Consider an element $g$   in $ \R_a;\L_b$\comma and let  $\wk$ in $
\R_a$ and  $\wh$ in $\L_b$ be elements such that
\begin{equation*} 
g= \wk;\wh\per \tag{15} \labelp{Eq:P23/15}
\end{equation*}
Compute:
\begin{multline*}
c;g  =\ssum a;\L_b;g
 =\ssum a;\L_b;\wk;\wh\\
  =\ssum a;\wk;\L_b;\wh
=\ssum a;\L_b;\wh
 =\ssum a;\L_b
 =c
\end{multline*}
by \refEq{P23/9} and complete distributivity, \refEq{P23/15}, the assumption that  $\L_b$ is normal in $\cs Gy$,
the assumption that $\wk $ is in the right stabilizer $ \R_a$ of $a$\comma the assumption that $\wh$ is in the subgroup $\L_b$ of $b$, which is closed
under the group operation of relative multiplication\comma and \refEq{P23/9}.
It follows from this computation that $g$ belongs to the right stabilizer $\RS c$, which completes the proof of \refEq{P23/14}\per

Use \refEq{else.13}, \refEq{P23/14}, and the right-regular version of  \refL{leftcos}(ii) to arrive at
\[Y_c=\R_a;\L_b\seq\RS c\seq Y_c\per\]
The first and last sets are equal, so equality holds everywhere.  Apply \refC{lrreg} (with $c$ in place of $a$) to conclude that $c$ is a right-regular element with   normal
right stabilizer $\cs Kc=\rar;\lbl$\per

The element  $\al\le \xoy $ is regular with normal stabilizers,  by assumption, and the same properties have been established for the element $c$.
Moreover,
\[a = a;y\le \abb=c
\]
by \refL{domain}(iii),(i) (with $b$, $y$, and $z$ in place of $a$, $x$, and $y$ respectively), and the definition of $c$.  Apply   Refinement \refT{th.what} (with $c$ in place of $b$) to
conclude that the  isomorphism $\hvphs a$ from $\cs Gx/\cs Hc$ to $\cs Gy/\cs Kc$ induced by $\vphs a$ coincides with $\vphs c$, in symbols,
\begin{equation*}\vphs c=\hvphs \al\per\tag{16}\labelp{Eq:P23/10}
\end{equation*}

Now write $d=\al\ssm;\al;\bl$.  In an entirely analogous fashion, one shows that
$d$ is a regular element below $\yoz$ with normal left and right stabilizers
\[\LS d = \rar;\lbl\qquad\text{and}\qquad \RS d = \RS\abp\comma
\]
and therefore, since $b\le d$,
\begin{equation*}\vphs d=\hvphs \bl\comma\tag{17}\labelp{Eq:P23/11}
\end{equation*} by Refinement \refT{th.what}.

The last step in the link is to prove that
\begin{equation*}\vphs \abp=\vphs c\,\vert\vphs d\per\tag{18}\labelp{Eq:P23/12}
\end{equation*}
Recall that
$\vph_{a;b}$ maps $\gxlab$ isomorphically to $\gzrab$\per  Fix an element $f $ in
$G_\wx$,  and select an element $h
$ in $G_\wz$ so that the coset $f;\cs H{a;b}$ of the normal subgroup $\cs H{a;b}$ is mapped by $\vphs{a;b}$ to the coset $\cs K{a;b};h$ of the normal subgroup $\cs K{a;b}$. It follows from the
definition of $\vphi{a;b}$ and \refCo{canonical} (with $a;b$ in place of $a$) that
\begin{equation*}\tag{19}\labelp{Eq:P23/12.1}
  f;\cs H{a;b};a;b=a;b;\cs K{a;b};h\per
\end{equation*}
 Use the left- and right-regular versions of \refC{stab} and the remark following it, together with \refEq{P23/12.1},   to arrive at
\begin{equation*}
f;a;b=(f;\L_{a;b});a;b= a;b;(\R_{a;b};h)=a;b;h\per\tag{20}\labelp{Eq:P23/16}
\end{equation*}
Use \refEq{P23/16} and  the implication from (i) to (iii) in \lmref{L:alfabeta} to obtain an element $g$ in $G_\wy$ such that
\begin{equation*} 
f;a=a;g
\quad\text{and}\quad
g;b=b;h\per \tag{21} \labelp{Eq:P23/18}
\end{equation*}
Compute:
\begin{multline*}\tag{22}\labelp{Eq:else.21}
f;c   = f;a;(\ssum\L_b)
 =a;g;(\ssum\L_b)\\=\ssum a;g;\L_b
  =\ssum a;\L_b;g
 =a;(\ssum\L_b);g
 = c;g
\end{multline*}
by \refEq{P23/9}, \refEq{P23/18}, complete distributivity, the assumption that $\L_b$ is normal, complete distributivity, and \refEq{P23/9}.
Use the left- and right-regular versions of \refC{stab} and the remark following it, together with \refEq{else.21}, to arrive at
\begin{equation*}(f;\LS c);c=f;c=c;g=c;(\RS c;g)\per\tag{23}\labelp{Eq:else.22}
\end{equation*}
Apply the definition of $\vphs c$ and \refCo{canonical} (with $c$ in place of $a$) to conclude from \refEq{else.22} that $\vphs c$ maps the coset $f;\cs Hc$ of $\cs Hc$ to the coset $\cs Kc;g$ of $\cs Kc$.
Use the equations
\[\LS c=\LS\abp\qquad\text{and}\qquad\RS c=\rar;\lbl
\] established earlier in the proof to conclude that
\begin{equation*} 
\vph_c (f;\L_{a;b})= \vphs c (f;\cs Hc)= \cs Kc;g=  \R_a;\L_b;g\per
\end{equation*}

 In an entirely analogous fashion, one shows that
\begin{equation*} 
\vph_d (g;\R_a;\L_b)=  \R_{a;b};h\per 
\end{equation*}

The subgroup $\cs Kb;\cs Hb$ is normal in $\cs Gy$, so the cosets  $g;\cs Ka;\cs Hb$ and $\cs Ka;\cs Hb;g$ are equal.
Combine this observation with the preceding equations to see that   the composite isomorphism  $\vphs c\,\vert\,\vphs d$ maps the coset
$f;\LS\abp$ to  the coset $\RS\abp;h$\comma just as does the isomorphism $\vphs\abp$\comma by \refEq{P23/12.1} and \refCo{canonical} (with $a;b$ in place of $a$)\per
Consequently, the two isomorphisms coincide.  This proves \refEq{P23/12}.

Together, \refEq{P23/10}--\refEq{P23/12} immediately yield \refEq{P23/7}:
\[\vphs{a;b}=\vphs c\rp\vphs d=\hvphs a\rp\hvphs b\per\] This completes the
proof of the theorem.
\end{proof}

 Relative Product \refT{else} can be combined with  Refinement
\refT{th.what} and   Isomorphism \refT{chariso} to obtain
an important characterization of the relational composition of two
quotient isomorphisms.  Recall  the notation introduced before \refT{chariso}:  the inner
automorphism of a quotient group $\cs Gx/\cs Hc$ determined by a coset $\h\eta$ of  $\cs Hc$   is the
mapping
$\tau$ defined by
\[\tau(\h\x)=\h\eta\ssm;\h\x;\h\eta
\]
for all cosets $\h\x$ of $\cs Hc$\per This inner automorphism induces an inner automorphism $\hat\tau$ of the quotient group $\cs Gx/\cs H{a;b}$ under the assumption that $\cs Hc$ is a subgroup of $\cs H{a;b}$.  In the statement of the next theorem, the mappings $\hvphs a$ and $\hvphs b$ are the isomorphism described before Relative Product \refT{else},  while $\hvphs c$ is the isomorphism from $\cs Gx/\cs H{a;b}$ to $\cs Gz/\cs K{a;b}$ induced by the isomorphism $\vphi c$ from $\cs Gx/\cs Hc$ to $\cs Gz/\cs Kc$, and $\hat\tau$ is the inner automorphism of $\cs Gx/\cs H{a;b}$ induced by an inner automorphism $\tau$ of $\cs Gx/\cs Hc$\per

\begin{theorem}[Second Relative Product Theorem] \labelp{T:secrp}
Let $\wx$\co $\wy$\co and $\wz$ be measurable atoms\comma and
\[a\le
\xoy\comma\qquad b\le
\yoz,\qquad  c\le\xoz\]    regular elements with normal stabilizers\per In particular\comma they may be atoms\per If
\[\h\eta; c\leq \abp\comma\]
 for some coset $\h\eta$   of $\h
c$\comma then
\[\h\al;\h\cl\seq \vph\inv_\al[\rar;\lbl]\qquad\text{and}\qquad\k\bl;\k\cl\seq \vphi\bl[\rar;\lbl]\comma\]
and
\[\hvphs
a\,\vert\,\hvphs b=\hat\tau\rp\hvphs c \comma
\]
where $\tau$ is the inner automorphism of $\G\wx/\h c$ determined by $\h\eta$\po
\end{theorem}
\begin{proof} The assumption that $a$ and $b$ are regular elements
with normal stabilizers implies that  $\al;\bl$ is a regular
element with normal stabilizers
\begin{equation*}
  \h{\al;\bl}=\vph\inv_\al[\rar;\lbl]\qquad\text{and}\qquad
\k{\al;\bl}=\vphi\bl[\rar;\lbl]\comma\tag{1}\labelp{Eq:secrp3.1}
\end{equation*} and that
\begin{equation*}\tag{2}\labelp{Eq:secrp2.01}
\vphi{\al;\bl}=\hvphs\al\rp\hvphs\bl\comma
\end{equation*}
by Relative Product \refT{else}.

The element $c$ is assumed to be regular with normal stabilizers, so its left translation
\begin{equation*}\tag{3}\labelp{Eq:secrp1}
d=\h\eta;\cl\leq \al;\bl\comma
\end{equation*} is also regular with normal stabilizers, and in fact,
\begin{align*}
\h\dl=\h\cl\qquad&\text{and}\qquad\k\dl=\k\cl\comma\tag{4}\labelp{Eq:secrp1.01}\\
\intertext{ by   Translation \refL{left.reg}(i).  These observations  show that  Refinement \refT{th.what} may be applied to    \refEq{secrp1} to conclude, first, that}
\h\dl\seq\h{\al;\bl}\qquad&\text{and}\qquad\k\dl\seq\k{\al;\bl}\comma\tag{5}\labelp{Eq:secrp2.1}\end{align*}
and, second, that the isomorphism  $\vphi d$ from $\cs Gx/\cs Hd$ to $\cs Gz/\cs Kd$ induces an
isomorphism
$\hvphs d$ from $\cs Gx/\cs H{a;b}$ to $\cs Gz/\cs K{a;b}$ with the property
\begin{equation*}\tag{6}\labelp{Eq:secrp5.01}
\hvphs\dl=\vphi{\al;\bl}\per
\end{equation*} Combine \refEq{secrp2.01} and \refEq{secrp5.01} to arrive at
\begin{equation*}\tag{7}\labelp{Eq:secrp5}
\hvphs\dl=\hvphs\al\rp\hvphs\bl\per
\end{equation*}

Isomorphism \refT{chariso} (with $z$, $c$, and $d$ in place of $y$, $a$, and $b$ respectively)  says that
\begin{equation*}\tag{8}\labelp{Eq:secrp6}
\vphi\dl=\tau\rp\vphi\cl\comma
\end{equation*}
where $\tau$ is the inner automorphism of $\G\wx/\h\cl$ determined by the coset $\h\eta$\per  The inclusions in \refEq{secrp1.01} and \refEq{secrp2.1} imply
that
\begin{equation*}\tag{9}\labelp{Eq:secrp9.1}
  \cs Hc\seq\h{\al;\bl}\qquad\text{and}\qquad\k c\seq\k{\al;\bl}\comma
\end{equation*}
  so the inner automorphism $\tau$ of $\cs Gx/\cs Hc$ induces an inner automorphism $\hat\tau$ of $\cs Gx/\cs H{a;b}$, and the isomorphism $\vphi c$  from $\cs Gx/\cs Hc$ to $\cs Gz/\cs Kc$ induces an isomorphism $\hvphs c$ from
$\cs Gx/\cs H{a;b}$ to $\cs Gz/\cs K{a;b}$, and
\begin{align*}
\hvphs\dl&=\hat\tau\rp\hvphs\cl\comma\tag{10}\labelp{Eq:secrp7}\\
\intertext{by     \refEq{secrp6}.
Combine \refEq{secrp5} and \refEq{secrp7} to  arrive at}
\hvphs\al\rp\hvphs\bl&=\hat\tau\rp\hvphs\cl\per
\end{align*}

Together, \refEq{secrp3.1}, \refEq{secrp1.01},  and \refEq{secrp2.1} show that
  $\h c$ and $\k c$ are
 subgroups of
 \begin{align*}
   \vph\inv_\al[\rar;\lbl]\qquad&\text{and}\qquad
\vphi\bl[\rar;\lbl] \tag{11}\labelp{Eq:secrp5.1}\\
\intertext{respectively. Since $\cs Ka$ and $\cs Hb$ are subgroups of $\cs Ka;\cs Hb$, the two groups in \refEq{secrp5.1} also include $\vphi a\mo[\cs Ka]$ and $\vphi b[\cs Hb]$ respectively, that is to say, they include $\cs Ha$ and $\cs Kb$, by  the definitions of $\vphi a$ and $\vphi b$. Use these observation   to arrive at}
\h\al;\h\cl\seq \vph\inv_\al[\rar;\lbl]\qquad&\text{and}\qquad\k\bl;\k\cl\seq \vphi\bl[\rar;\lbl]\per
 \end{align*}
 \end{proof}

The special case of the theorem in which
$c$ is below $ \al;\bl$ is important enough to merit separate formulation. In this case, the   translating coset $\cs H\eta$ is the identity coset $\cs Hc$, and therefore the inner automorphism $\tau$ determined by this coset is the
identity automorphism of $\cs Gx/\cs Hc$.

\begin{cor}
\labelp{C:abccor} Let $\wx$\comma $\wy$\comma and $\wz$ be measurable atoms\comma and
\[a\le
\xoy\comma\qquad b\le
\yoz,\qquad c\le\xoz\]    regular elements with normal
stabilizers\per If
 $\cl\leq\al;\bl$ \comma then
$\hvphs
a\,\vert\,\hvphs b =\hvphs c
$.
\end{cor}

Another special case of Relative Product \refT{secrp} is when one of the quotient groups
\[\cs Gx/\cs H{a;b},\qquad \cs Gy/(\cs Ka;\cs Hb)\comma\qquad\cs Gz/\cs K{a;b}\] is abelian (and hence all of them are abelian, since they are isomorphic to one another).  In this case, the inner automorphism
$\hat\tau$ mentioned in the theorem is again the identity automorphism, and therefore may be omitted from the final equation of the theorem.
\begin{cor} \labelp{C:abelrelprod}Let $\wx$\comma $\wy$\comma and $\wz$ be measurable atoms\comma and
\[a\le
\xoy\comma\qquad b\le
\yoz,\qquad\text{and}\qquad c\le\xoz\]    regular elements with normal
stabilizers\per  If the quotient   group $\G\wx/\cs H{a;b}$  is abelian\co
then
\[\hvphs a\,\vert\,\hvphs b = \hvphs c
\]
whenever  some translation of $\cl$ is below $\al;\bl$\po
\end{cor}

  Relative Product \refT{else} says that, under
suitable hypotheses, the inverse image of
$\rar;\lbl$ under $\vphs\al\mo$ is  $\LS\abp$\comma and the
  image of $\rar;\lbl$ under $\vphs\bl$ is $\RS\abp$\per
The following corollary gives, for atoms, an alternative
description of these images\per
\begin{theorem}[Image Theorem] \labelp{T:image}
Let $\wx$\comma $\wy$\comma and $\wz$ be measurable atoms\per If
\[a\le
\xoy\comma\qquad b\le
\yoz,\qquad c\le\xoz\]  are all atoms\comma
then
\[\vphs
a[\LS\al;\LS
c]=\rar;\lbl\comma\qquad
\vphs \bl[\rar;\lbl]=\RS \bl;\RS c\comma\]
and
\[\vphs \cl[\h\al;\h\cl]=\RS \bl;\RS c
\per\]
In particular\co
\[\LS\abp =\LS\al;\LS c\qquad\text{and}\qquad\RS\abp =\RS \bl;\RS c\per
\]
\end{theorem}


\begin{proof} 
 Consider elements  $a$, $b$, and $c$  satisfying the hypotheses of the
theorem.  All three  are   regular elements
with normal stabilizers, by Corollaries \ref{C:atoms} and
\ref{C:a-b.atoms}.  The relative product $a;b$ is a regular element below $\xoz$ with normal stabilizers, by Relative Product
\refT{else}, and in particular, it is non-zero, by \refL{regnonzero}.  The left-translations
of $\cl$ are atoms that partition  $\xoz$, by   Atomic Partition \refL{sumatom}, so one of
these left translations must be below the non-zero element $\al;\bl$.  Apply Relative Product
\refT{secrp} to obtain
\begin{equation*}\h\al;\h\cl\seq \vph\inv_\al[\rar;\lbl]\qquad\text{and}\qquad\k\bl;\k\cl
\seq\vphs \bl[\rar;\lbl]\per\tag{1}\labelp{Eq:P30/0}
\end{equation*}

The converses  $\al\ssm$, $\bl\ssm$, and $\cl\ssm$ are atoms  below
\[\yox,\qquad\zoy,\qquad\text{and}\qquad \zox\]
 respectively, by Lemmas \ref{L:laws}(vi) and \ref{L:square}(iii), the assumptions on $a$, $b$, and $c$, and monotony. Apply the argument from the
preceding paragraph to the atoms
$\al\ssm$,
$\cl$, and  $\bl$ in place of $a$, $b$, and $c$ respectively to get
\begin{align*}\h{\al\ssm};\h\bl\seq
\vph\inv_{\al\ssm}[\k{\al\ssm};\h\cl]\qquad&\text{and}\qquad\k\cl;\k\bl
\seq\vphs \cl[\k{\al\ssm};\h\cl]\per\tag{2}\labelp{Eq:P30/1}\\
\intertext{Apply  the same argument again to the atoms   $\bl$, $\cl\ssm$, and $\al\ssm$ in place of $a$, $b$, and $c$ respectively to get}
\h\bl;\h{\al\ssm}\seq
\vph\inv_\bl[\k\bl;\h{\cl\ssm}]\qquad&\text{and}\qquad\k{\cl\ssm};\k{\al\ssm}
\seq\vphs {\cl\ssm}[\k\bl;\h{\cl\ssm}]\per\tag{3}\labelp{Eq:P30/2}
\end{align*}

The equations
\begin{alignat*}{3}
\vphs {a\ssm}&=\vphs a\mo,&\qquad\LS{a\ssm}&=\RS
a\comma&\qquad
\RS{a\ssm}&=\LS a\comma\tag{4}\labelp{Eq:image.4}\\
\vphs {c\ssm}&=\vphs c\mo,&\qquad\LS{c\ssm}&=\RS
c\comma&\qquad
\RS{c\ssm}&=\LS c\tag{5}\labelp{Eq:image.5} \end{alignat*} are all valid,
by   Converse \refT{convthm}, so the inclusions in \refEq{P30/1} and \refEq{P30/2} may be rewritten in the forms
\begin{align*}\k\al;\h\bl\seq
\vphi\al[\h\al;\h\cl]\qquad&\text{and}\qquad\k\cl;\k\bl
\seq\vphs \cl[\h\al;\h\cl]\comma\tag{6}\labelp{Eq:P30/3}\\
\intertext{and} \h\bl;\k\al\seq
\vph\inv_\bl[\k\bl;\k\cl]\qquad&\text{and}\qquad\h\cl;\h\al
\seq\vphs \cl\mo[\k\bl;\k\cl]\per\tag{7}\labelp{Eq:P30/4}
\end{align*} For example, to obtain the first inclusion in \refEq{P30/3}, use the second equation in \refEq{image.4}, the first inclusion
in \refEq{P30/1}, the first equation in \refEq{image.4}, and the fact that the inverse of the inverse of $\vphi a$ is $\vphi a$, to arrive at
\begin{equation*}
  \cs Ka;\cs Hb=\cs H{a\ssm};\cs Hb\seq \vph\inv_{\al\ssm}[\k{\al\ssm};\h\cl]
  =(\vph\inv_a)\inv[\cs H a;\h\cl]=\vphi a[\cs Ha;\cs Hc]\per
\end{equation*}
The first equations in \refEq{P30/0} and \refEq{P30/3} yield
\begin{equation*}
  \vphi\al[\h\al;\h\cl]=\k\al;\h\bl\per\tag{8}\labelp{Eq:image.8}
\end{equation*}
Similarly, the second equation in \refEq{P30/0} and the first equation in \refEq{P30/4}, together with the assumption that the stabilizers
are normal subgroups, yield
\begin{equation*}
  \vphi\bl[\k\al;\h\bl]=\k\bl;\k\cl\per\tag{9}\labelp{Eq:image.9}
\end{equation*}
Finally, the second equations in \refEq{P30/3} and   \refEq{P30/4}, and the assumption that the stabilizers are normal, yield
\begin{equation*}
  \vphi\cl[\h\al;\h\cl]=\k\bl;\k\cl\per\tag{10}\labelp{Eq:image.10}
\end{equation*}

To obtain the first equation in the final assertion of the theorem, use   Relative Product
\refT{else} and \refEq{image.8},
\[\cs H{a;b}=\vph\inv_a[\cs Ka;\cs Hb]=\cs Ha;\cs Hc\per\] The second equation is obtained in a similar fashion, using \refEq{image.9} instead of \refEq{image.8}.
\end{proof}

\comment{

The following application of Relative Product \refT{secrp} and Image \refT{image} to atoms is the one of the main results that
will be needed in the sequel.

\begin{theorem}[Relative Product Theorem for Atoms] \labelp{T:abelrelprod}
Let $\wx$\comma $\wy$\comma and $\wz$ be measurable atoms\comma and
\[a\le
\xoy,\qquad
 b\le
\yoz,\qquad\text{and}\qquad c\le\xoz\]  atoms\per If
\[\h\eta; c\leq \abp\]
for some coset $\h\eta$   of $\h
c$\comma then
\[\h{a;b}=\h\al;\LS\cl=\vph\inv_\al[\rar;\lbl]\qquad\text{and}\qquad
\vphi\bl[\rar;\lbl]=\k\bl;\k\cl=\k{a;b}
\comma\]
and
\[\hvphs
a\,\vert\,\hvphs b=\hat\tau\rp\hvphs c \comma\]
where $\tau$ is the inner automorphism of $\G\wx/\h\cl$ induced by $\h\eta$\po
If
$c\leq a;b$\co then
\[\hvphs
a\,\vert\,\hvphs b=\hvphs c \per\]
\end{theorem}
}
\section{Semi-scaffolds and scaffolds}\labelp{S:sec6}

The previous sections take  a ``local" perspective, and study
properties of certain types of elements that lie below rectangles
with  measurable atoms for sides.  The most important elements of
this kind are regular elements with normal stabilizers, and in
particular, atoms.  This and the next sections take a
``global" perspective, and focus on the entire measurable relation
algebra.

\begin{df} \labelp{D:funct-dens}
A relation algebra is called
\textit{measurable} if the identity element is the sum of measurable atoms.  If the
identity element is the sum of  finitely
measurable atoms\comma  then the algebra is said to be
 \textit{finitely measurable}\po\qeddef
\end{df}

Fix a complete and atomic, measurable relation algebra $\f A$ for the remainder of the discussion. The assumption of completeness is only
needed in order not to have to worry about the existence of certain infinite sums.  All elements and operations are assumed to be those of $\f A$.
 Let $I$ be the set
of   measurable atoms in $\f A$.  The assumption of measurability is
expressed symbolically by the equation $\ident = \tsum I $\per
Define a binary relation $\mc E$ on the set $I$ by putting
\[\mc E=\{\pair x y:\wx,\wy\in I \text{ and } \xoy\neq 0\}\per
\]
For atoms $x$ and $\wy$, the conditions
\[\xoy\neq 0\qquad\text{and}\qquad 1;x;1=1;y;1\]
are equivalent, by \refL{laws.1}(viii), so an alternative definition of $\mc E$ is given by
\[\mc E=\{\pair x y:\wx,\wy\in I \text{ and } 1;x;1=1;y;1\}\per
\]
This second way of defining $\mc E$ makes clear that $\mc E$ is an equivalence
relation on the set $I$.

The principal notion that will be in this section is that of a
semi-scaffold.

\medskip
\begin{df} \labelp{D:scaffold}
A system $a= \langle a_{xy}:\pair x y\in\mc E\,\rangle$ of atoms
  is called a
\textit{semi-scaffold} if\co for
all pairs $\pair x y$  in $\mc E$\co and for all atoms $x$ in $I$\co the following
conditions are satisfied\per
\begin{enumerate}
\item[(i)] $\axy \le \xoy$\per
\item[(ii)] $\axx =\wx$\per
\item[(iii)] $\ayx=a_{xy}\ssm$\,\po
\end{enumerate}
A semi-scaffold is called a \textit{scaffold} if, for all pairs $\pair x y$ and $\pair y z$ in $\mc E$\comma
\begin{enumerate}
\item[(iv)] $\axz\le \axy; \ayz$\per
\end{enumerate}\qeddef
\end{df}

The conditions defining a semi-scaffold and a scaffold can  be
weakened somewhat. To this end, it is helpful to assume that the
set $I $ of measurable atoms is linearly ordered, say by a
relation $\,<\,$.
\begin{lm}\labelp{L:scafchar}  A system   $a=\langle a_{xy}:\pair x y\in\mc
E\,\rangle$ of atoms is a semi-scaffold if and only if the following assumptions  are satisfied\per
\begin{enumerate}
\item[(i)] $\axy\le \xoy$ for all $\pair xy$  in $\mc E$ with $x<y$\per
\item[((ii)] $\axx =\wx$ for  all $\wx$ in $I$\per
\item[(iii)] $\ayx=a_{xy}\ssm$  for all  $\pair x y$ in $\mc E$ with $x<y$\po
\end{enumerate}  A semi-scaffold is a scaffold if and only if
\begin{enumerate}
\item[(iv)] $\axz\le \axy; \ayz$ for all $\pair xy$ and $\pair yz$ in $\mc E$ with $x<y<z$\per
\end{enumerate}
\end{lm}
\begin{proof}  It is obvious that a system
of atoms satisfying conditions (i)--(iii) in the definition of a
semi-scaffold satisfies assumptions (i)--(iii) of the lemma.
 To establish
the reverse implication, assume that $a$ is  a system of atoms  satisfying
assumptions (i)--(iii) of the lemma.
Fix a pair $\pair xy$ in $\mc E$.  The first step is to verify that semi-scaffold condition (i) holds for this pair. If $x<y$, then condition (i) holds
by assumption.  If $x=y$, then
 \[\cs a\xy=\cs a\xx=x\le \xox=\xoy\comma\] by the hypothesis of this case, assumption (ii) of the lemma, and \refL{square}(i) (with $y=x$).
 If $x>y$, then
 $\cs a\yx\le\yox$, by assumption (i) of the lemma (with $x$ and $y$ interchanged), and consequently
 \[\cs a\xy=a\ssm_\yx\le(\yox)\ssm=\xoy\comma\]
by assumption (iii) of the lemma (with $x$ and $y$ interchanged), the preceding observation, monotony, and \refL{square}(iii). This completes the verification
of semi-scaffold condition (i).

Semi-scaffold condition (ii) coincides with assumption (ii) of the
lemma, so there is nothing to verify. Turn now to the verification
of semi-scaffold condition (iii). If $x<y$, then condition (iii)
holds by  assumption (iii) of the lemma. If $x=y$\comma then\[\cs
a\yx=\cs a\xx=x=x\ssm=\cs a\xx\ssm=\cs a\xy\ssm\comma\] by the
hypothesis $x=y$, assumption (ii) of the lemma, and
\refL{laws}(vii). If $x>y$, then $\cs a\xy=\cs a\yx\ssm$, by
assumption (iii) of the lemma (with $x$ and $y$ interchanged), and
therefore
\[\cs a\yx=(\cs a\yx\ssm)\ssm=\cs a\xy\ssm\comma\] by the first involution law and the preceding observation.

To prove the second assertion of the theorem, assume  that the system of atoms $a$   satisfies assumptions (i)--(iii) of the lemma.  It then satisfies
semi-scaffold conditions (i)--(iii), by the arguments of the previous paragraphs, so it may be assumed that these conditions hold for $a$.  The goal is to derive  condition (iv) from  these conditions and assumption (iv) of the lemma. Consider pairs $\pair xy$ and $\pair yz$  in $\mc E$. The atom $\cs a\xz$ is below $\xoz$, by semi-scaffold condition (i), so its domain is $x$, and $x;\cs a\xz=\cs a\xz$, by \refL{domain}. Consequently, if  $x=y$\comma
then\[\cs a\xy;\cs a\yz=\axx;\axz= x;\cs a\xz=\cs a\xz\comma\] by the hypothesis $x=y$,  semi-scaffold condition (ii), and the preceding observation.  A similar argument applies if $y=z$.   If $x=z$, then
\[\axy;\ayz=\axy;\ayx=\axy;\cs a\xy\ssm=\tsum\h{\axy}\ge \wx=\axx=\axz\comma
\]
by the hypothesis $x=z$, semi-scaffold condition (iii), the assumption that $\cs a\xy$ is an atom and hence regular, by \refC{atoms},  the fact that
the stabilizer $\h{\axy}$ is a subgroup of $\cs Gx$ and therefore contains the group identity element $x$, and semi-scaffold condition (ii).
Assume now that $x$, $y$, and $z$ are distinct.  If $x<y<z$, then  condition (iv) holds by assumption (iv) of the lemma. If $x<z<y$\comma
then
\[\axy\le\axz;\azy=\axz;\cs a\yz\ssm\comma
\] by the hypothesis of this case, assumption (iv)   of the lemma (with $y$ and $z$ interchanged), and  semi-scaffold  condition (iii). Apply the cycle laws for atoms in \refL{laws.1}(v) to arrive at
$\cs a\xz\le\cs a\xy;\cs a\yz$. Similarly, if say $z<x<y$, then
\[\cs a\zy\le\cs a\zx;\cs a\xy=\cs a\xz\ssm;\cs a\xy\comma\] by the hypothesis of this case, assumption (iv) of the lemma (with $z$, $x$, and $y$ in place of $x$, $y$, and $z$ respectively) and   semi-scaffold condition (iii) (with $z$  in place of $y$). Form the converse of both sides of this inequality, and use  semi-scaffold  condition (iii) (with $z$ in place of   $x$), monotony, and the two involution laws to obtain
\[\cs a\yz=\cs a\zy\ssm\le (\cs a\xz\ssm;\cs a\xy)\ssm=\cs a\xy\ssm;(\cs a\xz\ssm)\ssm=\cs a\xy\ssm;\cs a\xz\per\] Apply the cycle laws for atoms in \refL{laws.1}(v) to arrive at $\cs a\xz\le\cs a\xy;\cs a\yz$.  The remaining three cases of condition (iv) are treated in a similar fashion.
\end{proof}

With the help of the preceding lemma, it is easy to show that every atomic, measurable relation algebra has a semi-scaffold.

\begin{lm}[Semi-scaffold Existence Lemma]\labelp{L:semil} Every atomic\comma  measurable
relation algebra has a semi-scaffold\per
\end{lm}
\begin{proof} Assume that the set $I$ of measurable atoms is linearly ordered, say by a relation $\,<\,$.
 For each
$x$ in $I$, put
\begin{equation*}
\axx=x\comma
\end{equation*}  and observe that assumption (ii) of \refL{scafchar} holds.  For each pair $\pair xy$ in $\mc E$ with $x<y$, choose an atom $b\le \xoy$, and put
\[\axy=b\qquad\text{and}\qquad\ayx=b\ssm\per
\]
Such an atom  $b$ exists because the relation algebra is assumed to be atomic, and the
definition of $\mc E$, together with  the assumption that $\pair xy$ is in $\mc E$, imply that the rectangle  $\xoy$ is non-zero.  Observe that  $\ayx$ is also  an atom, by \refL{laws}(vi).  The choice of $\axy$ and the definition of $\ayx$ imply that assumptions (i) and (iii) of \refL{scafchar}  hold. Consequently, the system
\[a=\langle a_{xy}:\pair x y\in\mc
E\,\rangle\] of atoms is a semi-scaffold, by \refL{scafchar}.
\end{proof}

It is worth pointing out that, in general, it is not true that
every atomic, measurable relation algebra has a scaffold.

Fix a  semi-scaffold
\[a=\langle a_\wi : \pair x y\in \mc E \,\rangle
\]
in  $\f A$, and consider an arbitrary atom $b=\axy$ in this
semi-scaffold.  The right and left stabilizers $\cs Hb$ and $\cs
Kb$ are normal subgroups of the groups $\cs Gx$ and $\G y$ of permutations of
$x$ and $y$ respectively, by \refC{a-b.atoms}. Write
\[\cs H\xy=\cs Hb\qquad\text{and}\qquad \cs K\xy=\cs Kb\per\] The quotient groups $\cs Gx/\cs H\xy$ and $\cs Gy/\cs K\xy$ are isomorphic\comma and in fact there are cosets systems
\[\langle H_{\wi,\xi} :\xi <\kai\wi\,\rangle\qquad\text{and}\qquad \langle
K_{\wi,\xi} :\xi <\kai\wi\,\rangle\] of $\cs H\xy$ and $\cs K\xy$ in $\cs Gx$ and $\cs Gy$ respectively such that
\[ \ls\wi\xi;\ai=\ai;\rs\wi\xi\comma\] and the function $\vphi\xy$ from $\cs Gx/\cs H\xy$ to $\cs Gy/\cs K\xy$ defined by
\begin{equation*}
\vph_\wi (H_{\wi,\xi}) = K_{\wi,\xi}
\end{equation*} for each $\xi$
is an isomorphism, by Isomorphism \refT{isom} and the subsequent remarks.
  Without loss of generality, it may be
assumed that
\begin{equation*}
H_{\wi,0}= H_\wi
\qquad\text{and}\qquad
K_{\wi,0}= K_\wi\per
\end{equation*}

The left stabilizer $\cs H{b\ssm} $ coincides with $\cs Kb$,  the right stabilizer $\cs K{b\ssm}$ coincides with $\cs Hb$, and the quotient isomorphism $\vphi {b\ssm}$ coincides with $\vphi b\mo$, by Converse \refT{convthm} (with $b$ in place of $a$). Consequently,
since \[\cs a\yx=\cs a\xy\ssm=b\ssm\comma\] by semi-scaffold condition (ii) and the definition of $b$, we  have
\begin{align*}
  \cs H\yx=\cs K\xy\qquad&\text{and}\qquad \cs K\yx=\cs H\xy\per\\
  \intertext{Apply \refCo{conver} to  choose coset systems}
 \langle H_{\yx,\xi} :\xi <\kai\wi\,\rangle\qquad&\text{and}\qquad \langle
K_{\yx,\xi} :\xi <\kai\wi\,\rangle\\
\intertext{of $\cs H\yx$ and $\cs K\yx$ in $\cs Gy$ and $\cs Gx$ respectively such that}
\ls\yx\xi=\rs\xy\xi\qquad&\text{and}\qquad \rs\yx\xi=\ls\xy\xi
\end{align*}
 for each $\xi<\kai\xy$, and the function $\vphi\yx$ from $\cs Gy/\cs H\yx$ to $\cs Gx/\cs K\yx$ that is the inverse of
 $\vphi\xy$ is determined by
\begin{equation*}
\vph_\yx (H_{\yx,\xi}) = K_{\yx,\xi}\per
\end{equation*}

For each $\a <\kai\wi$, the element
\[\as\wi\a =\ls\wi\a;\ai=\ai;\rs\wi\a
\]  is well defined, by \refC{stab}, and  the system of elements $\langle\as\wi\a:\a<\kai\wi\rangle$ is a partition of
$\xoy$ into atoms, by Atomic Partition \refL{sumatom}. The following lemma summarizes these observations.

\begin{lm} \labelp{L:each.atom}  Let $\pair \wx \wy$ be a pair  in $\mc E$\po
The system of elements
$\langle\as\wi\a:\a<\kai\wi\rangle$  is an atomic partition of $\xoy$\po
Consequently\comma an element below $\xoy$ is an atom if and only if it has the form $\as\wi\a$ for a \opar
unique\cpar\
$\a<\kai \wi$\per
\end{lm}

The preceding lemma is local in character, referring to atoms below a given rectangle with measurable sides.
The lemma implies a corresponding global result.

\begin{lm}[Semi-scaffold Partition Lemma] \labelp{L:total.atom}  The system of elements
\[\langle\as\wi\a:\pair\wx\wy\in \mc E\text{ and }\a<\kai \wi\rangle\] is an atomic partition of the unit\per   Consequently\comma
an element in the algebra is an
atom if and only if it has the form $\as\wi\a$ for some \opar unique\cpar\ pair
$\pair\wx\wy$ in $
\mc E$ and a \opar unique\cpar\
$\a<\kai
\wi$\per
\end{lm}

\begin{proof} 
 The given relation algebra is assumed to be measurable, so
\begin{equation*} 
\ident=\ssum I\comma \tag*{(1)} \labelp{Eq:P29/1}
\end{equation*} by the definition of measurability.  Use the identity laws, \ref{Eq:P29/1}, complete distributivity, and definition of
$\mc E$ as the set of pairs $\pair xy$ such that
$x;1;y\neq 0
$ to obtain
\begin{multline*}
\tag{2}\labelp{Eq:total.atom.2}
1   = \ident;1; \ident
= (\ssum I);1; (\ssum I)\\
= \ssum \{ \xoy:\wx,\wy\in I\}
= \ssum \{ \xoy:\pair \wx \wy\in \mc E\}\per
\end{multline*} For each pair $\pair xy$ in $\mc E$, we have
\begin{equation*}\tag{3}\labelp{Eq:total.atom.3}
  \xoy=\tsum\{\as\wi\a:\a<\kai\wi\},
\end{equation*}
by \refL{each.atom}.  Combine \refEq{total.atom.2} and \refEq{total.atom.3} to arrive at
\begin{equation*}
  1=\tsum\{\as\wi\a:\pair xy\in\mc E\text{ and }\a<\kai\wi\}\per
\end{equation*} It remains to prove that two atoms $\as\wi\a$ and $\as {uv}\b$ are disjoint if $x\neq u$, or $y\neq v$, or $x=y$ and $u=v$, but $\a\neq \b$. In the first case, we have
\begin{equation*} 
\as\wi\a\cdot\as{uv}\b \le (\xoy)\cdot  (\uov)= (\wx\cdot \wu); 1;
(\wy\cdot \wv) =0;1;(y\cdot v)=0\comma \tag*{(4)}
\labelp{Eq:P29/2.1}\end{equation*} by monotony, the assumption
that $\langle\cs a\wi:\pair xy\in\mc E\rangle $ is a
semi-scaffold, \refL{square}(ii),  the assumption that $x$ and $u
$ are distinct atoms, and \refL{laws.1}(i).  A similar argument
applies in the second case, when $y\neq v$. In the third case, the
desired disjointness follows directly from \refL{each.atom}.
\end{proof}

The next lemma specifies the elements $\as\wi\a$ that are subidentity atoms

\begin{lm}[Semi-scaffold Identity Lemma] \labelp{L:diag.atom} An element $\as\wi\a$ is a
subidentity atom if and only if $\wx=\wy$  and $\a=0$\po
Consequently\co
\[\ident=\tsum\{\as{\wx\wx} 0:\wx\in I\}\per\]
\end{lm}
\begin{proof}  The rectangles $\xoy$ have a non-zero meet with the identity element
$\ident$ if and only if $x=y$, and in this case that meet is $\wx$, by \refL{square}(i) and the assumption that $x$ and $y$ are atoms.
This means that the  atom $\as\wi\a$, which is below  the rectangle $\xoy$, by \refL{each.atom}, can lie below $\ident$ only when
$ \wx=\wy$ and \[\as\xx\a= \wx=\al_\xx\per\] For this last
equality to hold, the index $\a$ must be $0$, by  the definition of $\as\wi\a$ and the assumption  about the indexing
of the coset systems. This proves the first assertion of the lemma. The second  follows
from the first  and the assumption that the \ra\ is measurable.
\end{proof}

The next lemma determines the converse of each atom $\as\wi\a $\per

\renewcommand\wj{\yx}
\begin{lm}[Semi-scaffold Converse Lemma] \labelp{L:aia.eq.ajb} For each pair
$\pair\wx\wy$  in
$\mc E$\comma and each $\a<\kai\wi$\comma we have
$\as\wi\a\ssm=\as\wj\b$\co where
 $H_{\wi,\lph}\ssm=H_{\wi,\bt}$\per
\end{lm}
\begin{proof} Assume
\begin{equation*}\tag{1}\labelp{Eq:aia.eq.ajb.1}
  H_{\wi,\lph}\ssm=H_{\wi,\bt}\per
\end{equation*} Use the definition of $\as\wi\a$, the second involution law,
semi-scaffold condition (iii) and \refEq{aia.eq.ajb.1}, the assumption
 about the system of cosets $K_{\wj,\bt}$, described after \refL{semil}
 and based on \refCo{conver}, and \refCo{canonical} and the definition of
$\as\yx\bt$ to arrive at
\begin{multline*}
(a_{\wi,\lph})\ssm  =( H_{\wi,\lph};a_\wi)\ssm = a_\wi\ssm; H_{\wi,\lph}\ssm
  =a_\wj; H_{\wi,\bt}\\
 = a_\wj; K_{\wj,\bt}
  = H_{\wj,\bt};a_\wj
 = a_{\wj,\bt}\per
\end{multline*}
\end{proof}

\renewcommand\wj{\wwz}
The next step is to determine the relative product of two atoms $\as\wi\a$ and $\as\wj\b$\per
 Begin with the  trivial  case.

\begin{lm} \labelp{L:eqzero}
Let $ (\wx,\wy)$ and $  (\ww,\wz)$ be pairs in $\mc E$\co and $\lph <
\kai\wi$ and $\bt <
\kai\wj$\per  If $\wy \not= \ww$\co
then $a_{\wi,\lph}; a_{\wj,\bt}=0$\per
\end{lm}

\begin{proof} 
If  $\wy$ and $\ww$ are not equal, then they are
disjoint, because they are  atoms, and therefore
\begin{equation*}
a_{\wi,\lph}; a_{\wj,\bt} \le (\xoy); (\woz)
= 0\comma
\end{equation*}
by semi-scaffold condition (i), monotony, and \refL{square}(v)\per
\end{proof}

\renewcommand\wj{\yz}
The non-trivial case of relative multiplication is more interesting and more involved.
\begin{lm}[Semi-scaffold Relative Product Lemma] \labelp{L:prewhat}
Let $(\wx,\wy)$ and $(\wy,\wz)$ be pairs in $\mc E$\co  and
$\lph < \kai\wi$ and $\bt <
\kai\wj$\per If $\x<\kai\wk$ is any index such that \[\hs\wk\x;\al_\wk\le
\al_\wi;\al_\wj\comma\] then
\[a_{\wi,\lph}; a_{\wj,\bt}=\tsum\{a_{\wk,\gm}:\gm<\kai \xz\text{ and }
\ls\wk\gm\seq\vphs\wi\mo[\rs\wi\a;\ls\wj\b];\hs\wk\x\myspace\}
\per
\]
\end{lm}
\begin{proof} The relative product $\ai;\aj$ is a regular element below $\xoz$ with normal
  stabilizers, by  Relative Product \refT{else}, and the
left stabilizer is the product group $\cs H\xy;\cs H\xz$, by Image \refT{image} (with $\pair xy$, $\pair yz$, and $\pair xz$ in place of $a$, $b$, and $c$ respectively). The
 isomorphism $\vphi\xy$ from $\cs Gx/\cs H\xy$ to $\cs Gy/\cs K\xy$ induces an isomorphism $\hvphs \xy$ from
\begin{align*}
\cs Gx/(\cs H\xy;\cs H\xz)\qquad&\text{to}\qquad \cs Gy/(\cs K\xy;\cs H\yz)\comma\\
\intertext{by Image \refT{image}.  To simplify notation, write}
\M=\LS\wi;\LS\wk\qquad&\text{and}\qquad\P=\RS\wi;\LS\wj\comma
\end{align*}
so that $\M$ is the left
stabilizer of $\ai;\aj$\comma  and $\hvphs\wi$ maps $\G\wx/\M$ isomorphically to $\G\wy/\P$\per Let
\[\langle\MS\eta:\eta<\la\rangle\qquad\text{and}\qquad
\langle\PS\eta:\eta<\la\rangle\] be coset systems of $\M$ and $P$ in $\G\wx$ and   $\G\wy$ respectively such that
\[\vphs\wi[\MS\eta]=\PS\eta
\]
for each $\eta<\la$.  The  coset product  $\rs\wi\a;\ls\wj\b$ is a coset of $\P$, so it must  coincide with  $\PS\rho $ for some $\rho<\la$, and therefore
\begin{equation*}\MS\rho=\vphs\wi\mo[\PS\rho]=\vphs\wi\mo[\rs\wi\a;\ls\wj\b]\per
\tag{1}\labelp{Eq:70/1}
\end{equation*}
The principal step in the proof is showing that
\begin{equation*}\MS\rho;\ai;\aj=a_{\wi,\lph};
a_{\wj,\bt}\per\tag{2}\labelp{Eq:70/2}
\end{equation*}

As $\LS\wi$ is a  subgroup of $\M$, there must be a partition
$\langle\Gamma_\eta:\eta<\la\rangle$
 of the index set $\kai\wi$ such that
\begin{equation*}\MS\eta=\tbigcup_{\zeta\in\Gamma_\eta}\ls\wi\zeta\tag{3}\labelp{Eq:70/3}
\end{equation*}
for each $\eta$.  Taking $\rho$ for $\eta$
gives
\begin{multline*}\rs\wi\a;\ls\wj\b=\vphs\wi[\MS\rho]=\vphs\wi[\tbigcup_{\zeta\in\Gamma_\rho}\ls\wi\zeta]\\
=\tbigcup_{\zeta\in\Gamma_\rho}\vphs\wi(\ls\wi\zeta)=\tbigcup_{\zeta\in\Gamma_\rho}\rs\wi\zeta\comma \tag{4}\labelp{Eq:70/4}
\end{multline*}
by \refEq{70/1}, \refEq{70/3}, the distributivity of function images over unions, and the
definition of
$\vphs\wi$\per
Recall from the remarks preceding \refL{each.atom} that
\begin{equation*}\ls\wi\zeta;\ai=\ai;\rs\wi\zeta\per\tag{5}\labelp{Eq:70/5}
\end{equation*}
for each $\zeta<\kai\wi$\per  Compute:
 \begin{align*}
 \MS\rho;\ai;\aj &=(\tbigcup_{\zeta\in\Gamma_\rho}\ls\wi\zeta);\ai;\cs a\yz=\tsum_{\zeta\in\Gamma_\rho}\ls\wi\zeta;\ai;\aj\\&=\tsum_{\zeta\in\Gamma_\rho}\ai;\rs\wi\zeta;\aj
  =\ai;(\tbigcup_{\zeta\in\Gamma_\rho}\rs\wi\zeta);\aj\\&=\ai;(\,\rs\wi\a;\ls\wj\b);\aj
 =(\ai;\,\rs\wi\a);(\ls\wj\b;\aj)\\&=a_{\wi,\lph};
a_{\wj,\bt}\comma
\end{align*}
by
\refEq{70/3},  complete
distributivity, \refEq{70/5}, complete
distributivity, \refEq{70/4}, associativity, and the definitions of $\as\wi\a$ and $\as\wj\b$\per
This proves \refEq{70/2}.

The group $\LS\wk$ is included in $\M$, by the definition of $\M$, so there is a
partition
$\langle\Delta_\eta:\eta<\la\rangle$ of $\kai \wk$ such that
\begin{equation*}\MS\eta=\tbigcup_{\g\in\Delta_\eta}\ls\wk\g\tag{6}\labelp{Eq:70/6}
\end{equation*}
for each $\eta<\la$. Let $\xi<\kai\xz$ be an index such that
\begin{equation*}\tag{7}\labelp{Eq:70/7.1}
  \hs\wk\x;\al_\wk\le
\al_\wi;\al_\wj\per
\end{equation*}
There is a unique index $\sigma<\mu$
such that
\begin{equation*}\tag{8}\labelp{Eq:70/10}
\MS\sigma=\MS\rho;\ls\wk\x\comma
\end{equation*}
and for this index we have
\begin{equation*}
\tbigcup_{\g\in\Delta_\sigma}\ls\wk\g=\MS\sigma=\MS\rho;\ls\wk\x=\vphs\wi\mo[\rs\wi\a;\rs\wj\b];\ls\wk\x
\end{equation*}by \refEq{70/6} (with $\sigma$ in place of $\eta$),  \refEq{70/10}, , and  \refEq{70/1}.
Thus, $\cs\Delta\sigma$ is precisely the  set of indices $\g$  such that
\begin{equation*}\tag{9}\labelp{Eq:70/11}
\ls\wk\g\seq\vphs\wi\mo(\rs\wi\a;\rs\wj\b);\ls\wk\x\qquad
\text{if and only if}\qquad\g\in\Delta_\sigma\per
\end{equation*}

Use assumption \refEq{70/7.1}  for the first and only time,
   Partition  \refL{GeneralizedPartitionLemma} (with $M$, $\rho$, $\ai;\aj$,  and $\ls\wk\x;\ak$  in place of
   $\cs Hb$,   $\eta$, $b$,  and  $a$ respectively), and \refEq{70/10}, to arrive at
\begin{equation*}\MS\rho;\ai;\aj=\tsum\MS\rho;\ls\wk\x;\ak=\tsum\MS\sigma;\ak\per\tag{10}
\labelp{Eq:70/7}
\end{equation*}
Conclude that
\begin{multline*}\tag{11}\labelp{Eq:70/11.1}
a_{\wi,\lph};a_{\wj,\bt}=\MS\rho;\ai;\aj=\tsum\MS\sigma;\ak\\=\tsum(\tbigcup_{\g\in\Delta_\sigma}\ls\wk\g);\ak
=\tsum_{\g\in\Delta_\sigma}\ls\wk\g;\ak =\tsum_{\g\in\Delta_\sigma}\as\wk\g\comma
\end{multline*}
by \refEq{70/2}, \refEq{70/7}, \refEq{70/6} (with $\sigma$ in
place of $\eta$), complete distributivity, and the definition of
$\as\wk\g$. In view of \refEq{70/11}, the equality of the first
and last terms in \refEq{70/11.1} is just what was to be shown.
\end{proof}

The  formula in \refL{prewhat} for computing the relative product $a_{\wi,\lph};a_{\wj,\bt}$ takes on a more
familiar form when the system
$a$ is actually a scaffold, and not just a semi-scaffold.

\begin{lm}[Scaffold Relative Product Lemma]\labelp{C:what}
 Let $ (\wx,\wy)$ and $  (\wy,\wz)$ be pairs in
$\mc E$\co  and
$\lph < \kai\wi$ and $\bt <
\kai\wj$\per  If   $a$ is a scaffold\comma then
\[a_{\wi,\lph}; a_{\wj,\bt}=\tsum\{a_{\wk,\gm}:\gm<\kai \wk\text{ and }
\ls\wk\gm\seq\vphs\wi\mo[\rs\wi\a;\ls\wj\b]\myspace\}
\per
\]
\end{lm}
\begin{proof}  The assumption that $a$ is a scaffold implies
 that $\ak\le\ai;\aj$\per  Consequently, the index $\x<\kai\xz$ in
the statement of \refL{prewhat} may  be taken to be $0$.  The term
$\vphs\wi\mo[\rs\wi\a;\ls\wj\b]$ in the formula for computing the
relative product $\as\xy\a;\as\yz\b$ is a coset of the product
group
\[\M=\h\wi;\h\wk\comma
\] which is the identity element of the quotient group $\cs Gz/M$.
Forming the relative product of this term on the right by the
subset $\hs\wk 0=\h\wk$ of $M$ does  not  change it:
\[\vphs\wi\mo[\rs\wi\a;\ls\wj\b];\hs\wk 0=\vphs\wi\mo[\rs\wi\a;\ls\wj\b]\per
\]
Thus, the formula for computing the relative product
$a_{\wi,\lph}; a_{\wj,\bt}$ that is given in  \refL{prewhat}
reduces in this case to  the formula that is given in the
statement of the present lemma.
\end{proof}

\section{Representation theorems}\labelp{S:sec7}

This section contains the main result of the paper, a
representation theorem for atomic, measurable relation algebras.
Continue with the assumption that $\f A$ is a complete and atomic,
measurable relation algebra. Here is a summary of what has been
accomplished so far in the analysis of the structure of $\f A$.

In \refD{gxdef}, a system of mutually disjoint groups
\[G=\langle \cs Gx:x\in I\rangle\] indexed by the set $I$ of measurable atoms in $\f A$ is defined,
where $\cs Gx$ is the group of permutations in $\f A$ of the measurable atom $x$ (see Lemmas \ref{L:nonzerofunct} and \ref{L:disjointgrps}).
 After \refD{funct-dens}, an equivalence relation $\mc E$ is defined on the set $I$ by
 putting a pair $\pair xy$ in $\mc E$ just in case the rectangle $\xoy$ is not zero. It is shown in Semi-scaffold Existence \refL{semil} that $\f A$ has a semi-scaffold \[a=\langle\cs a\xy:\pair xy \in \mc E\rangle\per\]
Fix such a semi-scaffold $a$
 for the remainder of the discussion.

 In terms of $a$, for each pair
  $\pair xy$ in $\mc E$, normal subgroups $\cs H\xy$ and $\cs K\xy$ of
   $\cs Gx$ and $\cs Gy$  respectively are
    defined after Semi-scaffold Existence \refL{semil} as the left and
    right stabilizers of the atom $\cs a\xy$.
    With the help of Isomorphism \refT{isom}, it is shown that there is a canonical quotient isomorphism $\vphi\xy$ from $\cs Gx/\cs H\xy$ to $\cs Gy/\cs K\xy$. In fact, there are coset systems
 \begin{align*}   \langle\ls\xy\x:\x<\kai\xy\rangle\qquad&\text{and}\qquad \langle\rs\xy\x:\x<\kai\xy\rangle\tag{1}\labelp{Eq:ss8.0}\\
\intertext{such that \[\ls\xy\x;\ai=\ai;\rs \xy\x\] for each $\x$,
and the isomorphism $\vphi\xy$ maps $\ls\xy\x$ to $\rs\xy\x$ for
each $\xi$. The system of isomorphisms
 \[\vph=\langle\vphi\xy:\pair xy\in\mc E\rangle\] has the following properties.
  First, for each $x$ in $I$, the atom $\cs a\xx$ coincides with $x$ by semi-scaffold condition (ii), so the isomorphism
 $\vphi\xx$ is the identity automorphism of $\cs Gx/\{x\}$, by Identity \refT{identthm}.
 Second, for each pair $\pair xy$ in $\mc E$, the atom $\cs a\yx$  coincides with the converse $\cs a\xy\ssm$, by semi-scaffold condition (iii), so the isomorphism
$\vphi \yx$ is the inverse of the isomorphism  $\vphi\xy$, by Converse \refT{convthm}. In fact,}
  \cs H\yx=\cs K\xy\qquad&\text{and}\qquad \cs K\yx=\cs H\xy\comma\\
  \intertext{and we may choose cosets systems}
 \langle H_{\yx,\xi} :\xi <\kai\wi\,\rangle\qquad&\text{and}\qquad \langle
K_{\yx,\xi} :\xi <\kai\wi\,\rangle\\
\intertext{of $\cs H\yx$ and $\cs K\yx$ in $\cs Gy$ and $\cs Gx$ respectively such that}
\ls\yx\xi=\rs\xy\xi\qquad&\text{and}\qquad \rs\yx\xi=\ls\xy\xi\comma\tag{2}\labelp{Eq:ss8.00}
\end{align*} so that $\vphi\yx$
 maps $\rs\xy\xi$ to $\ls\xy \x$ for each $\x$.

 Third, for every triple $\trip xyz$ in the set \[\ez 3=\{\trip xyz:\pair xy\text{ and }\pair yz\in\mc E\}\comma\]   the  atoms $\cs a\xy$, $\cs a\yz$, and $\cs a\xz$  satisfy the inequalities \[\cs a\xy\le\xoy\comma\qquad \cs a\yz\le \yoz\comma\qquad\text{and}\qquad \cs a\xz \le \xoz\comma\]  by semi-scaffold condition (i), and therefore
 \[\vphi \xy[\h
\xy;\h \xz]=\k \xy;\h \yz\comma\] by Image \refT{image}.
The preceding observations combine to show that conditions
(i)--(iii) in \refD{semiframedef} of a coset semi-frame are
satisfied.

 Turn now to the task of defining a coset system
 \begin{equation*}\tag{3}\labelp{Eq:ss8.1}
   C=\langle\cc xyz:\trip xyz\in\ez 3\rangle\end{equation*}  such that condition (iv) in \refD{semiframedef} is satisfied.
 For each triple $\trip x y z$  in $\ez 3$\, the relative product
$\axy;\ayz$ is a regular element below $\xoz$ with normal sta\-bi\-li\-zers,
 by Relative Product \refT{else}.  In particular, this relative product is not zero, by
 \refL{regnonzero}.  The atom $\cs a\xz$ is also below $\xoz$, by semi-scaffold condition (i),
  so there must be a coset $\ls\wk\zeta$ of $\cs H\xz$ (with $\zeta<\kai\xz$) such that
\begin{equation*}\tag{4}\labelp{Eq:ss7.1}
\as\wk\zeta =\ls\wk\zeta;\cs a\xz\le\cs a\xy;\cs a\yz\comma
\end{equation*}
 by \refL{each.atom} (with $z$ in place of $y$). Notice in passing that   \refL{each.atom} is  an easy consequence of Atomic Partition \refL{sumatom}.
Choose any such index $\zeta<\kai\xz$, and write
\begin{equation*}\tag{5}\labelp{Eq:ss8}
\cc x y z=\h\xy;\hs\xz\zeta\per
\end{equation*}
 Observe that $\cc \wx\wy\wz$ is a coset of the product
group $\h\xy;\h\xz$ in $\G\wx$\per  The coset $\ls\wk\zeta$
determines an inner automorphism $\tau $ of $\cs Gz/\cs H\xz$ that
is defined by
\[\tau(\ls\xz\eta)=\ls\xz\zeta\ssm;\ls\xz\eta;\ls \xz\zeta\] for each $\eta<\kai\xz$.
 In turn, $\tau$ induces an inner automorphism $\hat\tau$ of $\cs Gz/(\cs H\xy;\cs H\xz)$
 that is defined as follows. Write a given coset $\ls \xy\a;\ls\xz\b$ of $\cs H\xy;\cs H\xz$ as a union
\[\ls \xy\a;\ls\xz\b=\tbigcup  \{\ls\xz\eta:\eta\in \Gamma\}\]   of cosets of $\cs H\xz$, where
 $\Gamma$ is some subset  of the index set $\kai\xz$, and define $\hat\tau$ by
\begin{equation*}\tag{6}\labelp{Eq:ss8.3}
  \hat\tau(\ls \xy\a;\ls\xz\b)=\tbigcup\{\tau(\ls\xz\eta):\eta\in \Gamma\}
  =\tbigcup\{\ls\xz\zeta\ssm;\ls\xz\eta;\ls \xz\zeta:\eta\in \Gamma\}\per
\end{equation*}
The definition of $\cc xyz$ in \refEq{ss8}, the second involution law, complete distributivity, and the fact that $\cs H\xz$ is the identity coset of the quotient group $\cs Gx/\cs H\xz$ together imply that
\begin{multline*}\tag{7}\labelp{Eq:ss8.4}
\cc xyz\ssm;(\ls \xy\a;\ls\xz\b);\cc xyz\\=(\cs H\xz;\ls
\xz\zeta)\ssm;(\tbigcup\{\ls\xz\eta:\eta\in \Gamma\});(\cs H\xz;\ls
\xz\zeta)\\
 =\tbigcup\{\ls \xz\zeta\ssm;\cs H\xz\ssm; \ls\xz\eta;\cs H\xz;\ls \xz\zeta:\eta\in \Gamma\}\\
=\tbigcup\{\ls \xz\zeta\ssm; \ls\xz\eta;\ls \xz\zeta:\eta\in
\Gamma\}\per
\end{multline*}
Compare \refEq{ss8.3} with \refEq{ss8.4} to see that the inner
automorphism $\hat\tau$ of $\cs Gx/(\cs H\xy;\cs H\xz)$ induced by
$\tau$, that is to say, induced by the coset $\ls\xz\zeta$,
coincides with the inner automorphism determined by $\cc xyz$\per

 The index $\zeta$
was chosen so  that   \refEq{ss7.1} is satisfied.  Apply  Relative
Product \refT{secrp} (with $\zeta$ in place of $\xi$) to conclude
that
\[\hvphs\xy\rp\hvphs\yz=\hat\tau\rp\hvphs\xz\per
\] The remarks of the previous paragraph show that $\hat\tau$ coincides with the inner automorphism
  of $\G\wx/(\cs H\xy;\cs H\xz)$
determined by the coset $\cc x y z$\per  Consequently, semi-frame
condition (iv) is satisfied.

Take $C$ to be the system of cosets defined in \refEq{ss8.1} and \refEq{ss8}, and form the group triple
\[\mc F=(G\smbcomma \varphi\smbcomma C)\per\]
  The following theorem about $\mc F$ has been proved.
\begin{theorem}[Semi-frame Theorem]\labelp{T:sfthm}The group triple $\mc F$
is a coset semi-frame\per
\end{theorem}

We shall refer to $\mc F$ as the semi-frame \textit{associated with}
 the semi-scaffold $a$.

 It may appear as if the semi-frame
depends not only on the particular semi-scaffold $a$ that has been
selected, but also on the particular coset that is chosen to
satisfy \refEq{ss7.1}. However, it is not difficult to see that
this is in fact not the case.  If $\hs\xz{\zeta'}$ is any other
coset such that \refEq{ss7.1} (with $\zeta'$ in place of $\zeta$)
holds, then
\[\h\xy;\hs\xz{\zeta'}=\h\xy;\hs\xz{\zeta}=\cc xyz\per
\]
In more detail, the product subgroup $\cs H\xy;\cs H\xz$ is the
left stabilizer of the relative product $\axy;\ayz$\comma  by
Relative Product \refT{else} and Image \refT{image}. In
particular, it   leaves this relative product fixed under relative
multiplication on the left.   Use this fact (at the end of the
computation), together with the identity property of $\cs H\xz$ in
the  quotient group $\cs Gx/\cs H\xz$, \refEq{ss7.1}, and
Partition \refL{GeneralizedPartitionLemma} (with
\[\hs\xz\zeta;\axz\comma\qquad\axy;\ayz\comma\qquad \text{and}\qquad \cs H\xy;\cs H\xz\] in place of $a$,  $b$, and  $H_{b,\et}$ respectively), to obtain
\begin{multline*}
\tsum\h\xy;\hs\xz\zeta;\axz =\tsum\cs H\xy;\cs
H\xz;\hs\xz\zeta;\axz\\ =\cs H\xy;\cs H\xz;\axy;\ayz
=\axy;\ayz\per
\end{multline*} A similar argument shows that
\[\tsum\h\xy;\hs\xz{\zeta'};\axz=\axy;\ayz\comma
\]so
\[\tsum\h\xy;\hs\xz\zeta;\axz=\tsum\h\xy;\hs\xz{\zeta'};\axz\per
\]
Use  \refC{equalmult} to conclude that the two cosets
$\h\xy;\hs\xz\zeta$ and $\h\xy;\hs\xz{\zeta'}$ are equal.

In terms of the semi-frame $\mc F$, a complete and atomic Boolean
algebra with additional completely distributive operations, that is to say, a
complete and atomic Boolean algebra with complete operators
 \[\cra CF=\langle \craset C F\smbcomma\cup\smbcomma\diff
\smbcomma\otimes\smbcomma\mo\smbcomma\id U\rangle
\]  of the same similarity type as relation algebras, can be defined.  The atoms  of this algebra are the binary relations $\r
\xy
\a$ that are the subsets of the Cartesian product $\cs Gx\times\cs Gy$  defined by
\[\r\xy\a=\tbigcup\{\ls\xy\xi\times(\rs\xy\xi;\rs
\xy\a):\xi<\kai\xy\}\] for each pair $\pair xy$ in $\mc E$ and
each $\a<\kai\xy$. In particular, for $x=y$, the left and right
stabilizers are the trivial subgroups \[\cs
H\xy=\{x\}\qquad\text{and}\qquad\cs K\xy=\{y\}\comma\] and the
cosets are the singletons of elements of the respective groups, by
semi-frame condition (i), so that the definition of $\r\xx \a$\
assumes the form
\begin{equation*}\r{\xx}{\a}= \{\pair g {g; g_\a} : g\in \G
x\}\comma
\end{equation*} where $\rs\xx\a=\{\cs g\a\}$.  When $\a=0$, the coset
$\rs\xx\a$ coincides with the trivial subgroup $\cs K0=\{x\}$, by
convention, so that  $\cs g0=x$, and therefore $\r\xx 0$ is the
identity relation on the set $\cs Gx$.

 The elements of the algebra are arbitrary unions of sets of atoms, so the universe of
 the algebra is a set of binary relations on the base set \[U=\tbigcup\{\cs Gx:x\in I\}\per\]  The Boolean operations of the algebra are the binary set-theoretic operation $\,\cup\,$ of forming unions of  binary relations, and the unary set-theoretic operation $\,\sim\,
$ of forming complements of binary relations with respect to   the
unit, or universal, relation $U\times U$, which is the union of
the set of atoms.  The distinguished constant $\id U$ is the
identity relation on $U$, and \[\id U=\tbigcup\{\r\xx0:x\in
I\}\per\]

The  operation ${}\mo $ is the unary set-theoretic operation on
binary relations of forming the converse, or inverse, of a
relation. Semi-frame condition (ii) in \refD{semiframedef} implies
that this operation is
 determined on atoms by
\[\r\xy\a\mo=\r\yx\b\comma\qquad\text{where}\qquad \ls\xy\a\ssm=\ls\xy\b\comma\] that is
to say, where $\ls\xy\b$ is the coset inverse of $\ls\xy\a$ in the
quotient group $\cs Gx/\cs H\xy$, and the operation is extended to
all elements in the algebra by making it  completely distributive
over unions.

The binary operation $\,\otimes\,$ is defined on atoms as follows.
   For    pairs $\pair \wx\wy$ and $\pair \ww \wz$ in $\mc
E$ with $\wy\neq \ww$,
\[\r\xy\lph\otimes\r\wwz\bt=\varnothing
\]
for all $\lph<\kai \xy$ and $\bt<\kai\wwz$\comma and  for pairs
$\pair \wx\wy$ and $\pair \wy \wz$ in $\mc E$\comma
\[\r \xy \a \otimes \r \yz\b=\tbigcup\{\r \xz \g:
\hs \xz \g \seq \vphi \xy\mo[ \ks\xy \a;\hs \yz \b];\cc x y z\}\]
for all $\lph<\kai \xy$ and all $\bt<\kai\yz$. The operation is
extended to arbitrary pairs of elements in $\cra CF$ by making it
be completely distributive over arbitrary unions.

The algebra $\cra CF$ turns out to be a relation algebra if and
only if certain conditions
 called the \emph{coset conditions} are
satisfied, and in this case $\cra CF$ is called a \emph{coset
relation algebra}.  The coset conditions do not play a role in the
discussion below, so we do not go into them further (see
\cite{andgiv1}).

 In the proof of the representation theorem for atomic, measurable relation algebras,
  we shall use a form of the Atomic Isomorphism Theorem (see \cite{giv18}).
  The hypothesis of this form of the theorem is that two complete and atomic
  Boolean algebras with completely distributive operators are given, say $\f A$ and $\f B$, and say
  of the same similarity type as relation algebras, together with a bijection $\vth$
  from the set of atoms in $\f A$ to the set of atoms in $\f B$. Using relation algebraic
   notation for the operations, the conclusion of the theorem may be formulated as follows.
   The bijection $\vth$ can be extended to an isomorphism from $\f A$ to $\f B$ if and only if $\vth$ preserves the Peircean operations on atoms in the sense that
 \begin{alignat*}{3}
 c&\le a;b&\qquad&\text{if and only if}&\qquad \vth(c)&\le \vth(a);\vth(b),\\
  c&\le a\ssm&\qquad&\text{if and only if}&\qquad \vth(c)&\le \vth(a)\ssm,\\
   c&\le \ident&\qquad&\text{if and only if}&\qquad \vth(c)&\le \ident\comma
 \end{alignat*} for all atoms $a$, $b$, and $c$ in $\f A$, where the operations on the left
  (including the operation $\ident$ of rank $0$) are those of $\f A$, and the ones on the right are those of $\f B$. If these conditions are satisfied, then the isomorphism from $\f A$ to $\f B$ is the function $\psi$ defined by
 \[\psi(r)=\tsum\{\vth(a):a\in X\}\] for every element $r$ in $\f A$, where $X$ is the set of atoms in $\f A$ that are below $r$.

The next theorem  says that  coset relation algebras are
essentially the only possible examples of atomic, measurable
relation algebras.

\begin{theorem}[Representation Theorem]\labelp{T:main}Every
atomic\comma  measurable relation algebra  is
essentially isomorphic to a  coset
\ra\per
\end{theorem}

\begin{proof}  Start with an atomic, measurable relation algebra  $\f B$, and pass to its
 completion $\f A$, that is to say, pass to its minimal complete extension.
  The completion $\f A$ is well known to be  a complete and atomic relation algebra, and its
  atoms are the same as the atoms in $\f B$ (see \cite{mo}).  It follows that each subidentity atom $x$ is measurable not only in $\f B$, but also in $\f A$, because the same atoms are below the rectangle $x;1;x$ in both $\f B$ and $\f A$, and consequently $\f A$ is a complete and atomic, measurable relation algebra.

Let $I$ be the set of measurable atoms in $\f A$, and $\mc E$ the equivalence
 relation defined on $I$ by putting $\pair xy$ in $\mc E$ if and only if $x;1;y\neq 0$.  The algebra $\f A$ has a semi-scaffold, by Semi-scaffold Existence \refL{semil}.  Fix such a semi-scaffold
\[a=\langle\cs a\xy:\pair xy\in \mc E\rangle\comma\]   and let \[\mc F=\trip G \vph C\] be the coset semi-frame associated with $a$, where
\[G=\langle \cs Gx:x\in I\rangle\comma\quad \vph=\langle\vphi\xy:\pair xy\in\mc E\rangle\comma\quad
 C=\langle\cc xyz:\trip xyz\in\ez 3\rangle \] are as defined at the beginning of the section.
 (Here, Semi-frame \refT{sfthm} is being used.) The goal is to show that $\f A$ is
 isomorphic to $\cra CF$. It then follows   that $\cra CF$ is a relation algebra, and
 therefore  it automatically satisfies the coset conditions.  Conclusion: $\f A$
 is isomorphic to a coset relation algebra, so $\f B$ is essentially isomorphic to a coset relation algebra.

The distinct atoms in $\f A$ are the elements $\as\wi\a$ defined before \refL{each.atom},  by Semi-scaffold Partition \refL{total.atom}.
  The distinct atoms in $\cra CF$ are the binary relations $\r\xy\a$ defined after Semi-frame \refT{sfthm}.
 Let $\vth$ be the bijection from the set of atoms in $\f A$ to
the set of atoms in $\cra C F$ that is defined by
\begin{equation*}\tag{1}\labelp{Eq:rep81}
\vth(\as\xy\a)=\r\xy\a
\end{equation*} for every pair $\pair xy$ in $\mc E$ and every $\a<\kai\xy$.  It must be shown that $\vth$ preserves the Peircean
operations on atoms in the sense of the Atomic Isomorphism Theorem.

Fix    three arbitrary atoms  in $\f A$ and the corresponding images, under $\vth$, of these three atoms in $\cra CF$, say
\[\as\xy\a,\quad \as\wwz\b,\quad \as\uv\g\qquad\text{and}\qquad \r\xy\a,\quad \r\wwy \b,\quad \r\uv\g\] respectively.
Treat first the case of the operation of relative multiplication.  In view of \refEq{rep81}, it is to be shown that
\begin{align*}
               \as\uv\g\le\as\xy\a;\as\wwz\b\qquad&\text{if and only if}\qquad
               \r\uv\g\seq\r\xy\a\otimes\r\wwz\b\per\tag{2}\labelp{Eq:rep81.7}\\
               \intertext{\refL{eqzero}
and Semi-scaffold Relative Product \refL{prewhat} imply that}
\as\uv\g\le\as\xy\a;\as\wwz\b\qquad&\text{if and only if}\qquad
y=w\comma u=x\comma v=z\comma\tag{3}\labelp{Eq:rep81.2}
             \end{align*}
             and the coset $\ls\xz\g$ determined by the index $\g$
satisfies the inclusion
\begin{gather*}
\ls\wk\gm\seq\vphs\wi\mo[\rs\wi\a;\ls\wj\b];\hs\wk\zeta\comma\tag{4}\labelp{Eq:rep81.3}\\
\intertext{where $\zeta$ is the index chosen for the triple $\trip
xyz$ so that} \ls\xz\zeta;\cs a\xz\le\cs a\xy;\cs a\yz\per
 \end{gather*}
  The definition of the operation $\,\otimes\,$ implies that
  \begin{gather*}\r\uv\g\seq\r\xy\a\otimes\r\wwz\b\qquad\text{if and only if}\qquad y=w\comma u=x\comma v=z\comma\tag{5}\labelp{Eq:rep81.4}\\
\intertext{and the coset $\ls\xz\g$ determined by the index $\g$
satisfies the inclusion}
\ls\wk\gm\seq\vphs\wi\mo[\rs\wi\a;\ls\wj\b];\cc
xyz\comma\tag{6}\labelp{Eq:rep81.5}\\ \intertext{where $\cc xyz$
is the coset of the product group $\cs H\xy;\cs H\xz$ that is
defined by} \cc x yz=\cs
H\xy;\ls\xz\zeta\per\tag{7}\labelp{Eq:rep81.6}
 \end{gather*}
  Observe that the inclusion in \refEq{rep81.5} is equivalent to the one in \refEq{rep81.3}, because the right sides of these two inclusions are equal.  In more detail,\begin{multline*}
             \vphs\wi\mo[\rs\wi\a;\ls\wj\b];\cc xyz=\vphs\wi\mo[\rs\wi\a;\ls\wj\b];\cs H\xy;\ls\xz\zeta\\
             =\vphs\wi\mo[\rs\wi\a;\ls\wj\b];\cs H\xy;\cs H\xz;\ls\xz\zeta=\vphs\wi\mo[\rs\wi\a;\ls\wj\b];\ls\xz\zeta\comma
             \end{multline*} by \refEq{rep81.6},   the identity property of the coset  $\cs H\xz$ in the quotient group $\cs Gx/\cs H\xz$, and   the facts that $\vphs\wi\mo[\rs\wi\a;\ls\wj\b]$ is a  coset of the product group $\cs H\xy;\cs H\xz$, and this product group is the identity element
             in the quotient group $\cs Gx/(\cs H\xy;\cs H\xz)$.
              It follows that the condition in \refEq{rep81.3} may be replaced by the one in \refEq{rep81.5}, so that the inequality on the left side of \refEq{rep81.2} and the inclusion on the left side of \refEq{rep81.4} are both equivalent to the same condition, and therefore they are equivalent to each other. This establishes \refEq{rep81.7}.

 Turn next to the operation of converse, with the goal of showing that
 \begin{equation*}\tag{9}\labelp{Eq:rep81.06}
 \as\uv\g\le\as\xy\a\ssm\qquad\text{if and only if}\qquad  \r\uv\g\seq\r\xy\a\mo\per
\end{equation*}
 Semi-scaffold Converse \refL{aia.eq.ajb} implies that
 \begin{align*}\as \uv\g\le\as\xy\a\ssm\qquad&\text{if and only if}\qquad  u=y,\hspace{4pt}
  v=x,\hspace{4pt} \text{and}\hspace{4pt}\ls\xy\a\ssm=\ls\xy\g\comma\tag{10}\labelp{Eq:rep81.10}\\
 \intertext{and if the conditions on the right side of this equivalence are satisfied,
  then equality actually holds on the left side. Semi-frame condition (ii) implies that}
  \r \uv\g\seq\r\xy\a\mo\qquad&\text{if and only if}\qquad  u=y,\hspace{4pt}
  v=x,\hspace{4pt}
  \text{and}\hspace{4pt}\ls\xy\a\ssm=\ls\xy\g\comma\tag{11}\labelp{Eq:rep81.11}
  \end{align*} and if the conditions on the right side of this equivalence are satisfied,
  then equality actually holds. The conditions on the right sides of
  \refEq{rep81.10} and \refEq{rep81.11} are the same, so the inequalities on the left
  sides must be equivalent.  This establishes \refEq{rep81.06}.

  Turn finally to the identity element.  It is to be shown that
  that
\begin{align*} \as\uv\g\le\ident\qquad&\text{if and only if}\qquad
  \r\uv\g\seq\id U\per\tag{12}\label{Eq:rep81.12}\\
  \intertext{Semi-scaffold Identity \refL{diag.atom} implies
that}
  \as\uv\g\le\ident\qquad&\text{if and only if}\qquad u=v\text{ and
  }\g=0.\tag{13}\labelp{Eq:rep81.13}\\
  \intertext{The definition of the relation $\r\uv\g$ and semi-frame condition (i) imply that}
   \r\uv\g\seq\id U\qquad&\text{if and only if}\qquad u=v\text{ and
  }\g=0\per\tag{14}\labelp{Eq:rep81.14}
\end{align*}
As before, the conditions on the right sides of \refEq{rep81.13}
and \refEq{rep81.14} are the same, so the inequalities on the left
sides must be equivalent.  This proves \refEq{rep81.12}.

It has been shown that the bijection $\vth$ satisfies the
conditions of the Atomic Isomorphism Theorem.  Apply that theorem
to conclude that $\vth$ can be extended to an isomorphism from $\f
A$ to $\cra CF$.
\end{proof}

As was mentioned after \refD{semiframedef}, not every atomic, measurable relation algebra   has a scaffold, but
if there is a scaffold, then a stronger result than Representation \refT{main} is true.  The existence of a scaffold
means that for every triple $\trip xyz$ in $\ez 3$ with $x<y<z$, it is always possible to choose the atoms $\cs a\xy$, $\cs a\yz$ and   $\cs a\xz$ so that
the inequality \begin{align*}
              \cs a\xz&\le\cs a\xy;\cs a\yz \\
              \intertext{holds.  Consequently, the coset $\ls\xy\zeta$ of $\cs H\xz$ that is chosen to tanslate $\cs a\xz$ to a position below $\cs a\xy;\cs a\xz$ in the sense that}
                \ls\xz\zeta;\cs a\xz&\le\cs a\xy;\cs a\yz
               \end{align*}
 may always be taken to be the identity coset $\cs H\xz$, or put another way, one may always   choose $\zeta=0$. The  inner automorphism $\tau$ of the the quotient group
$\cs Gx/\cs H\xz$
that is determined by this  coset is then the identity automorphism of this group, and therefore the inner automorphism $\hat\tau$ of the quotient group $\cs Gx/(\cs H\xz;\cs H\xz)$ that is induced by $\tau$ is
identity automorphism of its group.  As a result, semi-frame condition (iv) assumes the form
\begin{equation*}\tag{1}\label{Eq:frame.01}
  \hvphs\xy\rp\hvphs\yz=\hvphs\xz\per
\end{equation*}  A semi-frame satisfying this condition instead of semi-frame condition (iv) is called a \emph{frame}.

Under these conditions, the shifting coset
$\cc xyz$ that is defined in terms of the coset $\ls\xz\zeta$ becomes the identity coset
\[\cc xyz=\cs H\xy;\cs H\xz \] of the quotient group $\cs Gx/(\cs H\xy;\cs H\xz)$, and the definition of the operator $\,\otimes\,$   between the atomic relations  $ \r\xy\a$ and $\r\yz\b$
assumes the form
\begin{equation*}\tag{2}\labelp{Eq:frame.02}
  \r\xy\a\otimes\r\yz\b=\tbigcup\{\r\xz\g:\ls\xz\g\seq \vphs\wi\mo[\rs\wi\a;\ls\wj\b]\}\per
\end{equation*}
  It was shown in Composition Theorem 3.7 of \cite{giv1} that under the hypothesis of \refEq{frame.01} and semi-frame condition (iii), the operation of relational composition
  between the atomic relations  $ \r\xy\a$ and $\r\yz\b$ satisfies the equation
\begin{equation*}\tag{3}\labelp{Eq:frame.03}
  \r\xy\a\rp\r\yz\b=\tbigcup\{\r\xz\g:\ls\xz\g\seq \vphs\wi\mo[\rs\wi\a;\ls\wj\b]\}\per
\end{equation*}  The right sides of  \refEq{frame.02} and \refEq{frame.03} are the same, so the left sides must be equal. Conclusion:  when an atomic, measurable relation algebra has a scaffold, the equation
\[\r\xy\a\otimes\r\yz\b=\r\xy\a\rp\r\yz\b\] holds, so that the defined operation $\,\otimes\,$ coincides with the set-theoretic operation of relational composition.

In this case, the coset relation algebra $\cra C F$ is  a set relation algebra, and actually a subalgebra of the full set
relation algebra    with   base set and
unit \[U=\tbigcup  \{\G x:x\in I\}\qquad\text{and}\qquad
E=\tbigcup\{\cs Gx\times\cs Gy:\pair xy\in\mc E\}\]
respectively\per  This algebra is called the \emph{group relation algebra on the frame} $\mc F$ in  \cite{giv1}, and is denoted  by
$\cra GF$.

The system of cosets \[C=\langle \cc xyz:\trip xyz\in\ez 3\rangle\] is entirely unnecessary in this case, as are the inner automorphisms $\tau$
determined by these cosets, so instead of considering group triples  $\mc F=\trip G \vph C
$ satisfying semi-frame conditions (i)--(iv) in \refD{semiframedef}, it suffices to consider group pairs $\mc F=\pair G \vph
$ satisfying the four frame conditions, namely conditions (i)--(iii) from \refD{semiframedef} and condition \refEq{frame.01} above.
This is the approach that is taken in \cite{giv1}.

 The group pair constructed from a scaffold $a=\langle\cs a\xy:\pair xy\in\mc E\rangle$ in a measurable relation algebra is called the \emph{group pair associated with} $a$.
\begin{theorem}\textnormal{(Frame Theorem)}
\labelp{T:cosetframe} The  group pair associated with a scaffold
in a measurable relation algebra is always   a frame\per
\end{theorem}
Notice that the hypothesis  on the measurable relation algebra of being atomic is unnecessary when there is a scaffold (or even a semi-scaffold).  The existence of a scaffold always implies that the measurable relation algebra
under consideration is atomic, by
the argument of Semi-scaffold Partition \refL{total.atom}.

\comment{The preceding theorem can be considerably improved if we make some additional
assumptions about the groups $\G\wx$\per

\begin{theorem}[Abelian Frame Theorem]\labelp{T:abelianframe}
 Suppose that $\mc F$ is the group pair associated with an  atomic semi-scaffold in a measurable
\ra\po If\comma  for all pairs $\pair \wx\wy$ and $\pair\wy\wz$ in $\mc E$  with $\wx$\co
$\wy$\co and $\wz$ distinct\,\co   the
quotient group
\[\G\wx/(\h\xy;\h\xz)\]
is Abelian\co  then $\mc F$ is a frame\per
\end{theorem}
\begin{proof}
The proof  proceeds exactly as the proof of \refT{cosetframe}, except for that case in
the verification of the  final frame condition when the  three
atoms $\wx,\wy,\wz$ are distinct. In this case, the product $\axy;\ayz $ is non-zero
(it is a regular element), so there is an atom $c$ below it.  The atom $c$  must be a
translation of the atom $\axz$\comma  by the Atomic Partiton Lemma, and $c$ is regular
with the same normal stabilizers as $\axz$\comma  by the Translation Lemmas
(or Corollaries \ref{C:atoms} and \ref{C:a-b.atoms}).  Since $\LS\cl$ is
included  $\vphs
\xy\mo[\RS\xy;\LS\yz]$  by \refC{abccor}, it follows that
$\LS \xz$ is included in
 $\vphs \xy\mo[\RS\xy;\LS\yz]$\per  Furthermore, the quotient
$\G\wy/(\RS\xy;\LS\yz)$\comma  and hence also its  image under $\vphs \xy\mo$, is
Abelian by assumption.  Therefore,
$\hvphs\xz=\hvphs\xy\,\vert\,\hvphs\yz$ by \refC{abelrelprod}.
\end{proof}

The point of the previous theorem is that, in an atomic, measurable \ra\ it  is trivial
to construct an atomic semi-scaffold, whereas it may be impossible to construct an atomic
scaffold. Nevertheless, if certain relevant quotient groups are Abelian, then by
starting with an atomic semi-scaffold, we still end up with a frame.}

The preceding observations and remarks, including Frame \refT{cosetframe}, show that the existence of a scaffold in a measurable relation algebra implies that the semi-frame $\mc F$ associated
with the given scaffold is actually a frame, and consequently the coset relation algebra $\cra CF$ constructed from $\mc F$ is actually the group relation algebra $\cra GF$, which is of course
a set relation algebra.    Representation \refT{main}
therefore assumes the
the following stronger form.

\begin{theorem}[Scaffold Representation Theorem]\labelp{T:loc} Every measurable relation algebra with a scaffold is essentially isomorphic to a group relation algebra\per
\end{theorem}

The preceding theorem implies that a measurable relation algebra with a scaffold is representable as a set relation algebra in a  stronger sense than is usually intended.  To explain this stronger sense, consider
an arbitrary relation algebra $\f B$  and its completion $\f A$.  An isomorphism $\vth$ from $\f A$ to a complete set relation algebra $\f C$---that is to say, to a set relation
algebra $\f C$ in which the union of every set of relations in  the algebra is again a relation in the algebra---must preserve all existing suprema as unions in the sense that,
for every subset $X$ of $\f A$ with $a=\tsum X$, we have
\[\vth(a)=\vth(\tsum X)=\tbigcup\{\vth(b):b\in X\}\per\]  The reason is that isomorphisms preserve  suprema, and the supremum of each subset $Y$ of $\f C$ is, by assumption, the union of the
relations in $ Y$.  The restriction of $\vth$ to $\f B$ is  an embedding of $\f B$ into
$\f C$ and therefore a representation of $\f B$ as a set relation algebra, but it also inherits from $\vth$ the stronger property of preserving all existing suprema  as unions. Indeed,   if $X$ is any subset of $\f B$ such that the supremum $a=\tsum X$ exists in $\f B$, then $a$
remains the supremum of $X$ in $\f A$, because $\f A$ is the completion of $\f B$, and therefore $\vth$ preserves this supremum as a union.     Representations that preserve all existing suprema as unions
are called \emph{complete representations}, and a relation algebra with a complete representation is said to be   \emph{completely representable}.

\begin{cor}\labelp{C:comprep}
Every measurable relation algebra with a scaffold is completely representable\per
\end{cor}

\medskip
It turns out that a kind of converse to \refC{comprep} is true.

\begin{theorem}\labelp{T:repscaf}Every  measurable relation algebra
that is completely representable has a scaffold\po
\end{theorem}
\begin{proof}Let $\f A$ be a measurable relation algebra, and assume that $\vth$ is a complete representation of $\f A$. Thus,  $\vth$ is an embedding of $\f A$
 into the set relation
algebra of all subrelations of some equivalence relation $E$ on a base
set $U$,  so that
\begin{equation*}\tag{1}\labelp{Eq:repscaf2}
\vth(1)=E\comma
\end{equation*}
and $\vth$ preserves all existing suprema in $\f A$ as
unions A completely representable
relation algebra is  always atomic (see \cite{hh97}), so  $\f A$
must be atomic. Take $X$ to be the set of atoms in $\f A$.  Each
element $d$ in $\f A$ is the sum of the atoms below it,
\begin{equation*}
d=\tsum\{a\in X:a\le d\}\comma
\end{equation*} and therefore
\begin{equation*}\tag{2}\labelp{Eq:repscaf3}
\vth(d)=\tbigcup\{\vth(a):a\in X\text{ and }a\leq d\}\comma
\end{equation*}  by  the assumption that the representation $\vth$ is
complete.

The set  $\mc E$   of   pairs $\pair\wx\wy$ of measurable atoms
such that $x;1;y\neq 0$, or equivalently, such that
\begin{equation*}
 1;x;1=1;y;1\comma
\end{equation*}   is  an equivalence relation on the set of measurable
atoms.  Let $\langle \ww_\xi:\xi\le\kp\rangle$ be a system of
representatives for the equivalence classes, so that each
equivalence  class of $\mc E$ contains exactly one $\ww_\xi$\per

The relations $\vth(\ww_\xi)$ are non-empty, because the elements $\ww_\xi$ are atoms, and they are included
in the equivalence relation $E$, which is  the   unit of the representing algebra. It follows that
each   relation of the form
 \begin{equation*}\tag{4}\labelp{Eq:repscaf5}
E\rp\vth(\ww_\xi)\rp E
\end{equation*}
is a non-empty union of  components of $E$, that is to say, it is a non-empty union of relations of the form $V\times V$, where each
$V$ is an   equivalence  class
of
$E$.  In more detail, if $\pair pq$ is a pair in $\vth(\cs w\x)$, then $p$ and $q$ are in the same equivalence class of $E$, call it $V$. If $r$ and $s$ are any other elements in $V$,
then the pairs $\pair rp$ and $\pair qs$ are both in $E$, and therefore the pair $\pair rs$ must be in \refEq{repscaf5}, by the definition of relational composition.  Thus, every pair in $V\times V$ belongs to \refEq{repscaf5}.   For each representative $\cs w\x$, choose an equivalence class $\cs V\x$ of $E$ such that
\begin{equation*}\tag{5}\labelp{Eq:repscaf6}
V_\xi\times V_\xi\seq E\rp\vth(\cs w\x)\rp
E\per
\end{equation*}

Consider now an arbitrary measurable atom $x$.  Let  $\ww_\xi$ be the representative of   $\wx$, so that
\begin{equation*}\tag{6}\labelp{Eq:repscaf.5}
  1;x;1=1;\cs w\x;1\comma
\end{equation*}
by the definition of $\mc E$\per Observe that
\begin{multline*}
 E\rp\vth(\wx)\rp
E=\vth(1)\rp\vth(\wx)\rp\vth(1)=\vth(1;\wx;1)=\vth(1;\ww_\xi;1)\\= \vth(1)\rp\vth(\ww_\xi)\rp \vth(1)=E\rp\vth(\ww_\xi)\rp E\comma
\end{multline*}
by  \refEq{repscaf2}, the representation properties of $\vth$, and \refEq{repscaf.5}.  It follows from this computation, the choice of $\cs w\x$, and \refEq{repscaf6} that
\begin{equation*}\tag{7}\labelp{Eq:repscaf.7}
V_\xi\times V_\xi\seq E\rp\vth(x)\rp
E\per
\end{equation*} Since $x$ is assumed to be a measurable atom in $\f A$, we have $0\neq x\le\ident$, and therefore
\[\varnot\neq \vth(x)\seq \vth(\ident)=\id U\comma\] by the representation properties of $\vth$.  Use this observation and  \refEq{repscaf.7}  to
choose an element $\cs px$ in $\cs V\x$ wih the property that the
pair $\pair{\cs px}{\cs px}$ belongs to the relation $\vth(x)$.

Consider next an arbitrary  pair
$\pair
\wx\wy$   in
$\mc E$\per The elements $\wx$ and $\wy$ are, by definition, in
the same equivalence class of $\mc E$,  and therefore they have the same representative, say $\ww_\xi$\per  The elements
$\alp x$ and
$\alp y$  are both chosen to be in $V_\xi$\comma so the pair  $\pairalp \wx\wy$ belongs to the component $V_\xi\times
V_\xi$\comma  and therefore also to the unit relation $E$.  The pairs  $\pairalp \wx\wx$  and
$\pairalp
\wy\wy$ are chosen to be in $\vth(x)$ and   $\vth(y)$ respectively, so   the pair  $\pairalp \wx\wy$ belongs to the relation
\[\vth(x)\rp E\rp\vth(y)\comma
\]
by the definition of relational composition. Since
\[\vth(x;1;y)=\vth(x)\rp\vth(1)\rp\vth(y)=\vth(x)\rp E\rp\vth(y)\comma \] by the representation properties of $\vth$ and\refEq{repscaf2}, it follows that
\begin{equation*}\tag{8}\labelp{Eq:repscaf7}
\pairalp\wx\wy\in\vth(x;1;y)\per
\end{equation*}

We now  come to the heart of the argument.    The element $\xoy$ is non-zero, by the definition of $\mc E$, so it must be
the sum of a non-empty set of atoms.  In view of \refEq{repscaf7}  and \refEq{repscaf3} (with
$\xoy$ in place of $d$), there must be an atom  below $\xoy$ whose image under $\vth $
contains  the pair
$\pairalp
\wx\wy$.  Moreover, this atom is  unique, because   \refEq{repscaf3} is a disjoint union of images of mutually distinct atoms.  Call the atom
$\axy$\per  In other words, $\axy$ is the unique atom in $\f A$ with the property that
\begin{equation*}\tag{9}\labelp{Eq:repscaf8}
\pairalp\wx\wy\in\vth(\axy)\per
\end{equation*}
We shall show that the system
\begin{equation*}\tag{10}\labelp{Eq:repscaf11}
\langle \axy:\pair \wx\wy\in\mc E\rangle
\end{equation*}
 is a scaffold.

First of all, $x$ is an atom below $x;1;x$, and the pair
$\pairalp
\wx\wx$ belongs to $\vth(x)$, by the choice of the element $\cs px$.  On the hand,  $\axx$ is defined to be the unique atom below $x;1;x$ with the
  property that its image under $\vth$
contains the pair $\pairalp
\wx\wx$.  Consequently,
\begin{equation*}\tag{11}\labelp{Eq:repscaf12}
\axx=\wx\per
\end{equation*}

 Second, $\cs a\xy$ is an atom below $\xoy$, so its converse $\cs a\xy\ssm$ is   an atom below $\yox$, by Lemmas \ref{L:laws}(vi) and \ref{L:square}(iii), and monotony.   Moreover, the pair $\pair {\cs py}{\cs px}$ belongs to
 the image $\vth(\cs a\xy\ssm)$, because the pair $\pair {\cs px}{\cs py}$ belongs to  $\vth(\cs a\xy)$, by definition, and
 \[\vth(\cs a\xy\ssm)=\vth(\axy)\mo\comma
\] by the representation properties of $\vth$.  On the other hand, the element $\cs a\yx$ is defined to be the unique atom below $\yox$
whose image under $\vth$ contains the pair $\pair {\cs py}{\cs
px}$, so it follows that
\begin{equation*}\tag{12}\labelp{Eq:repscaf13}
\ayx=\cs a\xy\ssm\per
\end{equation*}

 Finally, the pairs $\pairalp
\wx\wy$ and $\pairalp
\wy\wz$ are  in $\vth(\axy)$ and   $\vth(\ayz)$ respectively, by \refEq{repscaf8}, so the pair  $\pairalp
\wx\wz$ is in the relational composition  $\vth(\axy)\rp \vth(\axy)$, by the definition of relational composition.  Since
\[\vth(\axy)\rp\vth(\ayz)=\vth(\axy;\ayz)\comma
\] it follows that the pair  $\pairalp
\wx\wz$  belongs to the relation $\vth(\cs a\xy;\cs a\yz)$. The same pair also belongs to the relation $\vth(\axz)$, by \refEq{repscaf8} so the two relations
have a non-empty intersection.  Use the representation properties of $\vth$ to see that the elements    $\axy;\ayz$ and $\axz$ cannot be
disjoint.  Since $\cs a\xz$ is an atom, it follows that
\begin{equation*}\tag{13}\labelp{Eq:repscaf14}
\axz\leq\axy;\ayz\per
\end{equation*}
 Equations \refEq{repscaf12}--\refEq{repscaf14} are just
the conditions required for \refEq{repscaf11} to be a scaffold.
\end{proof}

When the unit element $E$ of the representation coincides with the universal relation
$U\times U$, the steps leading up to  \refEq{repscaf7} in the preceding proof are
unnecessary. In  this case, the equivalence relation $\mc E$ consists of all pairs of measurable atoms.  For each
measurable atom $\wx$, choose an element $\alp\wx$ in  $U$ such that the pair $\pairalp
\wx\wx$ is in $\vth(x)$.  Certainly, all of the pairs $\pairalp
\wx\wy$ are in $U\times U$, that is, in $\vth(1)$.  Because the pairs $\pairalp
\wx\wx$ and $\pairalp
\wy\wy$ are in $\vth(\wx)$ and $\vth(\wy)$ respectively, the pair $\pairalp
\wx\wy$ must be in
\[\vth(\wx)\rp\vth(1)\rp\vth(\wy)\per
\]
Thus, we obtain \refEq{repscaf7} directly.  The remainder of the proof is just  as before.

The next corollary is a direct consequence of the preceding
theorem and   Scaffold Isomorphism \refT{loc}.

\begin{cor}\labelp{C:measrep}Every measurable relation algebra that is completely representable is
essentially isomorphic to a group relation algebra\po
\end{cor}

 A finite relation algebra is necessarily complete, and any representation of it is
necessarily a complete representation.  Thus, the preceding corollary can be given
a strong formulation when the algebra in question is finite.

\begin{cor}\labelp{C:finmeasrep}Every  measurable relation algebra that is finite and
representable is  isomorphic to a group relation algebra\po
\end{cor}

\comment{

\noindent
\textbf{Variety generated by the coset relation algebras}

\smallskip One rather surprising consequence of the
Representation Theorem is that the class of algebras embeddable
into coset relation algebras forms an equationally axiomable
class, or a \textit{variety}, as such classes are usually
called.  The proof of this theorem is analogous to the proof of
Tarski's theorem that the class of representable \ra s forms a
variety.

\begin{theorem}\labelp{T:var} The class of algebras  embeddable
into coset \ra s is a variety\per
\end{theorem}
\begin{proof} Let $\sk K$ be the  class of all atomic,
measurable \ra s, and  denote the class of algebras that are
embeddable in some algebra of $\sk K$ by
$\sk S(\sk K)$.
\num 1 {$\sk K $ is first-order axiomatizable.}\enum In other
words, there is a set $\Gamma$ of first order sentences such
that an algebra $\f A$ is in $\sk K$ just in case it is a model
of $\Gamma$, that is, all the sentences of
$\Gamma$ are true of $\f A$. First of all, put the relation
algebraic axioms into
$\Gamma$.  Next, recall that the property of being an atom is
expressible in first-order logic: an atom is a non-zero
element that is disjoint from every element that it is not
below.  Consequently, the property of being an atomic algebra
is expressible by a first-order sentence
$\vph$  that says: below every non-zero element there is an
atom.  Put
$\vph$ into $\Gamma$.  The property of being a measurable atom
is also first-order expressible: an element $x$ is a measurable
atom just in case $\wx$ is a diagonal atom (an atom below
$\ident$), and  every non-zero element below $\xox$ is above some
non-zero  functional element.  Therefore, there is a
first-order sentence $\psi$ that expresses the property of an
algebra being measurable: below every non-zero diagonal element
there is a measurable atom.  Put $\psi$ into $\Gamma$.
Clearly, $\Gamma$ is a set of axioms for
$\sk K$.

A well-known theorem of Tarski \cite{t55} says that, for any
first-order axiomatizable class $\sk L$, the class $\sk S(\sk
L)$ of  algebras embeddable into algebras of $\sk L$ is
axiomatizable by a set of (first-order) universal sentences.
In particular,
\num 2 {$\sk S(\sk K) $ is axiomatizable by a set of universal
sentences.}\enum

 Every coset
\ra\ is in $\sk K$, by
\refT{cosetmeas}.  Consequently, every algebra embeddable into
a coset \ra\ is in $\sk S(\sk K) $.  For the reverse inclusion,
use  the Representation Theorem:  it says that every algebra in
$\sk K$ is embeddable into a coset
\ra.  Hence, every algebra in $\sk S(\sk K) $ is  embeddable
into a coset relation algebra.  This proves that
\num 3 {$\sk S(\sk K) $ is just the class of algebras
embeddable  into coset \ra s.}\enum

The direct product  of coset
\ra s is isomorphic to a coset \ra, by
\refC{directprod}.  Consequently, the direct product of
algebras embeddable into coset
\ra s  is itself embeddable into a coset \ra.  In other words,
$\sk S(\sk K)$ is closed under direct products.  It  is
obviously closed under subalgebras.  Hence,

\num 4 {$\sk S(\sk K) $ is closed under subdirect
products.}\enum

Let $\Theta$ be a set of universal sentences that axiomatizes
$\sk S(\sk K) $.  Such a set exists by (2).  It is well known
that for every universal sentence $\theta$ in the  language of
\ra s, there is an (effectively constructible)  equation
$\varepsilon_\theta$ in the language of relation algebras such
that
$\theta$ and $\varepsilon_\theta$ are equivalent in all simple
\ra s, that  is,
$\theta$ is valid in a simple \ra\ $\f A$ just in case
$\varepsilon_\theta$ is valid in
$\f A$ (see ???).  Let $\Delta$ be the set of equations
corresponding to  universal sentences in
$\Theta$,
\[\{\varepsilon_\theta:\theta\in\Theta\}\comma
\] together with the axioms of the theory of relation algebras.
\num 5 {$\Delta$ is a set of axioms for $\sk S(\sk K) $.}\enum

To prove (5), consider any model $\f A$  of $\Delta$.  Then
$\f A$ is a relation algebra.  Every relation algebra is
isomorphic to a subdirect product of simple relation algebras
(see \refT{}).  Let $\f B$ be a simple, subdirect factor of $\f
A$.  Then $\f B$ is a homomorphic image of $\f A$, so every
equation true of $\f A$ is true of $\f B$.  (Recall that
equations are preserved under the passage to homomorphic
images.)  It follows that each equation in
$\Delta$ is valid in $\f B$.  But $\f B$ is simple, by
assumption. Therefore, each sentence in
$\Theta$ is valid in $\f B$.  Consequently, $\f B$ is in $\sk
S(\sk K) $.  This shows that every simple, subdirect factor of
$\f A$ is in $\sk S(\sk K) $.  Since $\sk S(\sk K) $ is closed
under subalgebras and direct products, it follows that $\f A$
is in $\sk S(\sk K)$. In other words, every model of $\Delta$
is in $\sk S(\sk K)$.

For the  reverse inclusion, consider  any coset relation
algebra $\cra C F$, where
$\mc F$ is a  semi-frame that satisfies the  coset
conditions.  Certainly, $\cra C F$ is  in
$\sk S(\sk K)$, by (3), and hence it is a model of $\Theta$.
Suppose  that the  semi-frame $\mc F$ is  simple.  Then
$\cra C F$ is simple (in the algebraic sense of the word), by
\refT{simple1}.  Therefore each   equation  corresponding to a
sentence in $\Theta$ is true of $\cra C F$. Thus, $\cra C F$ is
a model of $\Delta$. Next, consider the case when $\mc F$ is
not simple.  By \refT{cosdecomp}, the algebra $\cra C F$ is
isomorphic to a direct product of coset relation algebras $\cra
C {F\famxi}$, where each $\mc F\famxi$ is a  \textit{simple}
semi-frame satisfying the coset conditions.  (In fact, the
semi-frames
$\mc F\famxi$ are just  the components of $\mc F$.)  We have
seen that each algebra
$\cra C {F\famxi}$ must be a model of $\Delta$.  Since
equations are preserved under the passage to direct products,
it follows that $\cra C F$ is a model of $\Delta$.  In other
words, every coset \ra\ is a model of $\Delta$.  Finally,
equations are also preserved under that passage to
subalgebras.  Thus, any algebra embeddable into a coset \ra\
will be a model of $\Delta$.  In view of (3), this establishes
(5), and hence the theorem.
\end{proof}

}
\section{Finitely measurable relation algebras}\labelp{S:sec8}

 Atomic Partition \refL{sumatom}   raises several  question.  For example, under what conditions
will there be at least one atom below a given rectangle  $\xoy$
when the sides $x$ and $y$ are  measurable atoms?  Speaking more
broadly, under what conditions will a measurable relation algebra
automatically be atomic?  It turns out that finite measurability
implies atomicity.  In other words, if the group $\cs Gx$ is
finite for every measurable atom $x$ in a measurable relation
algebra, then the algebra is automatically atomic.    Keep in mind
that finitely measurable relation algebras may be infinite in
size, because there may be infinitely many measurable atoms in the
algebra.

 The key observation is contained in the following lemma, which could have been proved immediately after   Product \refL{a-b.leftregular}. Recall that the set $X_a$ is a union of cosets
of $\h a$\per If it is the union of only finitely many such
cosets, then we shall say that $\h a$ has \textit{finite index} in
$X_a$\per This terminology parallels that of group theory, where
one speaks of the index of a subgroup in a group.

\begin{lm} \labelp{L:finiteindex}
Let $\wx$ and $\wy$ be measurable atoms\comma and $0<a\le\xoy$\po
If $\h\al$ has finite index in $X_\al$\comma then there is
left-regular element below $a$\po
\end{lm}
\begin{proof} 
Consider the collection of subsets $Z$ of $
X_\al$ with the following properties.  First,   $x$ is in $Z$. Second,
\begin{equation*} 
\prodreg_{f\in Z} f;a\not= 0\per 
\end{equation*}
Because $f;a=g;a$ if and  only if $f$ and
$g$ are in the same coset of $\h a$\comma by \refC{stab}, it may also be assumed   that
 $Z$ is a union of cosets of $\h a$\per

There are only finitely many cosets of $\h a$ included in $X_a$,
by assumption, and   each $Z$ under consideration is assumed to be
a union of some of these cosets,  so there are only   finitely
many possible choices for $Z$.
 Moreover, there certainly exist
sets $Z$ with the required properties. For instance, the left stabilizer
$\h a$ is such a set. In more detail, the identity element $x$ is in $\cs Ha$, because $\cs Ha$ is a subgroup of $\cs Gx$. Also,
\[\tprod_{f\in\cs Ha}f;a =a\not=0\comma\] because $\cs Ha$ is the left stabilizer of $a$. Finally, $\cs Ha$   is the union of the singleton coset $\{\cs Ha\}$.

  As there are only finitely many choices for the set $Z$, it is possible to choose one  of maximal cardinality. Let $Z$ be such a
  maximal choice, and  write
\begin{equation*} 
b=\prodreg_{f\in Z} f;a\comma \tag*{(1)} \labelp{Eq:P27/3}
\end{equation*}
with the goal of showing that $b$ is a left-regular element below
$a$. First of all,  $b$ is not $0$,  by the second property. Also,
$\wx$ is in $Z$, by the first property, so $\wx;a$ is one of the
elements in the product \ref{Eq:P27/3}. \refL{domain}(iii) implies
that $x;a=a$, so $a$ must be one of the elements in in the product
\ref{Eq:P27/3}, and therefore $b$ must be below $a$, by Boolean
algebra. It remains to show that $b$ is left-regular. This amounts
to proving  that
\begin{equation*} 
X_\bl= \L_\bl\comma \tag*{(2)} \labelp{Eq:P27/6}
\end{equation*} by \refC{lrreg}.

The inclusion from right to left in \ref{Eq:P27/6} is a  consequence of \refL{leftcos}(ii).  To establish the reverse inclusion,
consider an element $g$ in $ X_\bl$\comma with the goal of showing that $g$ is in $\cs Hb$\per Observe that \ref{Eq:P27/3} is essentially a finite product. Use \ref{Eq:P27/3} and the distributive law for   functions to obtain
\begin{equation*} 
g;b= g; (\prodreg_{f\in Z} f;a) =\prodreg_{f\in Z} g;f;a\per \tag*{(3)}
\labelp{Eq:P27/7}
\end{equation*}
  The product
$(g;b)\cdot b$ is non-zero, by \refL{leftcos}(i) (with $g$ and $b$
in place of $f$ and $a$ respectively), because $g$ is assumed to
belong to $X_\bl$. Consequently,
\begin{equation*}\tag{4}\label{Eq:P27.4}
(\prodreg_{f\in Z} g;f;a) \cdot  (\prodreg_{f\in Z} f;a)
=(g;b)\cdot b\not= 0\comma
\end{equation*}
by \ref{Eq:P27/7} and \ref{Eq:P27/3}.  Use \ref{Eq:P27/7},
\refEq{P27.4}, and the assumed maximality of the cardinality of
$Z$ to conclude that every element in the set
\begin{equation*} 
g;Z=\{g;f:f\in Z\}\tag*{(5)} \labelp{Eq:P27/8}
\end{equation*} must also belong to $Z$, that is to say,
\begin{equation*} 
g;Z\seq Z\per\tag*{(6)} \labelp{Eq:P27/9}
\end{equation*}

 For each coset $\cs H\xi$ of $\cs Ha$ that is included in $Z$, the coset $g;\h\xi$ is included in $g;Z$, by \ref{Eq:P27/8}, and
  these cosets are distinct for distinct $\xi$, by the cancellation law for groups. Consequently,  the sets $g;Z$
and
$Z$ are unions of the same finite number of cosets of $\h a$\per  Use this observation and
   \ref{Eq:P27/9} to arrive at
  $g;Z=Z$, from which it follows that
\[g;b=\prodreg_{f\in Z}g;f;a= \prodreg_{f\in Z} f;a=b\comma
 \]by \ref{Eq:P27/7} and \ref{Eq:P27/3}. This proves that $g$ is in the left stabilizer $\L_\bl$\comma so \ref{Eq:P27/6} holds and therefore $b$ is left-regular.\end{proof}

\begin{lm} \labelp{L:finite}
For any  measurable atoms $\wx$ and $\wy$\comma and any non-zero element $a\le\xoy$\comma if the set
$X_\al$ is finite\co then there is an atom below
$a$\po
\end{lm}
\begin{proof} Observe first that if $b\le a$, then $\cs Xb\seq\cs Xa$. Indeed, if $b\le a$, then the monotony laws imply
that
$b\ssm\le a\ssm$ and therefore that $b;b\ssm\le a;a\ssm$. Consequently,
\begin{equation*}
\ssum X_\bl =b;b\ssm\le a;a\ssm = \ssum X_\al\comma
\end{equation*}
by definition of the sets $X_\al$ and $X_\bl$\per
 As these   are sets of atoms, it follows that
$X_\bl $ is included in $ X_\al$\per

Let $W$ be the set  of left-regular elements
below $a$.
For every element $b$   in $W$,
\begin{equation*}
\L_\bl=X_\bl\seq X_\al\comma\end{equation*}
by \refC{lrreg} and the initial observation of the proof, so  the stabilizer of every element in $W$ is   finite. Moreover,
the set $W$ is not empty, by \refL{finiteindex}.  It is therefore possible to choose an element $b$ in $W$ such that its left stabilizer $\cs Hb$
has minimal finite cardinality among the left stabilizers of elements in $W$.  Clearly,
\begin{equation*}
  0<b\le a\comma\tag{1}\labelp{Eq:finite.1}
\end{equation*}
 by \refL{regnonzero} and the definition of $W$.

 The argument that
$b$ must be an atom proceeds by contraposition.  If $b$ is not an atom, then there
is a non-zero element $c$ that is strictly below $b$, by Boolean algebra and the first inequality in \refEq{finite.1}.  Obviously,
$c$ is also  below $a$, by the second inequality in \refEq{finite.1}, so the set
 $X_c $ is included in the set $X_\al$\comma by the initial observation of the proof (with $c$ in place of $b$).  In particular,
 the set $X_c$ must  be  finite, since this is true of $\cs Xa$. Apply \refL{finiteindex} (with $c$ in place of $a$) to obtain a left-regular element
$d$ below $ c$. It follows from the definition of the set $W$ that $d$ belongs to this set.  Both $b$ and  $d$
are left-regular elements, and
\[0<d=b\cdot d\le c< b\comma\] by the choices of the elements $c$ and $d$, together with \refL{regnonzero}.  These inequalities and \refL{char.less} imply that   the left stabilizer $\L_d$ must be  a proper
subgroup of the finite left stabilizer $
\L_\bl$. Consequently, the cardinality of  $\L_\bl$ cannot be   minimal among  the left stabilizers of  elements $b$ in $W$.
\end{proof}

\begin{theorem}\labelp{T:atomic} Every finitely measurable relation algebra is atomic\per
\end{theorem}
\begin{proof}Let $a$ be an
arbitrary non-zero element in the \ra, with the goal of showing that
there is an atom below
$a$.  Consider the set
$ I$  of measurable atoms. Each element $x$ in $I$
 is  finitely
measurable, by assumption, so the corresponding group $\G x$ is finite, by the  definition of finite measurability.
The identity element $\ident$ is the sum of the set $I$, by the assumption of measurability,
so
\begin{equation*}\tag{1}\labelp{Eq:atomic.1} 1 = \ident;1;\ident = (\tsum I);1;(\tsum I) =
\tsum\{\xoy:x,y\in I\}\comma
\end{equation*}
by the complete distributivity of relative multiplication.
It follows from \refEq{atomic.1} and Boolean algebra that
\begin{equation*}\tag{2}\labelp{Eq:atomic.2} a=a\cdot 1 = \tsum\{a\cdot (\xoy):x,y\in I\/\}\per
\end{equation*}
The element $a$ is assumed to be different from zero, so \refEq{atomic.2} implies that there must be measurable atoms $x$ and $y$ in
$I$ such that the element \[d=a\cdot (\xoy)\] is not $0$. The set
$X_d$ is a subset of the finite group $\G x$ and is therefore itself a finite set. Apply \refL{finite} (with $d$ in place of $a$) to
obtain  an atom $b$ below $d$.   Clearly, $b$ is also an atom below $a$.
\end{proof}

\section{A characterization of regular elements}\labelp{S:sec9}

The next theorem and its corollary characterize regular elements
in several illuminating ways. One consequence of the corollary is
that, in the presence of atoms, left-regularity implies
right-regularity and conversely. In other words, left-regular (or
right-regular) elements are regular. Moreover, regular elements
are precisely the elements that can be written in the form $\ssum
M;\al$ for some atom $\al$ below $\xoy$ and some subgroup $M$ of
$\G x$ that includes $\L_\al$\per Equivalently, regular elements
are precisely the elements that can be written in the form $\ssum
{\m \xi};\al$ for some atom $\al$ and some coset $\m \xi$ of a
subgroup $M$ of $\G x$ that includes $\L_\al$\per

\renewcommand\g{\zeta}
\begin{theorem}[Regular Characterization Theorem]\labelp{T:group}
Let $\wx$ and $\wy$ be measurable atoms\comma  and   $a\le \xoy$ a regular element
with normal stabilizers\po For each element $ b\leq \xoy$ with
$\LS\al\seq\LS\bl$\comma the following conditions are equivalent\per
\begin{enumerate}
\item[(i)] $b$ is left-regular\per
\item[(ii)] $b$ is regular\per
\item[(iii)]  $\bl=\tsum \MS\g;\al$ for some coset $\MS\g$
of a  subgroup
$M$  of
$G_x$ such that  $\L_\al\seq M$\per
\item[(iv)] $\bl=\ssum M;\cl$ for some  subgroup $M$  of
$G_x$ such that  $\L_\al\seq M$ and some \opar any\cpar\  left translation $c$ of
$\al$ that is below $\bl$\per
\item[(v)] $\bl = \tsum\LS{b,\rho};\al$ for some coset $\LS{b,\rho}$ of
$\LS\bl$\per
\item[(vi)] $\bl = \ssum\L_\bl;\cl$ for  some \opar any\cpar\  left translation $c$ of
$\al$ that is below $\bl$\po
\end{enumerate}
\end{theorem}\begin{proof} 
Obviously, (ii) implies (i).  To establish the implication from
(i) to (vi), assume that $\bl$ is left-regular. The   assumptions
that $a$  is regular with normal stabilizers, and that
$\LS\al\seq\LS\bl$, mean that the implication from (i) to (iii) in
Translation \refL{assertioneq} may be applied to obtain a left
translation $\cl$ of $\al$ that is below $\bl$\per
 Any such left translation $\cl$  is a left-regular element with the same normal
stabilizer $\LS\al$ as $a$\comma by  Translation Lemmas \ref{L:left.reg} and \ref{L:assertioneq}, and the assumptions on $a$.  Apply
 Partition \refL{GeneralizedPartitionLemma} (with $c$ and $\cs Hb$ in place of $a$  and  $\LS{b,\eta}$ respectively), and use the definition of $\cs Hb$ as
the stabilizer of $b$, to conclude that
\[\ssum \L_\bl;\cl=\L_\bl;\bl=\bl\comma\]
as desired.

To prove that (vi) implies (v), assume that
\begin{equation*}\tag{1}\labelp{Eq:group.1}
  \cl=\LS{\al,\rho};\al \qquad\text{and}\qquad b=\tsum\cs Hb;c
\end{equation*}
 for some coset $\LS{\al,\rho}$  of
$\LS\al$\per The set defined by
\begin{equation*}\tag{2}\labelp{Eq:group.2}
  \LS{b,\rho}=\LS\bl;\LS{\al,\rho}
\end{equation*}
 is
a coset of $\LS\bl$, because   the assumption   $\cs Ha\seq \cs Hb$  implies that
every element in $\LS{\al,\rho}$ gives rise to the same coset of $\cs Hb$.  Use \refEq{group.1} and \refEq{group.2}   to arrive at
\[\bl=\ssum
\L_\bl;\cl=\tsum\L_\bl;\LS{\al,\rho};\al=\tsum\LS{\bl,\rho};\al\per\]
This proves (v).

To see that
(v)   implies (iii), take \[M=\h\bl\qquad
\text{and}\qquad  \MS\g=\L_{\bl,\rho}\comma\] and use the assumption
that $\LS\al\seq\LS\bl$\per

To establish the implications from (iii) to (iv) and from (iv) to (ii), it is helpful to introduce some notation.
Let
$M$ be  a subgroup of $\G\wx$ that includes $\h\al$\per  Fix
left coset systems
\[\langle\LS\x:\x<\ka\rangle\qquad\text{and}\qquad \langle\MS\eta:\eta<\la\rangle\] for $\LS\al$ and $\M$ respectively in
$\G\wx$\per  The assumption   $\cs Ha\seq \M$ implies the existence of a partition $\langle \Gamma_\eta:\eta<\la\rangle$ of
the index set $\ka$ such that
\begin{equation*} 
\MS\eta =\tsbigcreg_{\x \in\Gamma_\eta} \L_\x\tag*{(3)} \labelp{Eq:P28/2}
\end{equation*}
for each $\eta<\la$.

We now take up the implication from (iii) to (iv). If (iii) holds for the  subgroup $M$, then
\begin{equation*}\bl=\tsum\MS\g;a=\tsum(\,\tbigcup_{\x\in\Gamma_\g}\LS\x\,);a=
\tsum_{\x\in\Gamma_\g}\LS\xi;\al\tag{4}\labelp{Eq:P28/0}
\end{equation*}
by \ref{Eq:P28/2} (with $\zeta $ in place of $\eta$) and complete distributivity. The equality of the first and last terms in \refEq{P28/0}, and Partition \refL{PartitionLemma}, together  imply that the elements $\LS\x;\al$
with
$\x$ in $\Gamma_\g$  are precisely the left translations of $\al$ that are below
$\bl$. Take
$\cl$ to be any one of these translations, say $c=\cs H\x;a$.  Certainly $\cl$ is below $\bl$, by \refEq{P28/0} and the definition of $c$.  The assumption in
\ref{Eq:P28/2} (with $\zeta$ in place of $\eta$) implies that every element in $\cs H\x$ gives rise to the same coset of $\M$ as every element in $\cs M\zeta$, and of course
that coset is $\cs M\zeta$.  Consequently,
\begin{equation*}\M;\LS\x=\M;\MS\g=\MS\g\per\tag{5}\labelp{Eq:P28/11}
\end{equation*} Use \refEq{P28/0} and \refEq{P28/11} to conclude that
\begin{equation*}b=\tsum\MS\g;a=\tsum\M;\LS\x;a=\tsum\M;c\comma
\end{equation*}
as desired.

\renewcommand\RS[1]{K_{#1}}

 \renewcommand\al{c}
Turn, finally, to the implication from (iv) to (ii).  The element
$a$ is assumed to be regular with normal stabilizers, and  $\cl$
is assumed to be a left translation of $a$, so $c$ must be a
regular element with the same  normal stabilizers as $a$, by
Translation \refL{left.reg}(i) and its right-regular version. In
particular,
\begin{equation*}\tag{6}\labelp{Eq:P28/6.1}
  \LS\cl=\LS a\per
\end{equation*}
Isomorphism \refT{isom} guarantees the
existence of  an isomorphism $ \vphs\al$ from $\G\wx/\LS\al$ to $\G\wy/\RS\al$
with the property that, writing $K_\x = \vphs\al (\L_\x)$, we have
\begin{equation*} 
\L_\x;\al = \al; \R_\x\comma \tag*{(7)} \labelp{Eq:P28/4}
\end{equation*}
for each $\xi<\ka$. Take $\eta=0$ in \ref{Eq:P28/2}, and use the convention $\cs M0=M$, to obtain \begin{equation*}
M=\tsbigcreg_{\x\in\Gamma_0} \L_\x\per\tag*{(8)}\labelp{Eq:P28/10}
\end{equation*}
  Put
\begin{equation*}
N=\tsbigcreg_{\x\in\Gamma_0} \R_\x\per\tag*{(9)}\labelp{Eq:P28/3}
\end{equation*}
The subgroup   $M$   of $\cs Gx$ is assumed in  (iv) to include
the subgroup  $\cs Ha$,  so it also includes $\cs Hc$, by
\refEq{P28/6.1}. Use group theory, the definition of a quotient
set   and \ref{Eq:P28/10} to see that   the quotient
\begin{equation*}\tag{10}\labelp{Eq:P28/10.1}
  M/\L_\al=\{ f/\L_c:f\in \M\}=\{\L_\x: \xi\in \cs\Gamma 0\}
\end{equation*}
 is a subgroup of $\G\wx/\LS\al$. The image of this subgroup under the quotient isomorphism $\vphs\al$ is  the set
\begin{multline*}\tag{11}\labelp{Eq:P28/11.1}
\vphs\al(M/\L_\al)=\{\vphs\al(\cs H\x):\x\in\cs \Gamma 0\}=\{\cs K\x:\x\in\cs \Gamma 0\}\\
=\{g/\cs Kc:g\in N\}=N/\cs Kc\comma \end{multline*}
by the definition of an image set,  \refEq{P28/10.1}, the definition of $ \vphs\al$, \ref{Eq:P28/3}, and the definition of the set $N/\cs Kc$. Isomorphisms preserve
the property of being a subgroup,  so the equality of the first and last terms in \refEq{P28/11.1} implies that $N/\cs Kc$ is a subgroup of $\G\wy/\RS\al$\per  It follows
from group theory that   $N$ must be a subgroup of $G_\wy$ that includes $\R_\al$.

Use   \ref{Eq:P28/10}, complete distributivity, \ref{Eq:P28/4}, complete distributivity, and  \ref{Eq:P28/3} to obtain
\begin{multline*} 
\ssum M;\al=\tsum(\tbigcup_{\x\in\Gamma_0}\L_\x );\al=\tsum_{\x\in\Gamma_0}\L_\x ;\al\\ =\tsum_{\x\in\Gamma_0}\al;\R_\x=\tsum \al;(\tbigcup_{\x\in\Gamma_0}\R_\x)=\ssum \al;N.
\tag*{(12)}
\labelp{Eq:P28/6}
\end{multline*}
Use the assumption in  (iv), the second involution law, complete
distributivity, the associative law, the regularity of $c$, and
the assumption that $M$ is a subgroup of $\cs Gx$ (and hence
closed under formation of converses and relative products), and
$\cs Hc\seq M$ to get
\begin{multline*} \tag{13}\labelp{Eq:P28/13.1}
\bl;\bl\ssm   = (\ssum M;\al); (\ssum M;\al)\ssm
   = (\ssum M;\al); (\ssum \al\ssm;M\ssm)\\
= \ssum M;\al; \al\ssm;M\ssm
= \ssum M;\L_\al;M\ssm
=\ssum M\per
\end{multline*}
A similar computation using (iv), the second involution law, complete distributivity, the associative law, the group properties of $M$, \ref{Eq:P28/6}, the regularity of $c$, and the fact proved
above that $N$ is a subgroup of $\cs Gy$ that includes $\cs Kc$, yields
\begin{multline*} \tag{14}\labelp{Eq:P28/14.1}
\bl\ssm;\bl  = (\ssum M;\al)\ssm; (\ssum M;\al)
  = (\ssum \al\ssm;M\ssm); (\ssum M;\al)\\
=\ssum \al\ssm;M\ssm;M;\al =\ssum \al\ssm;M;\al\\
  =\ssum \al\ssm;\al;N
=\ssum \R_\al;N
=\ssum N\per
\end{multline*}

In view of \refEq{P28/13.1}, \refEq{P28/14.1}, and the definition of regularity, the proof that $b$
is regular will be complete once it is shown that
\begin{equation*} 
M=\L_\bl
\quad\text{and}\quad
N= \R_\bl\per \tag*{(15)} \labelp{Eq:P28/8}
\end{equation*}
If $f$ is in
$M$, then  $f;M=M$, since $M$ is a group, and therefore
\begin{equation*}
f;\bl   = f;(\ssum M;\al)
  =\ssum f;M;\al
= \ssum M;\al= \bl,
\end{equation*}
by (iv) and complete distributivity. Thus, every element in $M$ belongs to the left stabilizer of $b$, so   $M\seq \L_\bl$\per To establish the reverse inclusion, observe that
\begin{equation*}\tag{16}\labelp{Eq:P28.9}
\ssum\L_\bl\le \ssum X_\bl=\bl;\bl\ssm =\ssum\M\comma
\end{equation*}
by \refL{leftcos}(ii) and monotony, the definition of the set $\cs
Xb$,  and \refEq{P28/13.1}. The sets being summed on the left and
on the right are sets of atoms, so the inequality in \refEq{P28.9}
implies the inclusion   $\L_\bl\seq M$.

 The proof of the second
equation in \ref{Eq:P28/8} is entirely analogous, but uses
\refEq{P28/14.1} instead of \refEq{P28/13.1}.
\end{proof}

\renewcommand\al{a}

\begin{cor}\labelp{C:group}
Let $\wx$ and $\wy$ be measurable atoms\per If there is an atom below
$\xoy$\comma then for each element
$b\le
\xoy$  the following conditions are equivalent\per
\begin{enumerate}
\item[(i)] $b$ is left-regular\per
\item[(ii)] $b$ is regular\per
\item[(iii)]$\bl=\ssum {\m\g};\al$ for some atom $a\le\xoy$ and some left coset
$\m\g$ of a subgroup
$M$  of
$G_x$ such that  $\L_\al\seq M$\per
\item[(iv)]$\bl=\ssum M;\al$ for some atom $a\le\xoy$ and some  subgroup $M$  of
$G_x$ such that  $\L_\al\seq M$\per
\item[(v)] $\bl = \ssum\L_{\bl,\rho};\al$ for some atom
$a\le\xoy$  and some coset $\L_{\bl,\rho}$ of $\LS\bl$\per
\item[(vi)] $\bl = \ssum\L_\bl;\al$ for some atom $a\le\xoy$\po
\end{enumerate}
\end{cor}
\begin{proof} It may be assumed
that $\bl$ is non-zero, since this is implied by each of the conditions (i)--(vi). The
hypothesis that there is an atom below the rectangle $\xoy$ implies that this rectangle is a sum of atoms, by
  Atomic Partition \refL{sumatom}.  In particular, there must be an atom $\al$ below
$\bl$, since   \[0<\bl\le \xoy\per\] Of course, $\LS\al\seq\LS\bl$\comma by \refL{lessthan}. Atoms are
regular elements with normal stabilizers, by Corollaries
\ref{C:atoms} and \ref{C:a-b.atoms}, so the hypotheses of \refT{group} are
satisfied.
  The equivalence of  statements (i)--(vi)
in that theorem, and the fact that any translation  of an atom  is
itself an atom, by \refL{sumatom}, immediately yield the
corollary.\end{proof}

\begin{cor}\labelp{C:stab1} Let $\wx$ and $\wy$ be measurable atoms\co and
$\al\le\xoy$   a regular element with normal stabilizers\po If
$\bl=\tsum\MS\g;\al$ for some coset $\MS\g$ of a subgroup $\M$ of $\G\wx$ such that
$\LS\al\seq\M$\co then $\M$ is the left stabilizer of $\bl$\po
\end{cor}

\begin{proof}  If $b=\tsum \MS\g;\al$\comma then for some (any)
translation
$\cl$ of $\al$ that is below $\bl$, we have
\[b=\tsum\M;\cl\qquad\text{and}\qquad b=\tsum\LS\bl;\cl\comma \]
 by the implications  in \refT{group} from (iii) to (iv) and
from (iii) to (vi) respectively.  Therefore, $M=\L_\bl$\comma by
\refC{equalmult} (with $M$, $\cs Hb$, and $c$ in place of $X$,
$Y$, and $a$ respectively).
\end{proof}

\renewcommand\g{\gamma}

 The  two previous corollaries show that, in general, one cannot expect the
stabilizers of regular elements to be {\it normal\/} subgroups. To
see this, suppose that $\al$ is an atom below $\xoy$.  Take $M$ to
be any subgroup of $G_\wx$ that includes $\L_\al$ but is not normal.
The element $\bl=\ssum M;\al$ is regular, by \refC{group}, and its
left stabilizer  is  the non-normal subgroup $M$, by \refC{stab1}.

\end{document}